\documentclass[a4paper]{amsart}
\usepackage{amscd}
\usepackage{amsmath}
\usepackage{amssymb}
\usepackage{amsthm}
\usepackage{bbm}
\usepackage{stmaryrd}
\usepackage[usenames,dvipsnames]{xcolor}
\usepackage[colorlinks,citecolor=blue,urlcolor=blue]{hyperref}

\usepackage{lipsum}

\usepackage[T1]{fontenc}

\newif\ifpdf
\ifx\pdfoutput\undefined
   \pdffalse        
\else
   \pdfoutput=1     
   \pdftrue
\fi

\ifpdf
   \usepackage[pdftex]{graphicx}
\else
   \usepackage{graphicx}
\fi

\numberwithin{equation}{section} \swapnumbers

\newtheorem{satz}{Satz}[section]

\newtheorem{theorem}[satz]{Theorem}
\newtheorem{proposition}[satz]{Proposition}
\newtheorem{corollary}[satz]{Corollary}
\newtheorem{lemma}[satz]{Lemma}
\newtheorem{assumption}[satz]{Assumption}

\newtheorem{definition}[satz]{Definition}

\newtheorem{remark}[satz]{Remark}

\newtheorem{example}[satz]{Example}

\newcommand{\bbr}{\mathbb{R}}
\newcommand{\bbb}{\mathbb{B}}
\newcommand{\bbe}{\mathbb{E}}
\newcommand{\bbn}{\mathbb{N}}
\newcommand{\bbp}{\mathbb{P}}

\newcommand{\cald}{\mathcal{D}}

\newcommand{\calf}{\mathcal{F}}
\newcommand{\calg}{\mathcal{G}}

\newcommand{\call}{\mathcal{L}}
\newcommand{\calm}{\mathcal{M}}
\newcommand{\caln}{\mathcal{N}}
\newcommand{\calp}{\mathcal{P}}

\newcommand{\loc}{{\rm loc}}
\newcommand{\supp}{{\rm supp}}
\newcommand{\Id}{{\rm Id}}
\newcommand{\ran}{{\rm ran}}
\newcommand{\lin}{{\rm lin}}
\newcommand{\tr}{{\rm Tr}}
\newcommand{\pr}{{\rm prox}}

\newcommand{\bbI}{\mathbbm{1}}

\newcommand{\la}{\langle}
\newcommand{\ra}{\rangle}

\DeclareMathOperator{\Tr}{Tr}

\usepackage[textsize=small]{todonotes}
\usepackage{enumerate}

\begin{document}

\title[Stochastic invariance in infinite dimension]{Stochastic invariance in infinite dimension beyond Lipschitz coefficients}

\author{Eduardo Abi Jaber \and Stefan Tappe}
\address{\'{E}cole Polytechnique, Mathematical Finance group, CMAP department, Paris, France}
\email{eduardo.abi-jaber@polytechnique.edu}
\thanks{Eduardo Abi Jaber is  grateful for the financial support from the Chaires FiME-FDD and  Financial Risks  at Ecole Polytechnique.}
\address{Albert Ludwig University of Freiburg, Department of Mathematical Stochastics, Ernst-Zermelo-Stra\ss{}e 1, D-79104 Freiburg, Germany}
\email{stefan.tappe@math.uni-freiburg.de}
\date{August 24, 2026}
\thanks{Stefan Tappe gratefully acknowledges financial support from the Deutsche Forschungsgemeinschaft (DFG, German Research Foundation) -- project number 444121509.}
\begin{abstract}
We establish necessary and sufficient conditions for stochastic invariance of closed subsets in Hilbert spaces for solutions to infinite-dimensional stochastic differential equations (SDEs) under mild assumptions on the coefficients. Our first characterization is formulated in terms of certain normal vectors to the invariance set and requires differentiability only of the dispersion operator, but not of the diffusion coefficient itself. The condition involves a suitable corrected drift expressed through the dispersion operator and its Moore-Penrose pseudoinverse, extending the classical Stratonovich correction term to the present low-regularity setting. Our second characterization is given in terms of the positive maximum principle for the infinitesimal generator of the associated diffusion process.  We illustrate our characterizations in the case of  invariant manifolds.
\end{abstract}


\maketitle\thispagestyle{empty}

\tableofcontents

\section{Introduction}

Consider an infinite dimensional stochastic differential equation (SDE) of the form
\begin{align}\label{SDE}
dX_t = b(X_t)dt + \sigma(X_t) dW_t, \quad X_0 = x,
\end{align}
driven by a trace class Wiener process $W$. The state space of the SDE \eqref{SDE} is a separable Hilbert space $H$, and the coefficients $b$ and $\sigma$ are continuous mappings of linear growth. We refer to Section \ref{sec-main-results} for more details concerning the precise mathematical framework.

Let $\cald \subset H$ be a closed subset of the state space. We say that the subset $\cald$ is \emph{stochastically invariant} with respect to the diffusion \eqref{SDE} if for each $x \in \cald$ there exists a weak solution $X$ to \eqref{SDE} which stays in $\cald$.

In certain situations, the stochastic invariance of a given subset $\cald \subset H$ has already been studied in the literature. In the finite dimensional case $\dim H < \infty$ we refer, for example, to  \cite{Milian-manifold, Bardi-1, Bardi-2, Aubin-Doss, DP-Frankowska, DP-Frankowska-convex, Eduardo-2, Eduardo-1}, and in the infinite dimensional situation $\dim H = \infty$ we refer, for example, to \cite{Filipovic-inv, Milian, Nakayama, Cannarsa, FTT-manifolds, Tappe-cones-2017, Tappe-cones-2024, Bhaskaran-Tappe}.

In the aforementioned references it is typically assumed that the subset $\cald$ has certain structural properties (for example that $\cald$ is a convex subset or a submanifold) or that the diffusion coefficient $\sigma$ is smooth. The latter assumption on the diffusion coefficient, which is actually assumed in many papers on stochastic invariance, is related with the so-called Stratonovich correction term
\begin{align}\label{eq:introstrato}
 x \mapsto \frac{1}{2} \sum_{j=1}^{\infty} D \sigma^j(x) \sigma^j(x),
\end{align}
which typically appears when formulating conditions on the drift coefficient $b$.

So far, the papers that do not impose smoothness on $\sigma$ in the infinite dimensional situation, make additional assumptions on the subset $\cald$. For example, in \cite{Cannarsa} a condition on the regularity of the distance function $d_{\cald}$ is imposed, in \cite{Tappe-cones-2017, Tappe-cones-2024} the subset $\cald$ is a closed convex cone, and in \cite{Bhaskaran-Tappe} the subset $\cald$ is a finite dimensional submanifold.

The main goal of this paper is to characterize stochastic invariance of a closed subset $\mathcal{D}$ of a Hilbert space exclusively in terms of suitable normal vectors to $\mathcal{D}$, without imposing smoothness assumptions on the diffusion coefficient $\sigma$. More precisely, rather than requiring differentiability of $\sigma$, we assume differentiability only of the dispersion operator  
\[
C := \sigma Q^{1/2}(\sigma Q^{1/2})^*,
\]
where $Q$ denotes the covariance operator of the driving Wiener process $W$.

This approach is inspired by \cite{Eduardo-1}, who addressed the problem in finite dimensions by introducing a modified drift correction involving the Moore-Penrose pseudoinverse of  $C$, thereby extending the classical Stratonovich correction term \eqref{eq:introstrato} to a low-regularity framework. However, the extension to infinite dimensions is highly non-trivial. In particular, even the definition of the corresponding correction term becomes delicate due to  the geometry of its range. We outline these difficulties while presenting our two main characterizations.

As already mentioned, we assume that the coefficients $b$ and $\sigma$ in \eqref{SDE} are continuous mappings of linear growth. Moreover, we assume that the closed subset $\cald \subset H$ satisfies the Heine-Borel property, but apart from that it can be arbitrary. We emphasize that in infinite dimension the continuity and linear growth conditions of the coefficients are generally not sufficient (see, e.g. \cite{GMR} for additional conditions) in order to ensure the existence of weak solutions to the SDE \eqref{SDE}, which is related to the lack of compact subsets of the state space $H$. However, if $\cald$ satisfies the Heine-Borel property, then the invariance conditions presented below ensure the existence of weak $\cald$-valued solutions.

For our first main result (Theorem \ref{thm-main}) we will additionally assume that $C$ is smooth (as already mentioned above) and that the operators $C(x)$, or equivalently $\sigma(x)$, for $x \in \cald$ have closed range. Then, denoting by $\caln_{\cald}^{1,\pr}(x)$ the set of all proximal normals to $\cald$ at $x$ (see \eqref{eq:proximalnormal}) and by $\call$ the infinitesimal generator of the diffusion \eqref{SDE}, the following statements are equivalent:
\begin{enumerate}
\item[(i)] $\cald$ is stochastically invariant with respect to the diffusion \eqref{SDE}.

\item[(ii)] For all $x \in \cald$ and all $u \in \caln_{\cald}^{1,\pr}(x)$ we have
\begin{align}\label{cond-C-zero-intro}
&C(x)u = 0,
\\ \label{cond-drift-N-intro} &\la u, b(x) \ra - \frac{1}{2} \sum_{j=1}^{\infty} \la u, D C^j(x) (C C^+)^j(x) \ra \leq 0.
\end{align}

\item[(iii)] The generator $\call$ satisfies the positive maximum principle.
\end{enumerate}
We remark that smoothness of $C$ is required as the derivative of $C$ appears in \eqref{cond-drift-N-intro}, and the closed range assumption is needed in order to define the Moore-Penrose pseudoinverse $C^+$, which also shows up in \eqref{cond-drift-N-intro}.

For our second main result (Theorem \ref{thm-main-sigma}) we can skip the closed range assumption and only assume that $\sigma$ is smooth. Then the following statements are equivalent:
\begin{enumerate}
\item[(i)] $\cald$ is stochastically invariant with respect to the diffusion \eqref{SDE}.

\item[(ii)] For all $x \in \cald$ and all $u \in \caln_{\cald}^{1,\pr}(x)$, we have
\begin{align}\label{cond-sigma-zero-intro}
&\sigma(x)^* u = 0,
\\ \label{cond-drift-sigma-N-intro}
&\langle u, b(x) \rangle - \frac{1}{2} \sum_{j=1}^{\infty} \langle u, D \sigma^j(x) \sigma^j(x) \rangle \leq 0.
\end{align}
\item[(iii)] The generator $\call$ satisfies the positive maximum principle.
\end{enumerate}
We will also provide several equivalent characterizations of the invariance conditions \eqref{cond-C-zero-intro}--\eqref{cond-drift-sigma-N-intro}.

The remainder of this paper is organized as follows. In Section \ref{sec-main-results} we provide the general mathematical framework and present our main results. In Section \ref{sec-manifolds} we apply our findings to the particular situation where the subset is a finite dimensional submanifold. In Section \ref{S:necessity} we provide the necessity proof, in Section \ref{sec-maximum-principle} we prove that the invariance conditions imply that the positive maximum principle is fulfilled, and in Section \ref{sec-sufficiency} we provide the sufficiency proof. The proofs of further key results is deferred to Sections \ref{S:proofofequalitydriftseries} and \ref{S:proofdoublestochasticinequality}. Moreover, for convenience of the reader, we provide the required background about several topics related to this paper in Appendices \ref{app-geometry}--\ref{app-operators-Hilbert}. This includes geometry in Hilbert spaces, linear operators in Hilbert spaces, and the required results about martingales and smooth functions in Banach spaces.

\section{Presentation of the main results}\label{sec-main-results}

In this section we provide the general mathematical framework and present our main results.

\subsection{The general framework}

In this section we provide the mathematical framework regarding the diffusion \eqref{SDE}. Let $H$ be a separable Hilbert space, and let $Q \in L_1^{++}(H)$ be a nuclear, self-adjoint, positive definite linear operator. There exist an orthonormal basis $\{ e_j \}_{j \in \bbn}$ of $H$ and a sequence $(\lambda_j)_{j \in \bbn} \subset (0,\infty)$ with $\sum_{j \in \bbn} \lambda_j < \infty$ such that
\begin{align}\label{diagonal}
Q e_j = \lambda_j e_j \quad \text{for all $j \in \bbn$.}
\end{align}
The space $H_0 := Q^{1/2}(H)$, equipped with the inner product
\begin{align}\label{new-inner-product}
\la h,g \ra_{H_0} := \la Q^{-1/2} h, Q^{-1/2} g \ra
\end{align}
is another separable Hilbert space, and the system $\{ f_j \}_{j \in \bbn}$ given by $f_j := \sqrt{\lambda_j} e_j$ for each $j \in \bbn$ is an orthonormal basis of $H_0$. Let $L_2^0(H) := L_2(H_0,H)$ be the space of all Hilbert-Schmidt operators from $H_0$ into $H$, which endowed with the Hilbert-Schmidt norm $\| \cdot \|_{L_2^0(H)}$ is another separable Hilbert space.

\begin{lemma}\label{lemma-HS-norm}
The following statements are true:
\begin{enumerate}
\item $Q^{1/2} : (H,\| \cdot \|) \to (H_0,\| \cdot \|_{H_0})$ is an isometric isomorphism.

\item The linear mapping $\Phi_Q : L_2^0(H) \to L_2(H)$ given by $\Phi_Q(T) = T Q^{1/2}$ is an isometric isomorphism.
\end{enumerate}
\end{lemma}

\begin{proof}
The first statement follows from the definition \eqref{new-inner-product} of the inner product on $H_0$. For the proof of the second statement, recall that $\{ f_j \}_{j \in \bbn}$ is an orthonormal basis of $H_0$. This gives us
\begin{align*}
\| T \|_{L_2^0(H)}^2 = \sum_{j=1}^{\infty} \| T f_j \|^2 = \sum_{j=1}^{\infty} \| T Q^{1/2} e_j \|^2 = \| T Q^{1/2} \|_{L_2(H)}^2,
\end{align*}
completing the proof.
\end{proof}

Let $b : H \to H$ and $\sigma : H \to L_2^0(H)$ be measurable mappings. For the definition of a $Q$-Wiener process $W$, which appears in the upcoming definition, see, for example \cite[Def. 4.2]{Da_Prato}.

\begin{definition}\label{def-martingale-solution}
Let $x \in H$ be arbitrary. A triplet $(\bbb,W,X)$ is called a \emph{martingale solution} to the SDE \eqref{SDE} with $X_0 = x$ if the following conditions are fulfilled:
\begin{enumerate}
\item $\bbb = (\Omega,\calf,(\calf_t)_{t \in \bbr_+},\bbp)$ is a stochastic basis; that is, a filtered probability space satisfying the usual conditions.

\item $W$ is an $H$-valued $Q$-Wiener process on the stochastic basis $\bbb$.

\item $X$ is an $H$-valued adapted, continuous process such that we have $\bbp$-almost surely
\begin{align*}
\int_0^{t} \big( \| b(X_s) \|_H + \| \sigma(X_s) \|_{L_2^0(H)}^2 \big) ds < \infty, \quad t \in \bbr_+
\end{align*}
and $\bbp$-almost surely
\begin{align*}
X_{t} = x + \int_0^{t} b(X_s) ds + \int_0^{t} \sigma(X_s) dW_s, \quad t \in \bbr_+.
\end{align*}
\end{enumerate}
\end{definition}

\begin{remark}
If there is no ambiguity, we will simply call $X$ a \emph{martingale solution} to the SDE \eqref{SDE} with $X_0 = x$.
\end{remark}

\subsection{The infinitesimal generator}

In this section we introduce the infinitesimal generator of the diffusion \eqref{SDE} and provide some of its properties. Suppose that $b$ and $\sigma$ are continuous. We define $\Sigma : H \to L_2(H)$ as $\Sigma(x) := \sigma(x) Q^{1/2}$ for each $x \in H$. Note that $\Sigma(x) \Sigma(x)^* \in L_1^+(H)$ for each $x \in H$. In what follows, we denote by $C(H)$ the space of all continuous functions $\phi : H \to \bbr$, and for $k \in \bbn$ we denote by $C^k(H)$ the space of functions $\phi : H \to \bbr$ of class $C^k$. Using the conventions from Remark \ref{remark-derivatives}, in the upcoming definition we regard $D \phi(x)$ as an element from $H$, and $D^2 \phi(x)$ as a self-adjoint operator from $L(H)$.

\begin{definition}
The \emph{infinitesimal generator} $\call : C^2(H) \to C(H)$ is defined as
\begin{align}\label{generator-def}
\call \phi(x) := \la D \phi(x), b(x) \ra + \frac{1}{2} \tr \big( D^2 \phi(x) \Sigma(x) \Sigma(x)^* \big), \quad x \in H.
\end{align}
\end{definition}

\begin{remark}\label{rem-law-of-diffusion}
Let $\bar{\sigma} : H \to L_2^0(H)$ be another continuous mapping such that
\begin{align}\label{Sigma-Sigma-star-coincide}
\Sigma(x) \Sigma(x)^* = \bar{\Sigma}(x) \bar{\Sigma}(x)^* \quad \text{for all $x \in H$,}
\end{align}
where $\bar{\Sigma} : H \to L_2(H)$ is analogously defined as $\bar{\Sigma}(x) := \bar{\sigma}(x) Q^{1/2}$ for all $x \in H$. Then the laws of the diffusions \eqref{SDE} and
\begin{align}\label{SDE-bar}
dX_t = b(X_t)dt + \bar{\sigma}(X_t) dW_t, \quad X_0 = x
\end{align}
coincide. Indeed, since the finite dimensional distributions of the solutions to the diffusions \eqref{SDE} and \eqref{SDE-bar} are determined by the corresponding transition semigroups (see \cite[Prop. 4.1.6]{EK}), they are in turn determined by the respective generators, and these coincide due to \eqref{Sigma-Sigma-star-coincide}.
\end{remark}

For each $j \in \bbn$ we define $\sigma^j : H \to H$ as $\sigma^j(x) := \sigma(x) f_j$ for all $x \in H$. Note that $\sigma^j(x) = \Sigma(x) e_j$ for each $x \in H$, and that
\begin{align}\label{norm-of-sigma}
\| \sigma(x) \|_{L_2^0(H)}^2 = \sum_{j=1}^{\infty} \| \sigma^j(x) \|^2, \quad x \in H.
\end{align}

\begin{lemma}\label{lemma-trace-sigma-part-2}
For all $x \in H$ and every linear operator $v \in L(H)$ we have
\begin{align*}
\tr \big( v \Sigma(x) \Sigma(x)^* \big) = \sum_{j=1}^{\infty} \la v \sigma^j(x), \sigma^j(x) \ra.
\end{align*}
\end{lemma}

\begin{proof}
By Lemma \ref{lemma-trace-commute-LR} we have
\begin{align*}
\tr \big( v \Sigma(x) \Sigma(x)^* \big) &= \tr \big( \Sigma(x)^* v \Sigma(x) \big) = \sum_{j=1}^{\infty} \la \Sigma(x)^* v \Sigma(x) e_j, e_j \ra
\\ &= \sum_{j=1}^{\infty} \la v \Sigma(x) e_j, \Sigma(x) e_j \ra = \sum_{j=1}^{\infty} \la v \sigma^j(x), \sigma^j(x) \ra,
\end{align*}
completing the proof.
\end{proof}

\begin{lemma}\label{lemma-generator-series}
For each $\phi \in C^2(H)$ we have
\begin{align*}
\tr \big( D^2 \phi(x) \Sigma(x) \Sigma(x)^* \big) = \sum_{j=1}^{\infty} \la D^2 \phi(x) \sigma^j(x), \sigma^j(x) \ra, \quad x \in H,
\end{align*}
and hence
\begin{align*}
\call \phi(x) = \la D \phi(x), b(x) \ra + \frac{1}{2} \sum_{j=1}^{\infty} \la D^2 \phi(x) \sigma^j(x), \sigma^j(x) \ra, \quad x \in H,
\end{align*}
or equivalently
\begin{align*}
\call \phi(x) = D \phi(x) b(x) + \frac{1}{2} \sum_{j=1}^{\infty} D^2 \phi(x) ( \sigma^j(x), \sigma^j(x) ), \quad x \in H.
\end{align*}
\end{lemma}

\begin{proof}
Taking into account the conventions from Remark \ref{remark-derivatives}, this is an immediate consequence of Lemma \ref{lemma-trace-sigma-part-2}.
\end{proof}

We can extend the notion of the infinitesimal generator as follows. Let $G$ be another separable Hilbert space.

\begin{definition}
We define $\call : C^2(H,G) \to C(H,G)$ as
\begin{align}\label{generazor-def-G}
\call \phi(x) := D \phi(x) b(x) + \frac{1}{2} \sum_{j=1}^{\infty} D^2 \phi(x) ( \sigma^j(x), \sigma^j(x) ), \quad x \in H.
\end{align}
\end{definition}

Due to Lemma \ref{lemma-generator-series}, the two definitions \eqref{generator-def} and \eqref{generazor-def-G} coincide in case $G = \bbr$.

\begin{lemma}\label{lemma-generator-growth}
For each $\phi \in C^2(H,G)$ we have
\begin{align*}
\| \call \phi(x) \| \leq \| D \phi(x) \| \, \| b(x) \| + \frac{1}{2} \| D^2 \phi(x) \| \, \| \sigma(x) \|_{L_2^0(H)}^2, \quad x \in H.
\end{align*}
\end{lemma}

\begin{proof}
Indeed, by \eqref{norm-of-sigma} we have
\begin{align*}
\| \call \phi(x) \| &\leq \| D \phi(x) \| \, \| b(x) \| + \frac{1}{2} \sum_{j=1}^{\infty} \| D^2 \phi(x) \| \, \| \sigma^j(x) \|^2
\\ &= \| D \phi(x) \| \, \| b(x) \| + \frac{1}{2} \| D^2 \phi(x) \| \, \| \sigma(x) \|_{L_2^0(H)}^2,
\end{align*}
completing the proof.
\end{proof}

\begin{lemma}\label{lemma-generator-functionals}
Let $\phi \in C^1(H,G)$ and $\ell \in L(G,\bbr)$ be arbitrary. Then for each $x \in H$ the following statements are true:
\begin{enumerate}
\item We have
\begin{align*}
D(\ell \circ \phi)(x)v = \ell \big( D \phi(x)v \big), \quad v \in H.
\end{align*}
\item If $\phi$ is even of class $C^2$, then we have
\begin{align*}
\call ( \ell \circ \phi )(x) = \ell \big( \call \phi(x) \big).
\end{align*}
\end{enumerate}
\end{lemma}

\begin{proof}
Note that $\ell \circ \phi : H \to \bbr$ is also of class $C^1$. By the first order chain rule we have
\begin{align*}
D(\ell \circ \phi)(x)v = D \ell(\phi(x)) \circ D \phi(x) v = \ell \big( D \phi(x)v \big), \quad v \in H.
\end{align*}
If $\phi$ is even of class $C^2$, then $\ell \circ \phi$ is also of class $C^2$, and by the second order chain rule we obtain
\begin{align*}
D^2(\ell \circ \phi)(x)(v,w) = D \ell(\phi(x)) \circ D^2 \phi(x)(v,w) = \ell \big( D^2 \phi(x)(v,w) \big), \quad v,w \in H,
\end{align*}
proving the claimed identities.
\end{proof}

We will use the following version of It\^{o}'s formula. For $k \in \bbn$ we denote by $C_b^k(H,G)$ the space of all $\phi : H \to G$ of class $C^k$ such that $\phi,D \phi,D^2 \phi,\ldots,D^k \phi$ are bounded, and we denote by $C_{b,\loc}^2(H,G)$ the space of all $\phi : H \to G$ such that $\phi|_V$ is of class $C_b^k$ for every open and bounded subset $V \subset H$.

\begin{theorem}\label{thm-Ito}
Let $G$ be another separable Hilbert space, and let $\phi \in C_{b,\loc}^2(H,G)$ be arbitrary. Then we have
\begin{align*}
\phi(X_t) = \phi(x) + \int_0^t \call \phi(X_s) ds + \int_0^t D \phi(X_s) \sigma(X_s) dW_s, \quad t \geq 0.
\end{align*}
\end{theorem}

\begin{proof}
The result is a consequence of It\^{o}'s formula for real-valued functions (see, e.g. \cite[Thm. 2.9]{Atma-book}) as well as Lemmas \ref{lemma-generator-series} and \ref{lemma-generator-functionals}.
\end{proof}

\begin{definition}\label{def-pos-max}
We say that the generator $\call$ satisfies the \emph{positive maximum principle} if $\call \phi(x) \leq 0$ for any $x \in \cald$ and any function $\phi : H \to \bbr$ of class $C^2$ such that $\displaystyle\max_{\mathcal{D}} \phi = \phi(x) \geq 0$.
\end{definition}

\subsection{The invariance property}

In this section we specify the invariance property of the diffusion and discuss further assumptions. In what follows, let us fix a closed subset $\cald \subset H$ of the state space.

\begin{definition}
The subset $\cald$ is said to be \emph{stochastically invariant} with respect to the diffusion \eqref{SDE} if, for all $x \in \cald$, there exists a weak solution $X$ to \eqref{SDE}, starting at $X_0 = x$ such that $X_t \in \cald$ for all $t \geq 0$, almost surely.
\end{definition}

As an immediate consequence of Remark \ref{rem-law-of-diffusion}, we have the following auxiliary result.

\begin{lemma}\label{lemma-invariance-law}
The subset $\cald$ is stochastically invariant with respect to the diffusion \eqref{SDE} if and only if it is stochastically invariant with respect to the diffusion \eqref{SDE-bar}, where $\bar{\sigma} : H \to L_2^0(H)$ is another continuous mapping such that
\begin{align}\label{Sigma-Sigma-star-coincide-on-D}
\Sigma(x) \Sigma(x)^* = \bar{\Sigma}(x) \bar{\Sigma}(x)^* \quad \text{for all $x \in \cald$,}
\end{align}
and where $\bar{\Sigma} : H \to L_2(H)$ is defined as $\bar{\Sigma}(x) := \bar{\sigma}(x) Q^{1/2}$ for all $x \in H$.
\end{lemma}

\begin{assumption}\label{ass-lin-growth}
The coefficients $b : H \to H$ and $\sigma : H \to L_2^0(H)$ are continuous and satisfy the linear growth condition
\begin{align}\label{linear-growth}
\| b(x) \| + \| \sigma(x) \|_{L_2^0(H)} \leq L (1 + \| x \|) \quad \text{for all $x \in H$}
\end{align}
with some constant $L > 0$.
\end{assumption}

\begin{remark}
By Lemma \ref{lemma-HS-norm} we can express the linear growth condition \eqref{linear-growth} as
\begin{align}\label{linear-growth-2}
\| b(x) \| + \| \Sigma(x) \|_{L_2(H)} \leq L (1 + \| x \|) \quad \text{for all $x \in H$.}
\end{align}
\end{remark}

Now, we define the \emph{dispersion coefficient} $C : \cald \to L_1^+(H)$ as
\begin{align*}
C(x) := \Sigma(x) \Sigma(x)^*, \quad x \in \cald.
\end{align*}

\begin{remark}\label{rem-kernels}
By Lemma \ref{lemma-kernels} we have
\begin{align*}
\ker(C(x)) = \ker(\Sigma(x)^*) = \ker(\sigma(x)^*) \quad \text{for all $x \in \cald$.}
\end{align*}
\end{remark}

\begin{assumption}\label{ass-extension}
The dispersion coefficient $C : \cald \to L_1^+(H)$ can be extended to a function $C : H \to L_1(H)$ of class $C^2$ such that $C(x)$ is self-adjoint for all $x \in H$, and
\begin{align}\label{C-lin-growth}
\| C(x) \|_{L_1(H)}^{1/2} \leq M ( 1 + \| x \| ), \quad  x \in H
\end{align}
for some constant $M > 0$. Note that for convenience of notation this extension is also denoted by $C$.
\end{assumption}

Note that for $x \notin \cald$ we generally have $C(x) \neq \Sigma(x) \Sigma(x)^*$, and that the linear operator $C(x) \in L_1(H)$ does not need to be nonnegative definite.

\begin{assumption}\label{ass-Sigma-self-adjoint}
$\Sigma(x)$ is self-adjoint and nonnegative definite for each $x \in \cald$.
\end{assumption}

Note that Assumption \ref{ass-Sigma-self-adjoint} implies $C(x) = \Sigma(x)^2$, and thus $\Sigma(x) = C(x)^{1/2}$ for all $x \in \cald$. This assumption does not mean a severe restriction. Indeed, define $\bar{\sigma} : H \to L_2^0(H)$ as
\begin{align}\label{Sigma-from-C-Q}
\bar{\sigma}(x) := \bar{\Sigma}(x) Q^{-1/2}, \quad x \in H,
\end{align}
where $\bar{\Sigma} : H \to L_2^+(H)$ is given by
\begin{align}\label{Sigma-from-C}
\bar{\Sigma}(x) &:= |C(x)|^{1/2}, \quad x \in H.
\end{align}
Then we have $\bar{\Sigma}(x) = \bar{\sigma}(x) Q^{1/2}$ for all $x \in H$, and $C(x) = \bar{\Sigma}(x)^2$ for all $x \in \cald$, showing \eqref{Sigma-Sigma-star-coincide-on-D}. Furthermore, Lemma \ref{lemma-construct-sigma} below assures that $\bar{\sigma}$ is continuous. Thus, by Lemma \ref{lemma-invariance-law} the subset $\cald$ is invariant with respect to \eqref{SDE} if and only if it is invariant with respect to \eqref{SDE-bar}.

Moreover, the linear operator $\bar{\Sigma}(x)$ is self-adjoint and nonnegative definite for each $x \in \cald$. Choosing the original extension $C : H \to L_1(H)$ and noting Lemma \ref{lemma-construct-sigma} below, we see that Assumptions \ref{ass-lin-growth}, \ref{ass-extension}, \ref{ass-Sigma-self-adjoint} also hold true for the diffusion \eqref{SDE-bar}.

\begin{lemma}\label{lemma-construct-sigma}
The mapping $\bar{\sigma} : H \to L_2^0(H)$ defined according to \eqref{Sigma-from-C-Q} is continuous, and there is a constant $N > 0$ such that
\begin{align}\label{Sigma-lin-growth}
\| b(x) \| + \| \bar{\sigma}(x) \|_{L_2^0(H)} \leq N ( 1 + \| x \| ) \quad \text{for all $x \in H$.}
\end{align}
\end{lemma}

\begin{proof}
By Lemma \ref{lemma-HS-norm} the linear mapping $\Phi_Q : L_2^0(H) \to L_2(H)$ given by $\Phi_Q(T) = T Q^{1/2}$ is an isometric isomorphism, and we have $\bar{\sigma} = \Phi_Q^{-1} \circ \bar{\Sigma}$. Moreover, by Corollary \ref{cor-square-abs-continuous} the mapping $\bar{\Sigma}$ is continuous, and we have
\begin{align*}
\| \bar{\Sigma}(x) \|_{L_2(H)} = \| C(x) \|_{L_1(H)}^{1/2} \quad \text{for all } x \in H.
\end{align*}
Together with \eqref{linear-growth} and \eqref{C-lin-growth}, this provides the desired result.
\end{proof}

For the next assumption, we prepare an auxiliary result.

\begin{lemma}\label{lemma-closed-range}
For each $x \in \cald$ the following statements are equivalent:
\begin{enumerate}
\item[(i)] $C(x)$ has closed range.

\item[(ii)] $\ran(C(x))$ is finite dimensional.

\item[(iii)] $\ran(\Sigma(x))$ is finite dimensional.

\item[(iv)] $\ran(\sigma(x))$ is finite dimensional.
\end{enumerate}
\end{lemma}

\begin{proof}
(i) $\Leftrightarrow$ (ii): This is a consequence of Theorem \ref{thm-closed-range-fd}.

\noindent(ii) $\Leftrightarrow$ (iii): Noting that $C(x) = \Sigma(x) \Sigma(x)^*$, this is a consequence of Proposition \ref{prop-closed-range-char} and Theorem \ref{thm-closed-range-fd}.

\noindent(iii) $\Leftrightarrow$ (iv): We have $\ran(\Sigma(x)) = \sigma(x) Q^{1/2}(H) = \sigma(x)(H_0) = \ran(\sigma(x))$, proving the stated equivalence.
\end{proof}

The following assumption ensures that we can define the Moore-Penrose pseudoinverse $C(x)^+$ for each $x \in \cald$. We refer to Appendix \ref{app-operators-Hilbert} for more details about the Moore-Penrose pseudoinverse in Hilbert spaces.

\begin{assumption}\label{ass-closed-range}
We suppose that $C(x)$ has closed range for each $x \in \cald$, or equivalently, that $\ran(\sigma(x))$ is finite dimensional for each $x \in \cald$.
\end{assumption}

\begin{remark}\label{rem-closed-range}
By Proposition \ref{prop-closed-range-char} the closed range assumption implies that
\begin{align*}
\ran (C(x)) = \ran (\Sigma(x)) = \ran (\sigma(x)) \quad \text{for all $x \in \cald$.}´
\end{align*}
\end{remark}

\subsection{The Heine-Borel property}\label{sec-Heine-Borel}

In this section we introduce the Heine-Borel property and provide some related results. A metric space $(M,d)$ has the \emph{Heine-Borel property} if every closed and bounded subset $A \subset M$ is compact, or equivalently, if for each $x \in M$ and every $r > 0$ the closed ball $\{ y \in M : d(x,y) \leq r \}$ is compact.

\begin{lemma}\label{lemma-HB-equivalent}
Let $(M,d)$ be a metric space, and let $\rho$ be another metric on $M$. Suppose there are mappings $\alpha,\beta : M \to (0,\infty)$ such that
\begin{align}\label{metrics-equivalent}
\alpha(x) \cdot d(x,y) \leq \rho(x,y) \leq \beta(x) \cdot d(x,y), \quad  x,y \in M.
\end{align}
Then $(M,d)$ has the Heine-Borel property if and only if $(M,\rho)$ has the Heine-Borel property.
\end{lemma}

\begin{proof}
By \eqref{metrics-equivalent} the two metrics $d$ and $\rho$ generate the same topology on $M$. In particular, for each $x \in M$ and every $r > 0$ we have
\begin{align*}
\{ y \in M : d(x,y) \leq r \} &\subset \{ y \in M : \rho(x,y) \leq \beta(x) r \},
\\ \{ y \in M : \rho(x,y) \leq r \} &\subset \bigg\{ y \in M : d(x,y) \leq \frac{r}{\alpha(x)} \bigg\}.
\end{align*}
Since closed subsets of compact sets are compact, this provides the stated equivalence.
\end{proof}

\begin{assumption}\label{ass-Heine-Borel}
The closed subset $\cald$ has the Heine-Borel property.
\end{assumption}

\begin{remark}
Let us start with some obvious observations:
\begin{enumerate}
\item If the Hilbert space $H$ is finite dimensional, then the closed subset $\cald$ always has the Heine-Borel property.

\item If the subset $\cald$ is compact, then it has the Heine-Borel property.

\item If the subset $\cald$ has the Heine-Borel property, then it is locally compact and $\sigma$-compact.
\end{enumerate}
\end{remark}

For a continuous map $\gamma : [0,1] \to \cald$ we define the \emph{length} $L(\gamma) \in [0,\infty]$ as
\begin{align*}
L(\gamma) := \sup_{\Pi \in \calp} \sum_{[s,t] \in \Pi} \| \gamma(t) - \gamma(s) \|,
\end{align*}
where $\calp$ denotes the set of all partitions $\Pi = \{ 0 = t_0 < t_1 < \ldots < t_n = 1 \}$ for some $n \in \bbn$. Next, we define the so-called \emph{intrinsic metric} (or \emph{length metric}) $d_L : \cald \times \cald \to [0,\infty]$ as
\begin{align*}
d_L(x,y) := \inf \{ L(\gamma) : \gamma \in \Gamma(x,y) \}, \quad x,y \in \cald,
\end{align*}
where $\Gamma(x,y)$ denotes the set of all continuous maps $\gamma : [0,1] \to \cald$ with $\gamma(0) = x$ and $\gamma(1) = y$.

\begin{proposition}\label{prop-length-space}
Suppose that the closed subset $\cald$ is locally compact, and that there is a constant $C > 0$ such that
\begin{align}\label{length-metric-cond}
d_L(x,y) \leq C \| x-y \|, \quad  x,y \in \cald.
\end{align}
Then $\cald$ has the Heine-Borel property.
\end{proposition}

\begin{proof}
Note that the metric space $(\cald,d_L)$ is a so-called length space. Moreover, by \eqref{length-metric-cond} we have
\begin{align*}
d_L(x,y) \leq C \cdot d(x,y) \leq C \cdot d_L(x,y), \quad  x,y \in \cald,
\end{align*}
where the metric $d$ is given by $d(x,y) = \| x-y \|$. Therefore, the two metrics $d_L$ and $d$ generate the same topology on $\cald$, and hence, the length space $(\cald,d_L)$ is locally compact and complete. Consequently, by the Hopf-Rinow-Cohn-Vossen theorem (see \cite[Thm. 2.5.28]{Burago}) the space $(\cald,d_L)$ has the Heine-Borel property, and thus, by Lemma \ref{lemma-HB-equivalent} the space $(\cald,d)$ has the Heine-Borel property as well.
\end{proof}

\begin{remark}
Note that condition \eqref{length-metric-cond} necessarily requires that $d_L(x,y) < \infty$ for all $x,y \in \cald$. For this, the closed subset $\cald$ must at least be path-connected.
\end{remark}

\begin{corollary}
If the closed subset $\cald$ is locally compact and convex, then it has the Heine-Borel property.
\end{corollary}

\begin{proof}
This is an immediate consequence of Proposition \ref{prop-length-space}, since by convexity of $\cald$ we have $d_L(x,y) = \| x-y \|$ for all $x,y \in \cald$.
\end{proof}

\subsection{Main results}

We are now in place to state our main theorem which provides necessary and sufficient conditions on $(b,C)$ for the stochastic invariance of $\mathcal D$. For this, for an operator $T \in L(H)$ and $j \in \bbn$ we agree on the notation $T^j = T e_j$, where $\{ e_j \}_{j \in \bbn}$ denotes the orthonormal basis such that \eqref{diagonal} is fulfilled.  We denote by $\caln_{\cald}^{1,\pr}(x)$ the set of all proximal normals to $\cald$ at $x$; that is
\begin{align}\label{eq:proximalnormal}
\caln_{\cald}^{1,\pr}(x) := \{ u \in H : d_{\cald}(x+u) = \| u \| \}.
\end{align}
 We will need the concept of weak convergence of series, defined as follows.
{\begin{definition}\label{D:weakconvseries}
Let $(x_j)_{j \in \bbn} \subset H$ and $x \in H$. Then the series $\sum_{j=1}^{\infty} x_j$ \emph{converges weakly} to $x$ if
\begin{align*}
\sum_{j=1}^{\infty} \la u, x_j \ra = \la u,x \ra \quad \text{for all $u \in H.$}
\end{align*}
In this case we write
\begin{align*}
\sigma \text{-} \sum_{j=1}^{\infty} x_j = x.
\end{align*}
\end{definition}}

\begin{remark}\label{rem-weak-conv-ser}
By the Fr\'{e}chet-Riesz theorem, the series $\sum_{j=1}^{\infty} x_j$ converges weakly to some $x \in H$ if there is a continuous linear functional $\ell \in L(H,\bbr)$ such that
\begin{align*}
\sum_{j=1}^{\infty} \la u, x_j \ra = \ell(u) \quad \text{for all $u \in H.$}
\end{align*}
\end{remark}

The necessary and sufficient condition \eqref{cond-drift-N} below will involve the series
\begin{align}\label{eq:seriesCC+}
\sum_{j=1}^{\infty} D C^j(x) (C C^+)^j(x),
\end{align}
which can be shown to be  weakly convergent in the sense of Definition~\ref{D:weakconvseries} for each $x \in \cald$,  thanks to Assumption~\ref{ass-extension},  see  Lemma \ref{lemma-series-1}. In view of the upcoming result, recall that we have introduced the positive maximum principle in Definition \ref{def-pos-max}.

\begin{theorem}\label{thm-main} Let $\cald \subset H$ be a closed subset and  Assumptions \ref{ass-lin-growth}, \ref{ass-extension}, \ref{ass-Sigma-self-adjoint}, \ref{ass-closed-range},  \ref{ass-Heine-Borel}   be in force.
The following statements are equivalent:
\begin{enumerate}
\item[(i)] $\cald$ is stochastically invariant with respect to the diffusion \eqref{SDE}.

\item[(ii)] For all $x \in \cald$ and all $u \in \caln_{\cald}^{1,\pr}(x)$ we have
\begin{align}\label{cond-C-zero}
&C(x)u = 0,
\\ \label{cond-drift-N} &\la u, b(x) \ra - \frac{1}{2} \sum_{j=1}^{\infty} \la u, D C^j(x) (C C^+)^j(x) \ra \leq 0.
\end{align}

\item[(iii)] The generator $\call$ satisfies the positive maximum principle.
\end{enumerate}
\end{theorem}

\begin{proof}
(i) $\Rightarrow$ (ii): This is a consequence of Theorem \ref{thmnecessity}.
\\ (ii) $\Rightarrow$ (iii): This is a consequence of Proposition \ref{prop-pos-max-prin}.
\\ (iii) $\Rightarrow$ (i): This is a consequence of Theorem \ref{thm-suff}.
\end{proof}

\begin{remark}\label{rem-drift-cond}
The series \eqref{eq:seriesCC+} does not need to be norm convergent. However, as already mentioned above, it is weakly convergent, and hence the series appearing in \eqref{cond-drift-N} converges for every $u \in H$. Moreover, according to Theorem \ref{thm-MP-characterizations} we have $C(x) C^+(x) = P_C(x)$ for all $x \in \cald$, where $P_C(x)$ denotes the orthogonal projection on the range of $C(x)$. Therefore, condition \eqref{cond-drift-N} can equivalently be written as
\begin{align}\label{cond-drift-N-projection}
\la u, b(x) \ra - \frac{1}{2} \sum_{j=1}^{\infty} \la u, D C^j(x) P_C^j(x) \ra \leq 0.
\end{align}
\end{remark}

The series \eqref{eq:seriesCC+} is closely related to the more standard Stratonovich series given by  
\begin{align}\label{eq:strato}
\sum_{j=1}^{\infty} D \sigma^j(x) \sigma^j(x),
\end{align}  
in the case where $\sigma$ is differentiable. This connection is made precise in the following proposition. We emphasize that our previous assumptions do not need to be in force for this result; in particular, the linear operators $\Sigma(x)$ do not need to be self-adjoint.

\begin{proposition}\label{P:C=sigma2}
Suppose that $\sigma  \in C^{1}(H,L_2^0(H))$. Then the following statements are true:
\begin{enumerate}
\item The mapping $C : H \to L_1^+(H)$ defined as
\begin{align}\label{C-def-sigma-smooth}
C(x) := \Sigma(x) \Sigma(x)^*, \quad x \in H
\end{align}
is also of class $C^1$.

\item If $\sigma$ is even of class $C_b^{1}$, then $C$ is also of class $C_{b}^1$.

\item For all $x \in H$ the series
\begin{align}\label{series-sigma2}
\sum_{j=1}^{\infty} D C^j(x) P_C^j(x)
\end{align}
is weakly convergent, and for all $u \in H$ we have
\begin{align}\label{series-sigma2-u}
\sum_{j=1}^{\infty} \langle u, D C^j(x) P_C^j(x) \rangle = \tr \big( DC(x) P_C(x)u \big).
\end{align}

\item For all $x \in H$ and all $u \in \ker (\Sigma(x)^*)$ we have
\begin{align}\label{series-sigma2-Stratonovich}
\sum_{j=1}^{\infty} \langle u, D C^j(x) P_C^j(x) \rangle = \sum_{j=1}^{\infty} \langle u, D \sigma^j(x) \sigma^j(x) \rangle.
\end{align}
\end{enumerate}
\end{proposition}  

\begin{proof}
See Section~\ref{S:proofofequalitydriftseries}.
\end{proof}  

\begin{remark}\label{rem-kernels-2}
Note that by the definition \eqref{C-def-sigma-smooth} of $C$ and Lemma \ref{lemma-kernels} we have $\ker(C(x)) = \ker(\Sigma(x)^*) = \ker(\sigma(x)^*)$ for all $x \in H$.
\end{remark}

We can relax the finite-range condition in Assumption~\ref{ass-closed-range} on $C$ to characterize the invariance in terms of the Stratonovich series \eqref{eq:strato}, if we impose additional smoothness on $\sigma$. 

\begin{assumption}\label{ass-sigma-smooth}
$\sigma \in C^{2}(H,L_2^0(H))$.
\end{assumption} 

Clearly, the smoothness assumption on $\sigma$ in Assumption \ref{ass-sigma-smooth} is stronger than that in Assumption~\ref{ass-extension}. Indeed, in $\mathbb{R}$, for $\sigma(x) = \sqrt{x}$, we have $C(x) = x$, which satisfies Assumption~\ref{ass-extension}, but $\sigma$ does not satisfy Assumption \ref{ass-sigma-smooth}.  

The necessary and sufficient conditions can now be stated as follows. 

\begin{theorem}\label{thm-main-sigma}  
Let $\cald \subset H$ be a closed subset, and  Assumptions \ref{ass-lin-growth}, \ref{ass-Sigma-self-adjoint}, \ref{ass-Heine-Borel}, and \ref{ass-sigma-smooth} be in force. The following statements are equivalent:
\begin{enumerate}
\item[(i)] $\cald$ is stochastically invariant with respect to the diffusion \eqref{SDE}.  

\item[(ii)] For all $x \in \cald$ and all $u \in \caln_{\cald}^{1,\pr}(x)$, we have  
\begin{align}\label{cond-sigma-zero}
&\sigma(x)^* u = 0,  
\\ \label{cond-drift-sigma-N}  
&\langle u, b(x) \rangle - \frac{1}{2} \sum_{j=1}^{\infty} \langle u, D \sigma^j(x) \sigma^j(x) \rangle \leq 0.
\end{align}
\item[(iii)] The generator $\call$ satisfies the positive maximum principle.
\end{enumerate}
\end{theorem}  

\begin{proof}
(i) $\Rightarrow$ (ii): This is a consequence of Proposition \ref{prop-nec-sigma-C2}.
\\ (ii) $\Rightarrow$ (iii): This is a consequence of Proposition \ref{prop-gen-u-v-sigma}.
\\ (iii) $\Rightarrow$ (i): This is a consequence of Theorem \ref{thm-suff}.
\end{proof}

\begin{remark}
The series \eqref{eq:strato} does not need to be norm convergent; it does not even need to be weakly convergent. However, due to Proposition \ref{P:C=sigma2}, for every $u \in \ker (\Sigma(x)^*)$ the series appearing in \eqref{cond-drift-sigma-N} converges, and coincides with the series appearing in \eqref{cond-drift-N-projection}. Noting that $\caln_{\cald}^{1,\pr}(x) \subset \ker (\Sigma(x)^*)$ due to \eqref{cond-sigma-zero} and Remark \ref{rem-kernels-2}, this ensures that condition \eqref{cond-drift-sigma-N} is meaningful.
\end{remark}

\subsection{The convergence of the series}

In this section we will show that the series \eqref{eq:seriesCC+} is weakly convergent.

\begin{lemma}\label{lemma-series-1}
For each $x \in \cald$ the following statements are true:
\begin{enumerate}
\item The series \eqref{eq:seriesCC+} is weakly convergent.

\item For all $u \in H$ we have
\begin{align*}
\sum_{j=1}^{\infty} \la u, D C^j(x) (C C^+)^j(x) \ra &= \tr \big( D C(x) ( C(x) C^+(x) ) u \big)
\\ &= \tr \big( D C(x) P_C(x) u \big)
\\ &= \sum_{j=1}^{\infty} \la u, D C(x) (\Sigma(x)^+ e_j) (C(x)\Sigma(x)^+ e_j) \ra.
\end{align*}
\end{enumerate}
\end{lemma}

\begin{proof}
The weak convergence of the series \eqref{eq:seriesCC+} and the first two identities are an immediate consequence of Proposition \ref{prop-series-app} and Theorem \ref{thm-MP-characterizations}. Furthermore, by Corollary \ref{cor-self-adjoint} and Lemma \ref{lemma-adjoint-isometry} the operator $D C(x) v$ is self-adjoint for each $v \in H$. Let us denote by $Cu : H \to H$ the mapping $y \mapsto C(y)u$. Then, by Proposition \ref{prop-smooth-linear}, Lemma \ref{lemma-inverse-rules} and Lemma \ref{lemma-trace-commute} we obtain
\begin{align*}
&\sum_{j=1}^{\infty} \la u, D C(x) (\Sigma(x)^+ e_j) (C(x)\Sigma(x)^+ e_j) \ra = \sum_{j=1}^{\infty} \la D C(x) (\Sigma(x)^+ e_j) u, C(x)\Sigma(x)^+ e_j \ra
\\ &= \sum_{j=1}^{\infty} \la D (C u)(x)(\Sigma(x)^+ e_j), C(x)\Sigma(x)^+ e_j \ra
\\ &= \sum_{j=1}^{\infty} \la \Sigma(x)^+ C(x) D (C u)(x) \Sigma(x)^+ e_j, e_j \ra = \tr \big( \Sigma(x)^+ C(x) D (C u)(x) \Sigma(x)^+ \big)
\\ &= \tr \big( D (Cu)(x) \Sigma(x)^+ \Sigma(x)^+ C(x) \big) = \tr \big( D (Cu)(x) ( \Sigma(x)^2 )^+ C(x) \big)
\\ &= \tr \big( D (Cu)(x) C^+(x) C(x) \big) = \tr \big( D (Cu)(x) C(x) C^+(x) \big)
\\ &= \tr \big( D C(x) ( C(x) C^+(x) ) u \big),
\end{align*}
which establishes the last identity.
\end{proof}

\subsection{Equivalent characterizations}\label{sec-eq-char}

In this section we provide equivalent characterizations of the invariance conditions appearing in Theorems \ref{thm-main} and \ref{thm-main-sigma}. In view of the upcoming findings, let us recall that $T_{\cald}^c(x)$ denotes the \emph{Clarke tangent cone} to $\cald$ at $x$, defined as
\begin{align*}
T_{\cald}^c(x) := \bigg\{ v \in H : \lim_{t \to 0^+ \atop \cald \ni x' \to x} \frac{d_{\cald}(x'+tv)}{t} = 0 \bigg\},
\end{align*}
and that $T_{\cald}^{\sigma}(x)$ denotes the \emph{weak Bouligand tangent cone} (or \emph{weak contingent cone}) to $\cald$ at $x$, defined as the set of all $v \in H$ such that there are sequences $(t_n)_{n \in \bbn} \subset (0,\infty)$ with $t_n \to 0^+$ and $(v_n)_{n \in \bbn} \subset H$ with $v_n \overset{\sigma}{\to} v$ such that $x + t_n v_n \in \cald$ for each $n \in \bbn$. Here the notation $v_n \overset{\sigma}{\to} v$ denotes weak convergence of the sequence; more precisely $\lim_{n \to \infty}\la u,v_n \ra = \la u,v \ra$ for all $u \in H$. Moreover, we denote by $\caln_{\cald}^{c}(x)$ the \emph{Clarke normal cone} to $\cald$ at $x$, defined as the polar cone $\caln_{\cald}^{c}(x) := T_{\cald}^{c}(x)^{\circ}$. For these cones we have the inclusions $\caln_{\cald}^{1,\pr}(x) \subset \caln_{\cald}^{c}(x)$ and $T_{\cald}^c(x) \subset T_{\cald}^{\sigma}(x)$ for each $x \in \cald$. We refer to Appendix \ref{app-geometry} for further details.

\begin{remark}\label{rem-main}
There are several equivalent characterizations of the dispersion operator condition \eqref{cond-C-zero}. Indeed, by Proposition \ref{prop-cond-vol} the following statements are equivalent:
\begin{enumerate}
\item[(i)] For all $x \in \cald$ and all $u \in \caln_{\cald}^{1,\pr}(x)$ we have \eqref{cond-C-zero}.

\item[(ii)] For all $x \in \cald$ and all $u \in \caln_{\cald}^c(x)$ we have \eqref{cond-C-zero}.

\item[(iii)] For all $x \in \cald$ and all $u \in \caln_{\cald}^{1,\pr}(x)$ we have \eqref{cond-sigma-zero}.

\item[(iv)] For all $x \in \cald$ and all $u \in \caln_{\cald}^c(x)$ we have \eqref{cond-sigma-zero}.

\item[(v)] For all $x \in \cald$ and all $w \in H_0$ we have $\sigma(x)w \in T_{\cald}^{c}(x)$

\item[(vi)] For all $x \in \cald$ and all $w \in H_0$ we have $\sigma(x)w \in T_{\cald}^{\sigma}(x)$.
\end{enumerate}
Moreover, according to Lemma \ref{lemma-adjoint-ker-ran}, condition \eqref{cond-sigma-zero} just means that $\sigma(x)w \perp u$ for all $x \in \cald$, all $w \in H_0$ and every proximal normal $u \in \caln_{\cald}^{1,\pr}(x)$.
\end{remark}

The insight of the following remark in particular applies if $\cald$ is a finite dimensional submanifold with boundary of $H$; see Section \ref{sec-manifolds}.

\begin{remark}\label{remark-tangent-cone-fin-dim}
Suppose that the linear space generated by $T_{\cald}^c(x)$ is finite dimensional for each $x \in \cald$. Then the closed range assumption (Assumption \ref{ass-closed-range}) in Theorem \ref{thm-main} can be skipped. Indeed, recall that Assumption \ref{ass-closed-range} is required in order to define the Moore-Penrose pseudoinverse $C^+$, which shows up in condition \eqref{cond-drift-N}. In this particular situation, Assumption \ref{ass-closed-range} already follows from condition \eqref{cond-C-zero}, which is a consequence of Lemma \ref{lemma-closed-range} and Remark \ref{rem-main}.
\end{remark}

\begin{remark}\label{rem-drift-prox-normal}
Concerning the drift condition \eqref{cond-drift-N}, we note that the following statements are equivalent:
\begin{enumerate}
\item[(i)] For all $x \in \cald$ and all $u \in \caln_{\cald}^{1,\pr}(x)$ we have \eqref{cond-drift-N}.

\item[(ii)] For all $x \in \cald$ and all $u \in \caln_{\cald}^p(x)$ we have \eqref{cond-drift-N}.
\end{enumerate}
Moreover, concerning the drift condition \eqref{cond-drift-sigma-N}, we note that the following statements are equivalent:
\begin{enumerate}
\item[(i)] For all $x \in \cald$ and all $u \in \caln_{\cald}^{1,\pr}(x)$ we have \eqref{cond-drift-sigma-N}.

\item[(ii)] For all $x \in \cald$ and all $u \in \caln_{\cald}^p(x)$ we have \eqref{cond-drift-sigma-N}.
\end{enumerate}
Indeed, just note that $\caln_{\cald}^p(x)$ is the cone generated by $\caln_{\cald}^{1,\pr}(x)$.
\end{remark}

The drift conditions in Theorem \ref{thm-main} and Theorem \ref{thm-main-sigma} can be expressed by means of the contingent curvature (see Definition \ref{def-curvature}) as follows.

\begin{remark}\label{rem-curvature-1}
Suppose that condition \eqref{cond-C-zero} is fulfilled for all $x \in \cald$ and $u \in \caln_{\cald}^{1,\pr}(x)$. Then the drift condition \eqref{cond-drift-N} is satisfied if and only if
\begin{align}\label{cond-curvature-C}
\la u,b(x) \ra + \frac{1}{2} \sum_{j=1}^{\infty} {\rm Curv}_{\cald}(x,u) ( (C C^+)^j(x), C^j(x) ) \leq 0
\end{align}
for all $x \in \cald$ and $u \in \caln_{\cald}^{1,\pr}(x)$. Indeed, by Remark \ref{rem-main} we have \eqref{cond-C-zero} for all $x \in \cald$ and $u \in \caln_{\cald}^c(x)$. Therefore, for all $x \in \cald$ and $u \in \caln_{\cald}^c(x)$ we obtain
\begin{align}\label{C-orth-1}
\la u,C^j(x) \ra = \la u,C(x) e_j \ra = \la C(x)u,e_j \ra = 0 \quad \forall j \in \bbn
\end{align}
as well as
\begin{align}\label{C-orth-2}
\la u, (C C^+)^j(x) \ra = \la u,C(x)C^+(x) e_j \ra = \la C(x) u, C^+(x) e_j \ra = 0 \quad \forall j \in \bbn.
\end{align}
Consequently, the stated equivalence \eqref{cond-drift-N} $\Leftrightarrow$ \eqref{cond-curvature-C} follows from Remark \ref{rem-curvature-arguments} and Proposition \ref{prop-curvature}.
\end{remark}

\begin{corollary}
Let $\cald$ be a finitely generated closed convex cone, and let Assumptions \ref{ass-lin-growth}, \ref{ass-extension}, \ref{ass-Sigma-self-adjoint} be in force. Then $\cald$ is stochastically invariant with respect to the diffusion \eqref{SDE} if and only if for all $x \in \cald$ and all $u \in \caln_{\cald}^{1,\pr}(x)$ we have
\begin{align*}
C(x)u = 0 \quad \text{and} \quad \la u,b(x) \ra \leq 0.
\end{align*}
\end{corollary}

\begin{proof}
Since the cone $\cald$ is finitely generated, the Heine-Borel property (Assumption \ref{ass-Heine-Borel}) is fulfilled, and by Lemma \ref{lemma-ccc-tangent-normal} the linear space generated by $T_{\cald}^c(x)$ is finite dimensional for each $x \in \cald$. Now, suppose that condition \eqref{cond-C-zero} is fulfilled. Then, taking into account \eqref{C-orth-1} and \eqref{C-orth-2}, by Proposition \ref{prop-curvature-cone} for all $x \in \cald$ and $u \in \caln_{\cald}^{1,\pr}(x)$ we have
\begin{align*}
{\rm Curv}_{\cald}(x,u) ( (C C^+)^j(x), C^j(x) ) = 0 \quad \forall j \in \bbn.
\end{align*}
Consequently, the result is a consequence of Theorem \ref{thm-main}, Remark \ref{remark-tangent-cone-fin-dim} and Remark \ref{rem-curvature-1}.
\end{proof}

The following condition \eqref{cond-curvature-sigma} has also been derived in \cite{Aubin-Doss} in the finite dimensional setting.

\begin{remark}\label{rem-curvature-2}
Suppose that condition \eqref{cond-sigma-zero} is fulfilled for all $x \in \cald$ and $u \in \caln_{\cald}^{1,\pr}(x)$. Then the drift condition \eqref{cond-drift-sigma-N} is satisfied if and only if
\begin{align}\label{cond-curvature-sigma}
\la u,b(x) \ra + \frac{1}{2} \sum_{j=1}^{\infty} {\rm Curv}_{\cald}(x,u) ( \sigma^j(x), \sigma^j(x) ) \leq 0
\end{align}
for all $x \in \cald$ and $u \in \caln_{\cald}^{1,\pr}(x)$. Indeed, by Remark \ref{rem-main} we have \eqref{cond-sigma-zero} for all $x \in \cald$ and $u \in \caln_{\cald}^c(x)$. Therefore, by Lemma \ref{lemma-adjoint-ker-ran} for all $x \in \cald$ and $u \in \caln_{\cald}^c(x)$ we have
\begin{align}\label{sigma-orth}
\la u,\sigma^j(x) \ra = 0 \quad \forall j \in \bbn.
\end{align}
Consequently, the stated equivalence \eqref{cond-drift-sigma-N} $\Leftrightarrow$ \eqref{cond-curvature-sigma} follows from Remark \ref{rem-curvature-arguments} and Proposition \ref{prop-curvature}.
\end{remark}

\begin{corollary}
Let $\cald$ be a finitely generated closed convex cone, and let Assumptions \ref{ass-lin-growth}, \ref{ass-Sigma-self-adjoint}, \ref{ass-sigma-smooth} be in force. Then $\cald$ is stochastically invariant with respect to the diffusion \eqref{SDE} if and only if for all $x \in \cald$ and all $u \in \caln_{\cald}^{1,\pr}(x)$ we have
\begin{align*}
\sigma(x)^* u = 0 \quad \text{and} \quad \la u,b(x) \ra \leq 0.
\end{align*}
\end{corollary}

\begin{proof}
Since the cone $\cald$ is finitely generated, the Heine-Borel property (Assumption \ref{ass-Heine-Borel}) is fulfilled. Now, suppose that condition \eqref{cond-sigma-zero} is fulfilled. Then, taking into account \eqref{sigma-orth}, by Proposition \ref{prop-curvature-cone} for all $x \in \cald$ and $u \in \caln_{\cald}^{1,\pr}(x)$ we have
\begin{align*}
{\rm Curv}_{\cald}(x,u) ( \sigma^j(x), \sigma^j(x) ) = 0 \quad \forall j \in \bbn.
\end{align*}
Therefore, the result is a consequence of Theorem \ref{thm-main} and Remark \ref{rem-curvature-2}.
\end{proof}

As Remark \ref{rem-main} shows, condition \eqref{cond-C-zero} can be characterized by means of tangent cones. Now, we would like to achieve a similar characterization for the drift condition \eqref{cond-drift-N}. For this purpose, we have to impose additional conditions
\begin{itemize}
\item on the regularity of the series \eqref{eq:seriesCC+} or the series \eqref{eq:strato},

\item or on the set $\cald$.
\end{itemize}
In view of the upcoming remark, recall that the series \eqref{eq:seriesCC+} is always weakly convergent.

\begin{remark}\label{rem-drift-aC-cont}
Suppose that the mapping $a_C : \cald \to H$ defined as
\begin{align}\label{aC-def}
a_C(x) := b(x) - \frac{1}{2} \, \sigma \text{-} \sum_{j=1}^{\infty} D C^j(x) (C C^+)^j(x)
\end{align}
is continuous. Then, by Proposition \ref{prop-normal-inward-a-cont} and Remark \ref{rem-more-general-inward-a} the following statements are equivalent:
\begin{itemize}
\item[(i)] For all $x \in \cald$ and all $u \in \caln_{\cald}^{1,\pr}(x)$ we have \eqref{cond-drift-N}.

\item[(ii)] For all $x \in \cald$ and all $u \in \caln_{\cald}^c(x)$ we have \eqref{cond-drift-N}.

\item[(iii)] For all $x \in \cald$ we have $a_C(x) \in T_{\cald}^c(x)$.

\item[(iv)] For all $x \in \cald$ we have $a_C(x) \in T_{\cald}^{\sigma}(x)$.
\end{itemize}
Moreover, the above equivalences remain true under the more general condition that for all elements $x \in \cald$ and $u \in H$, and all sequences $(x_n)_{n \in \bbn} \subset \cald$ and $(u_n)_{n \in \bbn} \subset H$  such that $u_n \in \caln_{\cald}^p(x_n)$ for all $n \in \bbn$ and $x_n \to x$ as well as $u_n \overset{\sigma}{\to} u$ it follows that $\la u_n,a_C(x_n) \ra \to \la u,a_C(x) \ra$.
\end{remark}

The insight of Remark \ref{rem-drift-aC-cont} applies if the mapping $\cald \to L(H)$, $x \mapsto P_C(x)$ is continuous. This is a consequence of Lemma \ref{lemma-series-1} and Lemma \ref{lemma-proj-continuous}.

\begin{remark}\label{rem-drift-a-sigma-cont}
Suppose that the series \eqref{eq:strato} is weakly convergent, and that the mapping $a_{\sigma} : \cald \to H$ defined as
\begin{align}\label{a-sigma-def}
a_{\sigma}(x) := b(x) - \frac{1}{2} \, \sigma \text{-} \sum_{j=1}^{\infty} D \sigma^j(x) \sigma^j(x)
\end{align}
is continuous. Then, by Proposition \ref{prop-normal-inward-a-cont} and Remark \ref{rem-more-general-inward-a} the following statements are equivalent:
\begin{itemize}
\item[(i)] For all $x \in \cald$ and all $u \in \caln_{\cald}^{1,\pr}(x)$ we have \eqref{cond-drift-sigma-N}.

\item[(ii)] For all $x \in \cald$ and all $u \in \caln_{\cald}^c(x)$ we have \eqref{cond-drift-sigma-N}.

\item[(iii)] For all $x \in \cald$ we have $a_{\sigma}(x) \in T_{\cald}^c(x)$.

\item[(iv)] For all $x \in \cald$ we have $a_{\sigma}(x) \in T_{\cald}^{\sigma}(x)$.
\end{itemize}
Moreover, the additional statement from Remark \ref{rem-drift-aC-cont} holds true with $a_C$ replaced by $a_{\sigma}$.
\end{remark}

Now, we will impose an additional condition on the set $\cald$. For the concept of $\varphi$-convexity, which weakens the notion of convexity, and the more general concept of local $\varphi$-convexity we refer to Definitions \ref{def-phi-convex} and \ref{def-local-phi-convex}.

\begin{remark}\label{rem-phi-convex}
Suppose that the set $\cald$ is locally $\varphi$-convex. Denoting by $a_C : \cald \to H$ the mapping \eqref{aC-def}, the following statements are equivalent:
\begin{itemize}
\item[(i)] For all $x \in \cald$ and all $u \in \caln_{\cald}^{1,\pr}(x)$ we have \eqref{cond-drift-N}.

\item[(ii)] For all $x \in \cald$ and all $u \in \caln_{\cald}^c(x)$ we have \eqref{cond-drift-N}.

\item[(iii)] For all $x \in \cald$ we have $a_C(x) \in T_{\cald}^c(x)$.
\end{itemize}
This is a consequence of Proposition \ref{prop-a-phi-convex}.
\end{remark}

\begin{remark}\label{rem-phi-convex-sigma}
Suppose that the set $\cald$ is locally $\varphi$-convex, and that that the series \eqref{eq:strato} is weakly convergent. Denoting by $a_{\sigma} : \cald \to H$ the mapping \eqref{a-sigma-def}, the following statements are equivalent:
\begin{itemize}
\item[(i)] For all $x \in \cald$ and all $u \in \caln_{\cald}^{1,\pr}(x)$ we have \eqref{cond-drift-sigma-N}.

\item[(ii)] For all $x \in \cald$ and all $u \in \caln_{\cald}^c(x)$ we have \eqref{cond-drift-sigma-N}.

\item[(iii)] For all $x \in \cald$ we have $a_{\sigma}(x) \in T_{\cald}^c(x)$.
\end{itemize}
This is a consequence of Proposition \ref{prop-a-phi-convex}.
\end{remark}

The findings of these two remarks in particular apply if for all $x \in \cald$ there is a closed neighborhood $C \subset H$ of $x$ such that $\cald \cap C$ is convex; see Proposition \ref{prop-phi-convex}.

\section{Invariant manifolds}\label{sec-manifolds}

In this section we apply our findings to the situation where the subset is a manifold. More precisely, let $\calm$ be a finite dimensional $C^1$-submanifold with boundary of $H$; we refer to Appendix \ref{app-submanifolds} for further details. We assume that $\calm$ is closed as a subset of $H$.

\begin{theorem}\label{thm-manifold-C1}
Let Assumptions \ref{ass-lin-growth}, \ref{ass-extension}, \ref{ass-Sigma-self-adjoint},  \ref{ass-Heine-Borel} (with $\cald = \calm$)  be in force. The following statements are true:
\begin{enumerate}
\item The submanifold $\calm$ is stochastically invariant with respect to the diffusion \eqref{SDE} if and only if for all $x \in \calm$ and all $w \in H_0$ we have
\begin{align}\label{cond-sigma-manifold}
&\sigma(x)w \in
\begin{cases}
T_x \mathcal{M}, & \text{if $x \in \mathcal{M} \setminus \partial \mathcal{M}$,}
\\ T_x \partial \mathcal{M}, & \text{if $x \in \partial \mathcal{M}$,}
\end{cases}
\end{align}
and for all $x \in \calm$ and all $u \in \caln_{\calm}^{1,\pr}(x)$ we have \eqref{cond-drift-N}.

\item Suppose, in addition, that one of the following conditions is satisfied:
\begin{itemize}
\item The submanifold $\calm$ is of class $C^2$.

\item The mapping $a_C : \calm \to H$ defined in \eqref{aC-def} is continuous.

\item The mapping $\calm \to L(H)$, $x \mapsto P_C(x)$ is continuous.
\end{itemize}
Then the submanifold $\calm$ is stochastically invariant with respect to the diffusion \eqref{SDE} if and only if for all $x \in \calm$ and all $w \in H_0$ we have \eqref{cond-sigma-manifold}, and for all $x \in \calm$ we have
\begin{align}\label{cond-drift-manifold}
b(x) - \frac{1}{2} \, \sigma \text{-} \sum_{j=1}^{\infty} D C^j(x) (C C^+)^j(x) \in
\begin{cases}
T_x \mathcal{M}, & \text{if $x \in \mathcal{M} \setminus \partial \mathcal{M}$,}
\\ (T_x \mathcal{M})_+, & \text{if $x \in \partial \mathcal{M}$.}
\end{cases}
\end{align}
\end{enumerate}
\end{theorem}

\begin{proof}
By Proposition \ref{prop-submanifold-sigma-regular}, for each $x \in \calm$ we have
\begin{align}\label{tangent-cone-manifold}
T_{\calm}^c(x) = T_{\calm}^{\sigma}(x) =
\begin{cases}
T_x \calm, & \text{if $x \in \calm \setminus \partial \calm$,}
\\ (T_x \calm)_+, & \text{if $x \in \partial \calm$.}
\end{cases}
\end{align}
Furthermore, by Remark \ref{rem-tang-boundary} we have
\begin{align}\label{tangent-space-at-boundary}
T_x \partial \calm = (T_x \calm)_+ \cap - (T_x \calm)_+, \quad x \in \partial \calm.
\end{align}
After these preparations, we can prove the two statements as follows:
\begin{enumerate}
\item Taking into account Remark \ref{rem-main} and \eqref{tangent-cone-manifold}, \eqref{tangent-space-at-boundary}, this statement is a consequence of Theorem \ref{thm-main} and Remark \ref{remark-tangent-cone-fin-dim}.

\item If the submanifold $\calm$ is of class $C^2$, then by Proposition \ref{prop-submanifold-phi-regular} it is locally $\varphi$-convex. Consequently, taking into account \eqref{tangent-cone-manifold}, this statement follows from the first statement as well as Remark \ref{rem-drift-aC-cont}, Lemma \ref{lemma-proj-continuous} and Remark \ref{rem-phi-convex}.
\end{enumerate}
\end{proof}

\begin{theorem}\label{thm-manifold-C1-sigma}
Let Assumptions \ref{ass-lin-growth}, \ref{ass-Sigma-self-adjoint},  \ref{ass-Heine-Borel} (with $\cald = \calm$), and \ref{ass-sigma-smooth} be in force. The following statements are true:
\begin{enumerate}
\item The submanifold $\calm$ is stochastically invariant with respect to the diffusion \eqref{SDE} if and only if for all $x \in \calm$ and all $w \in H_0$ we have \eqref{cond-sigma-manifold}, and for all $x \in \calm$ and all $u \in \caln_{\calm}^{1,\pr}(x)$ we have \eqref{cond-drift-sigma-N}.

\item Suppose, in addition, that the series \eqref{eq:strato} is weakly convergent, and that one of the following conditions is satisfied:
\begin{itemize}
\item The submanifold $\calm$ is of class $C^2$.

\item The mapping $a_{\sigma} : \calm \to H$ defined in \eqref{a-sigma-def} is continuous.
\end{itemize}
Then the submanifold $\calm$ is stochastically invariant with respect to the diffusion \eqref{SDE} if and only if for all $x \in \calm$ and all $w \in H_0$ we have \eqref{cond-sigma-manifold}, and for all $x \in \calm$ we have
\begin{align}\label{cond-drift-Stratonovich-manifold}
b(x) - \frac{1}{2} \, \sigma \text{-} \sum_{j=1}^{\infty} D \sigma^j(x) \sigma^j(x) \in
\begin{cases}
T_x \mathcal{M}, & \text{if $x \in \mathcal{M} \setminus \partial \mathcal{M}$,}
\\ (T_x \mathcal{M})_+, & \text{if $x \in \partial \mathcal{M}$.}
\end{cases}
\end{align}
\end{enumerate}
\end{theorem}

\begin{proof}
As in the proof of Theorem \ref{thm-manifold-C1} we derive \eqref{tangent-cone-manifold}, \eqref{tangent-space-at-boundary} and proceed as follows:
\begin{enumerate}
\item Taking into account Remark \ref{rem-main} and \eqref{tangent-cone-manifold}, \eqref{tangent-space-at-boundary}, this statement is a consequence of Theorem \ref{thm-main-sigma}.

\item If the submanifold $\calm$ is of class $C^2$, then by Proposition \ref{prop-submanifold-phi-regular} it is locally $\varphi$-convex. Consequently, taking into account \eqref{tangent-cone-manifold}, this statement follows from the first statement as well as Remark \ref{rem-drift-a-sigma-cont} and Remark \ref{rem-phi-convex-sigma}.
\end{enumerate}
\end{proof}

Note that Theorem \ref{thm-manifold-C1-sigma} is in line with the findings from earlier works about invariant manifolds in infinite dimension; see, e.g. \cite{Filipovic-inv, Nakayama, FTT-manifolds}. In contrast to these papers, we only require weak convergence of the Stratonovich series \eqref{eq:strato}.

\begin{remark}
Concerning condition \eqref{cond-drift-N} in Theorem \ref{thm-manifold-C1} and condition \eqref{cond-drift-sigma-N} in Theorem \ref{thm-manifold-C1-sigma}, note that $\caln_{\calm}^{1,\pr}(x)$ can in each case be replaced by $\caln_{\calm}^p(x)$, which is due to Remark \ref{rem-drift-prox-normal}. Moreover, for a submanifold $\calm$ and a point $x \in \calm$ the cone $\caln_{\calm}^p(x)$ generated by all proximal normals satisfies the following inclusions:
\begin{enumerate}
\item If $x \in \mathcal{M} \setminus \partial \mathcal{M}$, then we have
\begin{align*}
\caln_{\calm}^p(x) \subset (T_x \calm)^{\perp}.
\end{align*}
\item If $x \in \partial \mathcal{M}$, then we have
\begin{align*}
\caln_{\calm}^p(x) \subset (T_x \calm)^{\perp} \oplus \lin^+ \{ n_x \},
\end{align*}
where $n_x$ denotes the outward pointing normal vector to $\partial \calm$ at $x$.
\end{enumerate}
Recalling that $\caln_{\calm}^p(x) \subset \caln_{\calm}^{\sigma}(x)$,
this is a consequence of Proposition \ref{prop-submanifold-sigma-regular}.
\end{remark}

\section{The necessity proof}\label{S:necessity}

In this section we establish the necessity conditions for stochastic invariance of our two main results. In Section \ref{S:proofnecessitysigma} this is done in the situation of Theorem \ref{thm-main-sigma}, and in Section \ref{S:proofnecessityC} this is done in the situation of Theorem \ref{thm-main}. Both proofs rely on the following key lemma concerning the implications of the non-positivity of expressions involving terms up to double stochastic integrals.

\begin{lemma}\label{lemmadoubleintstohilbert}
	Let $\alpha \in H_0$ and $(\beta_t)_{t\geq 0}$,  $(\gamma_t)_{t\geq 0}$ and  $(\theta_t)_{t\geq 0}$ be predictable processes taking values respectively in $H_0$, $ L_2(H_0)$ and $\mathbb{R}$ and satisfying
	\begin{enumerate}
		\item
		$\beta$ is bounded,
		\item  $\int_{0}^t \|\gamma_s\|_{L_2(H_0)}^2 ds < \infty$, for all $ t\geq 0$,
		\item
		there exists  $\eta >0$ such that
		\begin{equation}\label{eqgammacontinuity}
		\int_{0}^{t} \int_0^s \bbe \big[ \|\gamma_r - \gamma_0\|_{L_2(H_0)}^2 \big] dr ds = O(t^{2+\eta}) \quad \mbox{for } t \rightarrow 0,
		\end{equation}
		\item
		$\theta$ is a.s.~continuous at $0$.
	\end{enumerate}
	Suppose that for all $t \geq 0$
	\begin{equation}\label{eqintegraledoublehilbert}
	\int_{0}^{t} \theta_s ds + \int_{0}^{t} \left( \alpha + \int_{0}^{s} \beta_rdr +  \int_{0}^{s} \gamma_r dW_r \right)  dW_s  \leq 0.
	\end{equation}
	Then we have $\alpha = 0$. If, moreover, the series
		\begin{align}\label{series-gamma-0}
		\tr(\gamma_0) := \sum_{j=1}^{\infty} \la \gamma_0 f_j, f_j \ra_{H_0}
		\end{align}
		converges\footnote{Note that $\gamma_0$ is only assumed to be a Hilbert-Schmidt operator, which does not need to be nuclear. In case $\gamma_0$ is a nuclear operator, then the series in \eqref{series-gamma-0} coincides with the trace of $\gamma_0$.}, then $\gamma_0$ is self-adjoint and we have
		\begin{align*}
		\theta_0-\frac{1}{2}\tr( \gamma_0  ) \leq 0.
		\end{align*}
\end{lemma}

\begin{proof}
    See Section~\ref{S:proofdoublestochasticinequality}.
\end{proof}

For the rest of this section we consider the mathematical framework from Section \ref{sec-main-results}, and Assumptions \ref{ass-lin-growth}, \ref{ass-Sigma-self-adjoint} will always be in force. The Heine-Borel property (Assumption \ref{ass-Heine-Borel}) of the closed subset $\cald \subset H$ will not be required throughout this section.

\subsection{Necessity under smoothness on $\sigma$: Theorem~\ref{thm-main-sigma}.}\label{S:proofnecessitysigma}

In this section we show that in the situation of Theorem \ref{thm-main-sigma} stochastic invariance implies that the invariance conditions \eqref{cond-sigma-zero} and \eqref{cond-drift-sigma-N} are fulfilled.

\begin{proposition}\label{prop-nec-sigma-C2}
	Suppose that Assumptions \ref{ass-lin-growth}, \ref{ass-Sigma-self-adjoint}, and \ref{ass-sigma-smooth} are in force.
	Assume that $\cald$ is invariant. Then, \eqref{cond-sigma-zero} and \eqref{cond-drift-sigma-N} hold for all $x \in \cald$ and $u \in \mathcal{N}^{1,\pr}_{\mathcal{D}}(x)$.
\end{proposition}

\begin{proof} 	Let $x \in \cald$ be arbitrary and let $(X,W)$ denote a  {weak solution} starting at $X_0=x$ such that  $X_t \in \mathcal{D}$ for all $t \geq 0 $. We also fix an arbitrary $u \in \mathcal{N}^{1,\pr}_{\mathcal{D}}(x)$.\\
	{\it Step 1.}
	We first claim that there exists a function $\phi \in C^{\infty}(H,\mathbb{R})$ such that $\displaystyle\max_{\mathcal{D}} \phi = \phi(x) = 0$ and $D \phi(x)= u$. Indeed, by Proposition \ref{prop-normal-2} there exists a constant $\kappa = \kappa(x,u) > 0$ such that
	\begin{align}\label{inequ-normal}
	\la u,y-x \ra \leq \frac{\kappa}{2} \| y-x \|^2 \quad \text{for all $y \in \cald$.}
	\end{align}
	We define $\phi : H \to \bbr$ as
	\begin{align*}
	\phi(y) := \la u,y-x \ra - \frac{\kappa}{2} \| y-x \|^2, \quad y \in H.
	\end{align*}
	Then $\phi$ is of class $C^{\infty}$ with first order derivative
	\begin{align*}
	D \phi(y)v &= \la u,v \ra - \kappa \la y-x, v \ra, \quad y,v \in H.
	\end{align*}
	In particular we have $D \phi(x)= u$, where we recall the conventions from Remark \ref{remark-derivatives}. Furthermore, we have $\phi(x) = 0$, and by \eqref{inequ-normal} we have $\phi(y) \leq 0$ for all $y \in \cald$, showing that $\displaystyle\max_{\mathcal{D}} \phi = \phi(x)$. \\
	{\it Step 2.} Choosing an appropriate open and bounded neighborhood $N(x) \subset H$ of $x$ and a stopping time $\tau > 0$ such that $X^{\tau} \in N(x)$, by Proposition \ref{prop-modify-fct-bdd} we may assume that $\phi$ and $\sigma$
	are of class $C_b^2$. Moreover, the stopped process $X^{\tau}$ is bounded, and we have $X^{\tau} \in \cald$.\\
	{\it Step 3.} Since $X^{\tau} \in \cald$, it follows by It\^o's formula (Theorem \ref{thm-Ito}) that, for all $t \geq 0$
	\begin{align}\label{Ito-ineqn-phi}
	0 \geq \phi(X_t^{\tau}) - \phi(x)
	= \int_0^t\mathcal L \phi (X_s^{\tau}) ds + \int_0^t D\phi(X_s^{\tau})\sigma(X_s^{\tau}) dW_s.
	\end{align}
	We introduce the mapping $\Phi : H \to L_2(H_0,\bbr)$ as
	\begin{align*}
	\Phi(y) := D \phi(y) \sigma(y), \quad y \in H.
	\end{align*}
	Recalling the conventions from Remark \ref{remark-derivatives}, we can write this mapping for all $y \in H$ as
	\begin{align*}
	\Phi(y) v = \la D \phi(y), \sigma(y)v \ra = \la \sigma^*(y) D \phi(y), v \ra, \quad v \in H_0.
	\end{align*}
	By the identification $L_2(H_0,\bbr) \cong H_0$ from Lemma \ref{lemma-Riesz-Parseval}, we may identify $\Phi$ with the mapping $\Phi : H \to H_0$ given by
	\begin{align}\label{Phi-in-proof}
	\Phi(y) := \sigma^*(y) D \phi(y) = B( \sigma^*(y), D \phi(y) ).
	\end{align}
	where $B : L_2(H,H_0) \times H \to H_0$ denotes the continuous bilinear operator
	\begin{align*}
	B(T,z) := Tz.
	\end{align*}
	Since according to Lemma \ref{lemma-adjoint-isometry} the linear mapping
	\begin{align}\label{isom-isom-HS}
	L_2(H_0,H) \to L_2(H,H_0), \quad T \mapsto T^*
	\end{align}
	is an isometry, by Proposition \ref{prop-smooth-linear} the mapping $\sigma^* : H \to L_2(H,H_0)$ is also of class $C_b^2$. Since $\phi : H \to \bbr$ is of class $C_b^2$ as well, by Proposition \ref{prop-Leibniz} we have $\Phi \in C_{b}^2(H,H_0)$. Moreover, recalling Remark \ref{rem-Wiener-series-1}, by \eqref{Ito-ineqn-phi} we have for all $t \geq 0$
	\begin{align*}
	0 \geq \int_0^t\mathcal L \phi (X_s^{\tau}) ds + \int_0^t \Phi(X_s^{\tau}) dW_s.
	\end{align*}
	Another application of It\^o's formula (Theorem \ref{thm-Ito}) yields
	\begin{align*}
	\Phi(X_s^{\tau}) = \Phi(x) + \int_0^s\mathcal L \Phi (X_r^{\tau}) dr + \int_0^s D\Phi(X_r^{\tau})\sigma(X_r^{\tau}) dW_r,
	\end{align*}
	where $\call \Phi : H \to H_0$ is defined according to \eqref{generazor-def-G}. Hence, we obtain
	\begin{align*}
	0 \geq \int_0^t\mathcal L \phi (X_s^{\tau}) ds + \int_0^t \bigg( \Phi(x) + \int_0^s\mathcal L \Phi (X_r^{\tau}) dr + \int_0^s D\Phi(X_r^{\tau})\sigma(X_r^{\tau}) dW_r \bigg) dW_s.
	\end{align*}
	Noting Remarks \ref{rem-Wiener-series-1} and \ref{rem-Wiener-series-2}, we can write this inequality as \eqref{eqintegraledoublehilbert}, where the constant $\alpha \in H_0$, the $H_0$-valued continuous process $\beta$, the $L_2(H_0)$-valued continuous process $\gamma$, and the $\bbr$-valued continuous process $\theta$ are given by
	\begin{align*}
	\alpha &:= \Phi(x), \quad \quad \beta_r := \mathcal L \Phi (X_r^{\tau}), \quad r \geq 0,
	\\ \gamma_r &:= \Psi(X_r^{\tau}), \quad r \geq 0, \quad \quad
	\theta_s := \mathcal L \phi (X_s^{\tau}), \quad s \geq 0,
	\end{align*}
	and where $\Psi : H \to L_2(H_0)$ is defined as $\Psi(y) := D \Phi(y) \sigma(y)$ for each $y \in H$.\\
{\it Step 4.}	We check that we can apply   Lemma \ref{lemmadoubleintstohilbert}.  Indeed, the process $\theta$ is continuous. Taking into account the linear growth condition \eqref{linear-growth}, the fact that $\Phi$ is of class $C_b^2$, and the boundedness of $X^{\tau}$, the process $\beta$ is bounded by Lemma \ref{lemma-generator-growth}. Note that
\begin{align*}
\Psi(y) = B(D \Phi(y), \sigma(y)), \quad y \in H,
\end{align*}
where $B : L(H,H_0) \times L_2(H_0,H) \to L_2(H_0)$ denotes the continuous bilinear operator $B(T,S) = TS$. Since $D \Phi : H \to L(H,H_0)$ is of class $C_b^1$ and $\sigma : H \to L_2(H_0,H)$ is of class $C_b^2$, by Proposition \ref{prop-Leibniz} we obtain $\Psi \in C_{b}^1(H,L_2(H_0))$. Therefore, we have $\int_0^t \| \gamma_s \|_{L_2(H_0)}^2 ds < \infty$, for all $t \geq 0$. Moreover, by Proposition \ref{prop-Cb1-Lipschitz} the mapping $\Psi : H \to L_2(H_0)$ is Lipschitz continuous, and hence, the estimate \eqref{eqgammacontinuity} follows from Lemma \ref{lemma-short-time-double-dt} below.\\
{\it Step 5.} We therefore  can apply  Lemma \ref{lemmadoubleintstohilbert} to \eqref{eqintegraledoublehilbert} to deduce  that $\alpha=0$, which implies $D \phi(x) \sigma(x) = 0$. Recalling $D \phi(x) = u$, this gives us $\la u,\sigma(x)w \ra = 0$ for all $w \in H_0$, and hence $\la \sigma(x)^* u,w \ra_{H_0} = 0$ for all $w \in H_0$, showing that $\sigma(x)^* u = 0$. This proves \eqref{cond-sigma-zero}.\\
{\it Step 6.}	Since $\sigma(x)^* u = 0$, by Lemma \ref{lemma-trace-decomp} below the series \eqref{series-gamma-0} converges, and we have
\begin{align}\label{tr-series-expr}
\tr(\gamma_0) = \sum_{j=1}^{\infty} \la D \phi(x), D \sigma^j(x) \sigma^j(x) \ra + \sum_{j=1}^{\infty} \la D^2 \phi(x) \sigma^j(x), \sigma^j(x) \ra.
\end{align}
We therefore  can apply  Lemma \ref{lemmadoubleintstohilbert} again to \eqref{eqintegraledoublehilbert} to deduce  that
	$\theta_0-\frac{1}{2}\Tr(\gamma_0) \leq 0$. Taking into account Lemma \ref{lemma-generator-series}, this inequality combined with \eqref{tr-series-expr} shows that
	\begin{align*}
	0 &\geq \call \phi(x) - \frac{1}{2} \tr ( \gamma_0 )
	= \la D \phi(x), b(x) \ra + \frac{1}{2} \sum_{j=1}^{\infty} \la D^2 \phi(x) \sigma^j(x),\sigma^j(x) \ra
	\\ &\quad - \frac{1}{2} \bigg( \sum_{j=1}^{\infty} \la D \phi(x), D \sigma^j(x) \sigma^j(x) \ra + \sum_{j=1}^{\infty} \la D^2 \phi(x) \sigma^j(x), \sigma^j(x) \ra \bigg)
	\\ &= \langle D\phi(x), b(x) \rangle - \frac{1}{2} \sum_{j=1}^{\infty} \langle D\phi(x), D \sigma^j(x) \sigma^j(x) \rangle.
	\end{align*}
	Recalling that $D\phi(x) = u$, this proves \eqref{cond-drift-sigma-N}.
\end{proof}

\begin{lemma}\label{lemma-short-time-double-dt}
For $t \to 0$ we have
\begin{align*}
\int_0^t \int_0^s \bbe \big[ \| \Psi(X_r^{\tau}) - \Psi(x) \|_{L_2(H_0)}^2 \big] dr ds = O(t^3).
\end{align*}
\end{lemma}

\begin{proof}
Since $\Psi$ is Lipschitz continuous, there is a constant $M > 0$ such that
\begin{align*}
\| \Psi(x) - \Psi(y) \|_{L_2(H_0)} \leq M \| x-y \| \quad \text{for all $x,y \in H$.}
\end{align*}
Furthermore, since $X^{\tau}$ is bounded, there is a constant $C > 0$ such that $\| X^{\tau} \| \leq C$. Let $t \in [0,1]$ be arbitrary. By the linear growth condition \eqref{linear-growth} we obtain
\begin{align*}
\bbe [ \| X_t^{\tau} - x \|^2 ] & \leq 2 \, \bbe \Bigg[ \bigg\| \int_0^t b(X_s^{\tau}) ds \bigg\|^2 \Bigg] + 2 \, \bbe \Bigg[ \bigg\| \int_0^t \sigma(X_s^{\tau}) d W_s \bigg\|^2 \Bigg]
\\ & \leq 2t \bbe \bigg[ \int_0^t \| b(X_s^{\tau}) \|^2 ds \bigg] + 2 \, \bbe \bigg[ \int_0^t \| \sigma(X_s^{\tau}) \|_{L_2^0(H)}^2 ds \bigg]
\\ &\leq K t,
\end{align*}
where the constant $K > 0$ is given by $K = 2 L^2 (1 + C)^2$, and where $L > 0$ stems from \eqref{linear-growth}. Therefore, we obtain
\begin{align*}
&\int_0^t \int_0^s \bbe \big[ \| \Psi(X_r^{\tau}) - \Psi(x) \|_{L_2(H_0)}^2 \big] dr ds \leq M^2 \int_0^t \int_0^s \bbe \big[ \| X_r^{\tau} - x \|^2 \big] dr ds
\\ &\leq M^2 K \int_0^t \int_0^s r dr ds = \frac{M^2 K}{2} \int_0^t s^2 ds = \frac{M^2 K}{6} t^3.
\end{align*}
Hence, the conclusion follows.
\end{proof}

\begin{lemma}\label{lemma-trace-decomp}
The series \eqref{series-gamma-0} converges, and we have \eqref{tr-series-expr}.
\end{lemma}

\begin{proof}
Recall that $u \in \caln_{\cald}^{1,\pr}(x)$ is given by $u = D \phi(x)$. By Remark \ref{rem-kernels-2} we have $u \in \ker (\Sigma(x)^*)$. Hence, by Proposition \ref{P:C=sigma2} we have
\begin{align*}
\sum_{j=1}^{\infty} \la u, D \sigma^j(x) \sigma^j(x) \ra = \sum_{j=1}^{\infty} \la u, DC^j(x) P_C^j(x) \ra = \tr \big( D C(x) P_C(x) u \big),
\end{align*}
and by Lemma \ref{lemma-generator-series} we have
\begin{align*}
\sum_{j=1}^{\infty} \la D^2 \phi(x) \sigma^j(x), \sigma^j(x) \ra = \tr \big( D^2 \phi(x) \Sigma(x) \Sigma(x)^* \big),
\end{align*}
showing that these two series are convergent. Recalling that $\Phi : H \to H_0$ is given by \eqref{Phi-in-proof}, by Proposition \ref{prop-Leibniz} for each $v \in H$ we have
\begin{align*}
D \Phi(x)v &= B( D \sigma^*(x) v, u ) + B( \sigma^*(x), D^2 \phi(x)v )
\\ &= (D \sigma^*(x) v) u + \sigma^*(x) D^2 \phi(x) v,
\end{align*}
where, according to Remark \ref{remark-derivatives} the second order derivative is considered as a self-adjoint operator $D^2 \phi(x) \in L(H)$. Hence, for all $w \in H_0$ we have
\begin{align*}
\gamma_0 w = D \Phi(x) \sigma(x) w = (D \sigma^*(x) \sigma(x) w) u + \sigma^*(x) D^2 \phi(x) \sigma(x) w.
\end{align*}
Note that $\sigma^j = \Psi_{f_j} \circ \sigma$ for all $j \in \bbn$, where for any $w \in H_0$ the continuous linear operator $\Psi_w : L_2(H_0,H) \to H$ is given by $\Psi_w(T) := Tw$. Thus, by Proposition \ref{prop-smooth-linear} we obtain
\begin{align}\label{sigma-j-diff-in-proof}
D \sigma^j(x)v = (D \sigma(x) v)f_j, \quad v \in H \text{ and } j \in \bbn.
\end{align}
Recalling that \eqref{isom-isom-HS} is a linear isometry, by Proposition \ref{prop-smooth-linear} and \eqref{sigma-j-diff-in-proof} we have
\begin{align*}
&\sum_{j=1}^{\infty} \la ( D \sigma^*(x) \sigma(x) f_j ) u, f_j \ra_{H_0} = \sum_{j=1}^{\infty} \la ( D \sigma^*(x) \sigma^j(x) ) u, f_j \ra_{H_0}
\\ &= \sum_{j=1}^{\infty} \la ( D \sigma(x) \sigma^j(x) )^* u, f_j \ra_{H_0} = \sum_{j=1}^{\infty} \la u, (D \sigma(x) \sigma^j(x)) f_j \ra = \sum_{j=1}^{\infty} \la u, D \sigma^j(x) \sigma^j(x) \ra.
\end{align*}
Moreover, we obtain
\begin{align*}
&\sum_{j=1}^{\infty} \la \sigma^*(x) D^2 \phi(x) \sigma(x) f_j, f_j \ra_{H_0} = \sum_{j=1}^{\infty} \la \sigma^*(x) D^2 \phi(x) \sigma^j(x), f_j \ra_{H_0}
\\ &= \sum_{j=1}^{\infty} \la D^2 \phi(x) \sigma^j(x), \sigma(x) f_j \ra = \sum_{j=1}^{\infty} \la D^2 \phi(x) \sigma^j(x), \sigma^j(x) \ra.
\end{align*}
Combining the preceding findings, we see that the series \eqref{series-gamma-0} converges with limit given by \eqref{tr-series-expr}.
\end{proof}

\subsection{Necessity under smoothness on $C$: Theorem~\ref{thm-main}.}\label{S:proofnecessityC}

In this section we show that in the situation of Theorem~\ref{thm-main} stochastic invariance implies that the invariance conditions \eqref{cond-C-zero} and \eqref{cond-drift-N} are fulfilled. In addition to Assumptions \ref{ass-lin-growth}, \ref{ass-Sigma-self-adjoint}, we now also suppose that Assumptions \ref{ass-extension},  \ref{ass-closed-range} are in force. By spectral decomposition, there are functions $q_k : H \to H$ and $\mu_k : H \to \bbr$ for $k \in \bbn$ such that $\sum_{k=1}^{\infty} |\mu_k(y)| < \infty$ for all $y \in H$ and
\begin{align}\label{E:spectralC-y}
C(y) = \sum_{k=1}^{\infty} \mu_k(y) \langle \cdot,q_k(y) \rangle \, q_k(y), \quad y \in H.
\end{align}
Here, for any $y \in H$ the ordering of the eigenvalues is such that $|\mu_k(y)| \geq |\mu_{k+1}(y)|$ for all $k \in \bbn$, and for all $k < l$ with $|\mu_k(y)| = |\mu_l(y)|$ and $\mu_k(y) \neq \mu_l(y)$ we have $\mu_k(y) > 0$ and $\mu_l(y) < 0$. Furthermore, the family $(q_k(y))_{k \in \bbn}$ of eigenvectors is an orthonormal basis of $H$.

Consider the particular situation $y \in \cald$. Since $C(y) \in L_1^+(H)$ has closed range, there exists a finite number $r(y) \in \bbn_0$ such that $\mu_1(y) \geq \mu_2(y) \geq \ldots \geq \mu_{r(y)}(y) > 0$ and $\mu_k(y) = 0$ for all $k > r(y)$. Now, let us fix an arbitrary $x \in \cald$, and set $r := r(x)$. As a consequence of \eqref{E:spectralC-y}, we have the representation
\begin{align}\label{E:spectralC}
C(x) = \sum_{k=1}^r \mu_k(x) \langle \cdot,q_k(x) \rangle \, q_k(x)
\end{align}
with eigenvalues $\mu_1(x) \geq \mu_2(x) \geq \ldots \geq \mu_r(x) > 0$, and the family $(q_k(x))_{k=1,\ldots,r}$ of eigenvectors is an orthonormal system in $H$.

We will treat the case of distinct eigenvalues in Section \ref{sec-sub-nec-1}, and the general situation in Section \ref{sec-sub-nec-2}.

\subsubsection{The drift condition: The case of distinct eigenvalues}\label{sec-sub-nec-1}

In this section we consider the situation, where the dispersion operator has distinct eigenvalues. More precisely, let us fix an arbitrary $x \in \cald$, and suppose that the eigenvalues in the spectral decomposition \eqref{E:spectralC} are such that $\mu_1(x) > \mu_2(x) > \ldots > \mu_r(x) > 0$.

\begin{proposition}\label{prop:distincteigenoperator}
	Let $x \in \mathcal{D}$ be such that the spectral decomposition of $C(x)$ is given by \eqref{E:spectralC} with
	 $\mu_1(x) > \mu_2(x) > \ldots > \mu_r(x) > 0$. Then there exist an open neighborhood $N(x)$ of $x$ as well as functions $ q^x_1,\ldots,q^x_r: N(x) \to H$ and $ \mu^x_1,\ldots,\mu^x_r : N(x) \to \bbr_+$ of class $C^2$ such that
		\begin{enumerate}
			\item[(i)]
			$\mu_k^x(x)=\mu_k(x)$ and $q_k^x(x)=q_k(x)$ for every $k \leq r$,
			\item[(ii)] $\mu_1^x(y) > \ldots > \mu_r^x(y) > 0$ are simple isolated eigenvalues of $C(y)$ with eigenvectors $q_1^x(y),\ldots,q_r^x(y)$, which form an orthonormal system in $H$, for all $y \in N(x)$,
			\item[(iii)] we have $\mu_k(y) = \mu_k^x(y)$ for all $y \in N(x)$ and all $k=1,\ldots,r$,
			\item[(iv)] after changing signs of the function $q_k : N(x) \to H$ in \eqref{E:spectralC-y} at suitable points, if required, we have $q_k(y) = q_k^x(y)$ for all $y \in N(x)$ and all $k=1,\ldots,r$,
			\item[(v)]
			$ \sigma_x: N(x) \to L^0_2(H)$ defined by
			\begin{align}\label{sigma-x-eigenvalues}
			\sigma_x (y) := \sum_{k=1}^r \sqrt{\mu^x_k(y)} \langle \cdot, f_k \rangle_{H_0} \, q^x_k(y)  , \quad  y \in N(x),
			\end{align}
			 is of class $C^2$.
		\end{enumerate}
	Moreover, we have
	\begin{align}
	  ( \sigma_x(x)  Q^{1/2})^* u &= 0, \label{E:sigmabaru} \\
		 \sum_{j=1}^{\infty} \langle u, D \sigma_x^j(x) \sigma_x^j(x) \rangle &= \sum_{j=1}^{\infty} \langle u, D  {C}^j(x) {P^j_C(x)}   \rangle, \label{E:sigmabarc}
	\end{align}
	 for all $u \in \ker(C(x))$.
\end{proposition}

Before moving to the proof, we notice that the previous proposition allows one to freeze the vanishing eigenvalues of $C$ at the point $x$.
\begin{remark}\label{R:bsigma2C} Note that $(\sigma_x Q^{1/2})(\sigma_x  Q^{1/2})^*$ is not necessarily equal to $C$   since they do not necessarily have the same eigenvalues when $y \neq x$. Indeed, the  family $(\mu_k(y))_{k=1,\ldots,r}$ does not necessarily exhaust all the eigenvalues of $C(y)$, because $C(y)$ can have eigenvalues vanishing at the point $x$.   However, the eigenvalues match  at the  specific point $x$ so that
\begin{align}\label{C-as-composition}
(\sigma_x(x) Q^{1/2})(\sigma_x (x) Q^{1/2})^* = C(x).
\end{align}
In fact,  since $f_k = Q^{1/2}e_k$,
\begin{align}\label{E:barsigmaQ12}
 \sigma_x (x)Q^{1/2} = \sum_{k =1}^r \sqrt{\mu_k(x)} \langle  Q^{1/2} \cdot, Q^{1/2}e_k \rangle_{H_0} \, q_k(x)=\sum_{k =1}^r \sqrt{\mu_k(x)} \langle  \cdot, e_k \rangle \, q_k(x).
\end{align}
Thus,
\begin{equation}\label{E:barsigmaQ123}
\begin{aligned}
 (\sigma_x(x) Q^{1/2})(\sigma_x (x) Q^{1/2})^* &=  \sum_{k =1}^r \sqrt{\mu_k(x)} \langle (\sigma_x (x) Q^{1/2})^* \cdot, e_k \rangle \, q_k(x) \\
&=  \sum_{k =1}^r \sqrt{\mu_k(x)} \langle  \cdot,  \sigma_x (x) Q^{1/2} e_k \rangle \, q_k(x) \\
&=  \sum_{k =1}^r {\mu_k(x)} \langle  \cdot,  q_k(x) \rangle \, q_k(x)\\
&= C(x).
\end{aligned}
\end{equation}
\end{remark}

\begin{proof}[Proof of Proposition \ref{prop:distincteigenoperator}]
The statements (i)--(v) are consequences of Lemmas \ref{lemma-implicit-eigenvalues}, \ref{lemma-simple-eigenvalues-2} below and the fact that distinct eigenvalues of self-adjoint operators have orthogonal eigenvectors.

Now, let us define $C_x := \Sigma_x \Sigma_x^* : N(x) \to L_1(H)$, where $\Sigma_x := \sigma_x Q^{1/2} : N(x) \to L_2(H)$. By \eqref{C-as-composition} and Lemma \ref{lemma-kernels} we have
\begin{align*}
\ker (C(x)) = \ker (C_x(x)) = \ker ( \Sigma_x(x)^* ) = \ker ( \sigma_x(x)^* ),
\end{align*}
which in particular proves \eqref{E:sigmabaru}. Moreover, for every $y \in N(x)$ an analogous calculation as in \eqref{E:barsigmaQ12} leads to
$$ \Sigma_x (y) = \sum_{k=1}^r \sqrt{\mu_k^x(y)} \langle  Q^{1/2} \cdot, Q^{1/2}e_k \rangle_{H_0} \, q_k^x(y) = \sum_{k=1}^r \sqrt{\mu_k^x(y)} \langle  \cdot, e_k \rangle \, q_k^x(y), $$
and hence, an analogous calculation as in \eqref{E:barsigmaQ123} shows that
\begin{align*}
C_x(y)  = \sum_{k=1}^r \mu^x_k(y) \langle \cdot, q^x_k(y) \rangle \,q^x_k(y)   \quad \text{for all $y \in N(x)$.}
\end{align*}
Let $u \in \ker(C(x))$ be arbitrary. Then it follows from Proposition \ref{P:C=sigma2} (applied with $\sigma_x$ here) that
	\begin{align}\label{E:ubarsigmabarC}
	\sum_{j=1}^{\infty} \langle u, D \sigma_x^j(x) \sigma_x^j(x) \rangle = \sum_{j=1}^{\infty} \langle u, D  C_x^j(x) P^j_{C_x}(x) \rangle,
	\end{align}
	where the orthogonal projection on the range of $C_x(y)$ is given by
$$ P_{C_x}(y) = \sum_{k=1}^r \langle  \cdot,  q^x_k(y) \rangle q^x_k(y) \quad \text{for all $y \in N(x)$.}$$
Thus, $P_{C_x}$ is of class $C^2$ on $N(x)$, and by the previous identities and \eqref{E:spectralC-y} we have
\begin{align*}
C_x(y) = \sum_{k=1}^r \mu_k^x(y) \langle  \cdot, q_k^x(y) \rangle \, q_k^x(y) = C(y) P_{C_x}(y)
\end{align*}
for all $y \in N(x)$, showing that $C_x = C  P_{C_x} $. Hence, we have $C_x = B(C,P_{C_x})$, where
\begin{align*}
B : L_1(H) \times L(H) \to L_1(H)
\end{align*}
denotes the continuous bilinear operator $B(T,S) := TS$. Therefore, by Proposition \ref{prop-Leibniz} it follows that
\begin{align*}
D C_x(y) v &= B( D C(y) v, P_{C_x}(y) ) + B(C(y), P_{C_x}(y) v)
\\ &= DC(y) v P_{C_x}(y) + C(y) D P_{C_x}(y) v, \quad y \in N(x), \quad v \in H.
\end{align*}
	By Corollary \ref{cor-self-adjoint} and Lemma \ref{lemma-adjoint-isometry} the operator $DC(x)v$ is self-adjoint for each $v \in H$. Denoting by $Cu : H \to H$ the mapping $y \mapsto C(y)u$, by Proposition \ref{prop-smooth-linear} we have
	\begin{align*}
	D (Cu)(x)v = ( D C(x) v )u, \quad v \in H.
	\end{align*}
	Observing that $P_{C_x}(x) =P_C(x)$ and recalling that $u \in \ker (C(x))$, the previous two identities yield
	\begin{align*}
	(DC_x(x) P^j_{C_x}(x))^* u &= ( D C(x) P_C^j(x) P_C(x) )^* u + ( C(x) DP_{C_x}(x) P_C^j(x) )^* u
	\\ &= P_C(x) (DC(x) P_C^j(x)  )^*u + ( DP_{C_x}(x) P_C^j(x) )^* C(x) u
	\\ &= P_C(x) (DC(x) P_C^j(x)  ) u = P_C(x) D(Cu)(x) P_C(x) e_j.
	\end{align*}
Thus, by Lemma \ref{lemma-trace-commute} we obtain
\begin{align*}
&\sum_{j=1}^{\infty} \langle u, D  C_x^j(x) P_{C_x}^j(x)  \rangle = \sum_{j=1}^{\infty} \langle u, D  C_x (x) P_{C_x}^j(x) e_j  \rangle = \sum_{j=1}^{\infty}\langle (D  C_x (x) P_{C_x}^j(x) )^* u,  e_j  \rangle \\
&= \sum_{j=1}^{\infty} \langle P_C(x) D(Cu)(x) P_C(x) e_j ,  e_j  \rangle = \tr \big( P_C(x) D(Cu)(x) P_C(x) \big) \\
&= \tr \big( D(Cu)(x) P_C(x) P_C(x) \big) = \tr \big( D(Cu)(x) P_C(x) \big) \\
&=  \sum_{j=1}^{\infty}\langle D(Cu)(x) P^j_C(x)  ,  e_j  \rangle =  \sum_{j=1}^{\infty}\langle (DC(x)   P^j_C(x)) u,  e_j  \rangle \\
&=\sum_{j=1}^{\infty}\langle u, (DC(x)   P^j_C(x)) e_j  \rangle =\sum_{j=1}^{\infty}\langle u,  DC^j(x) P_C^j(x)  \rangle,
\end{align*}
	yielding  \eqref{E:sigmabarc} thanks to  \eqref{E:ubarsigmabarC}.
\end{proof}

\begin{lemma}\label{lemma-implicit-eigenvalues}
Let $x \in \cald$ be arbitrary, let $\mu(x)$ be a simple eigenvalue of $C(x)$, and let $q(x)\in H$ be a corresponding eigenvector with $\| q(x) \| = 1$. Then there exist a neighborhood $N(x) \subset H$ of $x$, and unique mappings $q^x : N(x) \to H$ and $\mu^x : N(x) \to \bbr$ of class $C^2$ with $q^x(x) = q(x)$ and $\mu^x(x) = \mu(x)$ such that
\begin{align*}
C(y) q^x(y) = \mu^x(y) q^x(y) \quad \text{and} \quad \| q^x(y) \| = 1 \quad \text{for all $y \in N(x)$.}
\end{align*}
\end{lemma}

\begin{proof}
The proof follows the ideas presented in \cite{R55} (see also \cite[Theorem (p.177)]{K19}). We define the function $f : H \times (H \times \bbr) \to H \times \bbr$ as
\begin{align*}
f(y,(q,\mu)) := \big( (C(y) - \mu) q, \| q \|^2 - 1 \big).
\end{align*}
Then $f$ is of class $C^2$ with
\begin{align*}
f \big( x,(q(x),\mu(x)) \big) = 0.
\end{align*}
The partial derivative $T := D_2 f(x,(q(x),\mu(x))) \in L(H \times \bbr)$ is given by
\begin{align*}
T(p,\lambda) = \big( (C(x) - \mu(x)) p - \lambda q(x), 2 \la q(x),p \ra \big).
\end{align*}
The linear operator $T$ is an isomorphism. Indeed, note that $H = H_1 \oplus H_2$, where $H_1 = \lin \{ q(x) \}$ and $H_2 = H_1^{\perp}$. Since $\mu(x)$ is a simple eigenvalue, the linear operator $C(x) - \mu(x) \in L(H_2)$ is invertible. Denoting by $P_i : H \to H_i$, $i=1,2$ the corresponding orthogonal projections, we have $P_1 y = \la y,q(x) \ra q(x)$ for all $y \in H$. Thus, a straightforward calculation shows that the inverse of $T$ is given by
\begin{align*}
T^{-1}(r,\kappa) := \Big( \frac{\kappa}{2} q(x) + (C(x) - \mu(x))^{-1} P_2 r, - \la q,r \ra \Big).
\end{align*}
Consequently, the Implicit Function Theorem (Theorem \ref{thm-implicit}) implies that there are a neighborhood $N(x) \subset H$ of $x$ and a unique map $(q^x,\mu^x) : N(x) \to H \times \bbr$ of class $C^2$ such that $q^x(x) = q(x)$, $\mu^x(x) = \mu(x)$ and
\begin{align*}
f \big( y,(q^x(y),\mu^x(y)) \big) = 0 \quad \text{for all $y \in N(x)$.}
\end{align*}
This concludes the proof.
\end{proof}

\begin{lemma}\label{lemma-simple-eigenvalues-2}
Let $x \in \cald$ be arbitrary. Suppose that the spectral decomposition of $C(x)$ is given by \eqref{E:spectralC} with $\mu_1(x) > \mu_2(x) > \ldots > \mu_r(x) > 0$, and let $N(x) \subset H$ be an open neighborhood of $x$. Furthermore, let $\mu_1^x,\ldots,\mu_r^x : N(x) \to \bbr$ and $q_1^x,\ldots,q_r^x : N(x) \to H$ be continuous mappings such that $\mu_k^x(x) = \mu_k(x)$ and $q_k^x(x) = q_k(x)$ for all $k=1,\ldots,r$, and $\mu_k^x(y)$ is an eigenvalue of $C(y)$ with eigenvector $q_k^x(y)$ such that $\| q_k^x(y) \| = 1$ for all $y \in N(x)$ and all $k=1,\ldots,r$. Then there exists $\delta > 0$ such that for all $y \in N(x)$ with $\| x-y \| < \delta$ the following statements are true:
\begin{enumerate}
\item $\mu_1^x(y), \ldots, \mu_r^x(y)$ are simple eigenvalues of $C(y)$, and we have
\begin{align}\label{eigenvalues-ordered}
\mu_1^x(y) > \mu_2^x(y) > \ldots > \mu_r^x(y) > 0.
\end{align}
\item We have $\mu_k(y) = \mu_k^x(y)$ for all $k=1,\ldots,r$.

\item We have $q_k(y) = q_k^x(y)$ or $q_k(y) = - q_k^x(y)$ for all $k=1,\ldots,r$.
\end{enumerate}
\end{lemma}

\begin{proof}
Let us define $\mu_k^x : N(x) \to \bbr$ as $\mu_k^x \equiv 0$ for all $k \geq r+1$. Moreover, we define $\epsilon > 0$ as
\begin{align*}
\epsilon := \min \{ \mu_{k}^x(x) - \mu_{k+1}^x(x) : k = 1,\ldots,r \}.
\end{align*}
Then we have
\begin{align}\label{dist-EW-0}
\mu_k^x(x) - \mu_{k+1}^x(x) \geq \epsilon, \quad k=1,\ldots,r,
\end{align}
and in particular $\mu_r^x(x) \geq \epsilon$. By the continuity of $C$ and $\mu_1^x,\ldots,\mu_r^x$ there exists $\delta > 0$ such that for all $y \in N(x)$ with $\| x-y \| < \delta$ we have
\begin{align}\label{EW-est-1}
\| C(x) - C(y) \|_{L_1(H)} &< \frac{\epsilon}{8},
\\ \label{EW-est-2} \sum_{k=1}^r | \mu_k^x(x) - \mu_k^x(y) | &< \frac{\epsilon}{8}.
\end{align}
Now, let $y \in N(x)$ with $\| x-y \| < \delta$ be arbitrary. By \eqref{dist-EW-0} and \eqref{EW-est-2} we have
\begin{align}\label{dist-EW-1}
\mu_k^x(y) - \mu_{k+1}^x(y) \geq \frac{3\epsilon}{4} > \frac{\epsilon}{2}, \quad k=1,\ldots,r.
\end{align}
In particular, we obtain \eqref{eigenvalues-ordered}, and the eigenvalues $\mu_1^x(y), \ldots, \mu_r^x(y)$ of $C(y)$ are simple. Moreover, by \eqref{dist-EW-1} we have $\mu_r^x(y) > \frac{\epsilon}{2}$, and hence
\begin{align}\label{EW-est-2b}
\mu_1^x(y) > \mu_2^x(y) > \ldots > \mu_r^x(y) > \frac{\epsilon}{2}.
\end{align}
Furthermore, by Proposition \ref{prop-eigenvalue-map} we have
\begin{align*}
\sum_{k=1}^{\infty} | \mu_k(x) - \mu_k(y) | \leq \| C(x) - C(y) \|_{L_1(H)} < \frac{\epsilon}{8}.
\end{align*}
Together with \eqref{EW-est-2} we obtain
\begin{align*}
&\sum_{k=1}^r | \mu_k^x(y) - \mu_k(y) | + \sum_{k=r+1}^{\infty} |\mu_k(y)| = \sum_{k=1}^{\infty} | \mu_k^x(y) - \mu_k(y) |
\\ &\leq \sum_{k=1}^r | \mu_k^x(y) - \mu_k^x(x) | + \sum_{k=1}^{\infty} | \mu_k(x) - \mu_k(y) | < \frac{\epsilon}{4}.
\end{align*}
In particular, we have
\begin{align}\label{EW-est-3}
\sum_{k=1}^r | \mu_k^x(y) - \mu_k(y) |  < \frac{\epsilon}{4}
\end{align}
as well as
\begin{align*}
|\mu_k(y)| < \frac{\epsilon}{4} \quad \text{for all $k \geq r+1$.}
\end{align*}
Thus, in view of \eqref{EW-est-2b} we obtain
\begin{align}\label{EW-est-4}
\mu_r^x(y) > |\mu_k(y)| \quad \text{for all $k \geq r+1$.}
\end{align}
Therefore, by \eqref{dist-EW-1}, \eqref{EW-est-3} and \eqref{EW-est-4} we deduce that
\begin{align*}
\mu_1(y) > \mu_2(y) > \ldots > \mu_r(y) > |\mu_{r+1}(y)|.
\end{align*}
Recalling the ordering of the eigenvalues $(\mu_k(y))_{k \in \bbn}$ in \eqref{E:spectralC-y}, this proves that $\mu_k(y) = \mu_k^x(y)$ for all $k=1,\ldots,r$. Now, since $\mu_k(y)$ is a simple eigenvalue of $C(y)$, and $q_k(y)$ and $q_k^x(y)$ are normalized eigenvectors, we have $q_k(y) = q_k^x(y)$ or $q_k(y) = - q_k^x(y)$ for all $k=1,\ldots,r$.
\end{proof}

\begin{proposition}\label{propdistincteigen}
	Suppose that Assumptions \ref{ass-lin-growth}, \ref{ass-extension}, \ref{ass-Sigma-self-adjoint}, \ref{ass-closed-range} are in force. Assume that $\mathcal{D}$ is stochastically invariant. Let $x \in \cald$ be such that the eigenvalues in the spectral decomposition \eqref{E:spectralC} satisfy $\mu_1(x) > \mu_2(x) > \ldots > \mu_r(x) > 0$.
	Then, \eqref{cond-C-zero} and \eqref{cond-drift-N} hold for all $u \in \mathcal{N}^{1,\pr}_{\mathcal{D}}(x)$.
\end{proposition}

\begin{proof}  
	Let $x \in \cald$ be arbitrary and let $(X,W)$ denote a  {weak solution} starting at $X_0=x$ such that  $X_t \in \mathcal{D}$ for all $t \geq 0 $. We also fix an arbitrary $u \in \mathcal{N}^{1,\pr}_{\mathcal{D}}(x)$.\\
	\textit{Step 1.} As in the proof of Proposition \ref{prop-nec-sigma-C2} we show that there exists a function $\phi \in C^{\infty}(H,\mathbb{R})$ such that $\displaystyle\max_{\mathcal{D}} \phi = \phi(x) = 0$ and $D \phi(x)= u$.\\
	\textit{Step 2.} Choosing an appropriate open and bounded neighborhood $N(x) \subset H$ of $x$ as in Proposition~\ref{prop:distincteigenoperator}, and a stopping time $\tau > 0$ such that $X^{\tau} \in N(x)$, by Proposition \ref{prop-modify-fct-bdd} we may assume that $\phi$ and $\sigma_x$ are of class $C_b^2$. Moreover, the stopped process $X^{\tau}$ is bounded, and we have $X^{\tau} \in \cald$.\\
	\textit{Step 3.} Performing the procedure described after Assumption \ref{ass-Sigma-self-adjoint}, we can restrict the study to the situation
	\begin{align}\label{sigma-eigenvalues}
	\sigma(y)= \sum_{k=1}^{\infty} \sqrt{\mu_k(y)}  \langle \cdot, f_k \rangle_{H_0} q_k (y), \quad y \in N(x),
	\end{align}
	since $\sigma(y)Q^{1/2}(\sigma(y)Q^{1/2})^* = C(y) $, meaning that the law of the diffusion is unchanged; see Lemma \ref{lemma-invariance-law}.  Since $\mathcal{D}$ is invariant under the diffusion $X$, we have $\phi(X_t) \leq \phi(x)$, for all $t\geq 0$.  By the above and It\^{o}'s formula (Theorem \ref{thm-Ito}) we obtain \eqref{Ito-ineqn-phi}.
	Recall that the sequence $(W^j)_{j \in \bbn}$ defined as
	\begin{align*}
	W^j := \frac{1}{\sqrt{\lambda_j}} \la W,e_j \ra
	\end{align*}
	is a sequence of independent real-valued standard Wiener processes; see \cite[Prop. 4.3]{Da_Prato}. Moreover, we have $W=\sum_{j=1}^{\infty} \sqrt{\lambda_j} W^j e_j$. The two $H$-valued processes
	\begin{align*}
	\bar{W} := \sum_{k=1}^r \sqrt{\lambda_k} W^k e_k \quad \mbox{and} \quad \bar{W}^{\perp} := \sum_{k>r} \sqrt{\lambda_k} W^k e_k
	\end{align*}
	are independent trace class Wiener processes with covariance operators $\bar{Q} = Q P_V$ and $\bar{Q}^{\perp} = Q (\Id - P_V)$, where $V \subset H$ denotes the subspace $V := \lin \{ e_1,\ldots,e_r \}$, and $P_V$ the orthogonal projection on $V$. Moreover, we have the decomposition $W = \bar W + \bar W^{\perp}$. Let $\sigma_x : N(x) \to L_2^0(H)$ be given as in \eqref{sigma-x-eigenvalues}, and recall that by Proposition~\ref{prop:distincteigenoperator} we have $\mu_k(y) = \mu_k^x(y)$ and $q_k(y) = q_k^x(y)$ for all $k=1,\ldots,r$ and all $y \in N(x)$. Thus, noting \eqref{sigma-eigenvalues} and \eqref{sigma-x-eigenvalues} we have
	\begin{align*}
	\int_{0}^{t} \sigma(X_{s}^{\tau})d\bar{W}_s = \sum_{k=1}^{r} \int_{0}^{t} \sigma(X_{s}^{\tau}) f_k d W_s^k = \sum_{k=1}^{r} \int_{0}^{t} \sigma_x(X_{s}^{\tau}) f_k d W_s^k = \int_{0}^{t} \sigma_x(X_{s}^{\tau})d \bar{W}_s.
	\end{align*}
	Taking also into account Lemma \ref{lemma-associativity} below, the above inequality \eqref{Ito-ineqn-phi} can be written in the form
	\begin{align*}
	0 &\ge \int_{0}^{t } \mathcal{L}\phi(X_{s}^{\tau})ds + \int_{0}^{t} D\phi(X_{s}^{\tau})  \sigma(X_{s}^{\tau})d \bar{W}_s + \int_{0}^{t} D\phi(X_s^{\tau}) \sigma(X_{s}^{\tau})d\bar W^{\perp}_s
	\\ &= \int_{0}^{t } \mathcal{L}\phi(X_{s}^{\tau})ds + \int_{0}^{t} D\phi(X_{s}^{\tau})  \sigma_x(X_{s}^{\tau})d \bar{W}_s + \int_{0}^{t} D\phi(X_s^{\tau}) \sigma(X_{s}^{\tau})d\bar W^{\perp}_s.
	\end{align*}
	Let $(\mathcal{F}_s^{\bar{W}})_{s\ge 0}$ be the completed filtration   generated by $\bar W$. Noting that the subspace $V$ can be expressed as $V = \lin \{ f_1,\ldots,f_r \}$, by \eqref{sigma-x-eigenvalues} we have
	\begin{align}\label{sigma-projection}
	\sigma_x(y) P_V = \sigma_x(y) \quad \text{for all $y \in N(x)$.}
	\end{align}
	Thus, taking the conditional expectation $\bbe_{\calf_s^{\bar{W}}}$, by  Lemma~\ref{L:condexphilbert} below we obtain
	\begin{align*}
	0&\ge \int_{0}^{t} \mathbb{E}_{\mathcal{F}_s^{\bar{W}}} [\mathcal{L}\phi(X_{s}^{\tau})]ds + \int_{0}^{t} \mathbb{E}_{\mathcal{F}_s^{\bar{W}}} [D \phi(X_{s}^{\tau}) \sigma_x(X_{s}^{\tau})]d \bar{W}_s
	\\ &= \int_{0}^{t} \mathbb{E}_{\mathcal{F}_s^{\bar{W}}} [\mathcal{L}\phi(X_{s}^{\tau})]ds + \int_{0}^{t} \mathbb{E}_{\mathcal{F}_s^{\bar{W}}} [D \phi(X_{s}^{\tau}) \sigma_x(X_{s}^{\tau})] P_V d W_s
	\\ &= \int_{0}^{t} \mathbb{E}_{\mathcal{F}_s^{\bar{W}}} [\mathcal{L}\phi(X_{s}^{\tau})]ds + \int_{0}^{t} \mathbb{E}_{\mathcal{F}_s^{\bar{W}}} [D \phi(X_{s}^{\tau}) \sigma_x(X_{s}^{\tau}) P_V ] d W_s
	\\ &= \int_{0}^{t} \mathbb{E}_{\mathcal{F}_s^{\bar{W}}} [\mathcal{L}\phi(X_{s}^{\tau})]ds + \int_{0}^{t} \mathbb{E}_{\mathcal{F}_s^{\bar{W}}} [D \phi(X_{s}^{\tau}) \sigma_x(X_{s}^{\tau})] d W_s.
	\end{align*}
	We introduce $\Phi_x : H \to L_2(H_0,\bbr)$ as
	\begin{align*}
	\Phi_x(y) := D \phi(y) \sigma_x(y), \quad y \in H.
	\end{align*}
	As in the proof of Proposition \ref{prop-nec-sigma-C2} we may regard $\Phi_x$ as a mapping $\Phi_x : H \to H_0$, and we have $\Phi_x \in C_{b}^2(H,H_0)$. Then we can write the previous inequality as
	\begin{align*}
	0 \ge \int_{0}^{t } \mathbb{E}_{\mathcal{F}_s^{\bar{W}}} [\mathcal{L}\phi(X_{s}^{\tau})]ds + \int_{0}^{t} \mathbb{E}_{\mathcal{F}_s^{\bar{W}}} [\Phi_x(X_{s}^{\tau})] d W_s.
	\end{align*}
	Another application of It\^{o}'s formula (Theorem \ref{thm-Ito}) to $\Phi_x(X^{\tau})$ gives us
	\begin{align*}
	0 &\ge \int_{0}^{t } \mathbb{E}_{\mathcal{F}_s^{\bar{W}}} [\mathcal{L}\phi(X_{s}^{\tau})]ds
	\\ &\quad + \int_{0}^{t} \mathbb{E}_{\mathcal{F}_s^{\bar{W}}} \bigg[ \Phi_x(x) + \int_0^s \call \Phi_x(X_r^{\tau})dr + \int_0^s D \Phi_x(X_r^{\tau}) \sigma_x(X_r^{\tau}) dW_r \bigg] dW_s.
	\end{align*}
	Together with Lemma~\ref{L:condexphilbert} below and identity \eqref{sigma-projection} this yields \eqref{eqintegraledoublehilbert},
	where the constant $\alpha \in H_0$, the $H_0$-valued predictable process $\beta$, the $L_2(H_0)$-valued predictable process $\gamma$, and the $\bbr$-valued predictable process $\theta$ are given by
	\begin{align*}
	\alpha &:= \Phi_x(x), \quad \quad
	\beta_r := \mathbb{E}_{\mathcal{F}_{r}^{\bar{W}}} [ \mathcal L \Phi_x (X_r^{\tau}) ], \quad r \geq 0,
	\\ \gamma_r &:= \mathbb{E}_{\mathcal{F}_{r}^{\bar{W}}} [ \Psi_x(X_r^{\tau}) ], \quad r \geq 0, \quad \quad
	\theta_s := \mathbb{E}_{\mathcal{F}_{s}^{\bar{W}}} [ \mathcal L \phi (X_s^{\tau}) ], \quad s \geq 0,
	\end{align*}
	and where $\Psi_x : H \to L_2(H_0)$ is defined as $\Psi_x(y) := D \Phi_x(y) \sigma_x(y)$ for each $y \in H$.\\
	\textit{Step 4.}   We now check that we can apply   Lemma \ref{lemmadoubleintstohilbert}.  Indeed, given $T >0$, by Lemma \ref{lemma-cond-exp} below we have $\theta_s= \bbe_{\calf^{\bar{W}}_T}\left[ \call \phi (X_s^{\tau}) \right]$ for all $s \leq T$, showing that $\theta$ is continuous at zero. Taking into account the linear growth condition \eqref{linear-growth}, the fact that $\Phi_x$ is of class $C_b^2$, and the boundedness of $X^{\tau}$, the process $\beta$ is bounded by Lemma \ref{lemma-generator-growth}. Note that
\begin{align*}
\Psi_x(y) = B(D \Phi_x(y), \sigma_x(y)), \quad y \in H,
\end{align*}
where $B : L(H,H_0) \times L_2(H_0,H) \to L_2(H_0)$ denotes the continuous bilinear operator $B(T,S) = TS$. Since $D \Phi_x : H \to L(H,H_0)$ is of class $C_b^1$ and $\sigma_x : H \to L_2(H_0,H)$ is of class $C_b^2$, by Proposition \ref{prop-Leibniz} we obtain $\Psi_x \in C_{b}^1(H,L_2(H_0))$. Therefore, we have $\int_0^t \| \gamma_s \|_{L_2(H_0)}^2 ds < \infty$, for all $t \geq 0$. Moreover, by Proposition \ref{prop-Cb1-Lipschitz} the mapping $\Psi_x : H \to L_2(H_0)$ is Lipschitz continuous, and hence, by the triangle inequality and the H\"{o}lder inequality for conditional expectations as well as Lemma \ref{lemma-short-time-double-dt} (applied with $\Psi_x$ here) we obtain
	\begin{align*}
	&\int_0^t \int_0^s \bbe \big[ \| \gamma_r - \gamma_0 \|_{L_2(H_0)}^2 \big] dr ds
	\\ &= \int_0^t \int_0^s \bbe \big[ \| \bbe_{\mathcal{F}_{r}^{\bar{W}}} [ \Psi_x(X_r^{\tau}) - \Psi_x(x) ] \|_{L_2(H_0)}^2 \big] dr ds
	\\ &\leq \int_0^t \int_0^s \bbe \big[ \bbe_{\mathcal{F}_{r}^{\bar{W}}} [ \| \Psi_x(X_r^{\tau}) - \Psi_x(x) \|_{L_2(H_0)} ]^2 \big] dr ds
	\\ &\leq \int_0^t \int_0^s \bbe \big[ \bbe_{\mathcal{F}_{r}^{\bar{W}}} [ \| \Psi_x(X_r^{\tau}) - \Psi_x(x) \|_{L_2(H_0)}^2 ] \big] dr ds
	\\ &= \int_0^t \int_0^s \bbe \big[ \| \Psi_x(X_r^{\tau}) - \Psi_x(x) \|_{L_2(H_0)}^2 \big] dr ds = O(t^3).
	\end{align*}
    \textit{Step 5.} We therefore can apply Lemma~\ref{lemmadoubleintstohilbert} to~\eqref{eqintegraledoublehilbert}  to deduce  that $\alpha=0$. Since $\sigma(x) = \sigma_x(x)$, this implies $D \phi(x)\sigma(x) = 0$, and thus $\sigma(x)^* u = 0$. By Remark \ref{rem-kernels} we obtain $C(x)u = 0$, showing \eqref{cond-C-zero}.\\
	\textit{Step 6.} Note that $\gamma_0 = \Psi_x(x) = D \Phi_x(x) \sigma_x(x)$ and $\Phi_x(y) = D \phi(y) \sigma_x(y)$ for all $y \in H$.
	Since $\sigma_x(x)^* u = 0$, we can apply Lemma \ref{lemma-trace-decomp} and deduce that the series \eqref{series-gamma-0} converges with limit
\begin{align}\label{tr-series-expr-2}
\tr(\gamma_0) = \sum_{j=1}^{\infty} \la D \phi(x), D \sigma_x^j(x) \sigma_x^j(x) \ra + \sum_{j=1}^{\infty} \la D^2 \phi(x) \sigma_x^j(x), \sigma_x^j(x) \ra.
\end{align}
	We therefore can apply  Lemma~\ref{lemmadoubleintstohilbert} again  to~\eqref{eqintegraledoublehilbert} to deduce  that $\theta_0-\frac{1}{2}\Tr(\gamma_0  )  \leq 0$. Taking into account Lemma \ref{lemma-generator-series}, this inequality combined with \eqref{tr-series-expr-2} and the fact that $\sigma (x) = \sigma_x(x)$ shows that
	\begin{align*}
	0 &\ge  \call\phi(x)-\frac12 \Tr(\gamma_0) \\
	&= \la D\phi(x), b(x) \ra + \frac{1}{2} \sum_{j=1}^{\infty} \la D^2 \phi(x)  \sigma^j(x),  \sigma^j(x) \ra \\
	&\quad - \frac 12 \bigg( \sum_{j=1}^{\infty} \la D \phi(x), D \sigma_x^j(x) \sigma_x^j(x) \ra + \sum_{j=1}^{\infty} \la D^2 \phi(x) \sigma_x^j(x), \sigma_x^j(x) \ra \bigg) \\
	&= \la D\phi(x), b(x) \ra - \frac{1}{2} \sum_{j=1}^{\infty} \la D\phi(x),  D\sigma_x^j(x) \sigma_x^j(x) \ra.
	\end{align*}
	Recalling that $D \phi(x) = u \in \ker(C(x))$, this is equivalent to \eqref{cond-drift-N} thanks to \eqref{E:sigmabarc} and Remark \ref{rem-drift-cond}.
	\end{proof}

\begin{lemma}\label{lemma-associativity}
Let $G,F$ be separable Hilbert spaces, let $\Phi$ be an $L_2(H_0,G)$-valued predictable bounded process, and let $A$ be an $L(G,F)$-valued predictable bounded process. We define the $G$-valued square-integrable martingale $Y$ as $Y_t := \int_0^t \Phi_s dW_s$, $t \geq 0$. Then we have
\begin{align*}
\int_0^t A_s \Phi_s dW_s = \int_0^t A_s dY_s, \quad t \geq 0.
\end{align*}
\end{lemma}

\begin{proof}
We only sketch the proof and provide the result for elementary processes. Suppose that $\Phi = \varphi \bbI_{(u,v]}$ and $A = a \bbI_{(u,v]}$ with $u < v$, an $L(H,G)$-valued $\calf_u$-measurable random variable $\varphi$, and an $L(G,F)$-valued $\calf_u$-measurable random variable $a$. Then we have
\begin{align*}
Y_t = \int_0^t \Phi_s dW_s = \varphi ( W_{v \wedge t} - W_{u \wedge t} ),
\end{align*}
and hence
\begin{align*}
Y_{v \wedge t} - Y_{u \wedge t} = \varphi(W_{v \wedge t} - W_{u \wedge t}) - \varphi(W_{u \wedge t} - W_{u \wedge t}) = \varphi(W_{v \wedge t} - W_{u \wedge t}).
\end{align*}
Moreover, we have $A \Phi = a \varphi \bbI_{(u,v]}$, and thus
\begin{align*}
\int_0^t A_s \Phi_s dW_s = a \varphi ( W_{v \wedge t} - W_{u \wedge t} ) = a ( Y_{v \wedge t} - Y_{u \wedge t} ) = \int_0^t A_s dY_s,
\end{align*}
which provides the desired identity.
\end{proof}

The following elementary lemma extends \cite[Lemma 5.4]{xio} to the infinite dimensional setting.

\begin{lemma}\label{L:condexphilbert}
Let $\bar{W}$ and $\bar{W}^{\perp}$ be two independent $H$-valued trace class Wiener processes with covariance operators $\bar{Q}, \bar{Q}^{\perp} \in L_1^+(H)$, and let $G$ be another separable Hilbert space. Moreover, we denote by $(\calf_t^{\bar{W}})_{t \geq 0}$ the completed filtration generated by $\bar{W}$.
\begin{enumerate}
\item For any $L_2(\bar{Q}^{1/2}(H),G)$-valued, $(\calf_t)_{t \geq 0}$-predictable, integrable process $\Phi$ we have
\begin{align*}
\mathbb{E}_{\mathcal{F}_t^{\bar{W}}} \left[ \int_{0}^{t} \Phi_s d \bar{W}_s \right] = \int_{0}^{t}  \mathbb{E}_{\mathcal{F}_s^{\bar{W}}}\left[ \Phi_s \right]d \bar{W}_s, \quad t \geq 0.
\end{align*}
\item For any $L_2((\bar{Q}^{\perp})^{1/2}(H),G)$-valued, $(\calf_t)_{t \geq 0}$-predictable, integrable process $\Phi$ we have
\begin{align*}
\mathbb{E}_{\mathcal{F}_t^{\bar{W}}}\left[ \int_{0}^{t} \Phi_s d\bar{W}^{\perp}_s \right] = 0, \quad t \geq 0.
\end{align*}
\item Moreover, it holds similarly for any $G$-valued, $(\calf_t)_{t \geq 0}$-predictable, integrable process $\theta$ that
	\begin{align*}
		\mathbb{E}_{\mathcal{F}_t^{\bar{W}}}\left[ \int_{0}^{t} \theta_s ds \right] = \int_{0}^{t}  \mathbb{E}_{\mathcal{F}_s^{\bar{W}}}\left[ \theta_s \right]ds, \quad t \geq 0.
	\end{align*}
\end{enumerate}
\end{lemma}

\begin{proof}
We only sketch the proof and provide the result for elementary processes. Suppose that $\Phi = \varphi \bbI_{(u,v]}$ with $u < v$, and an $L(H,G)$-valued $\calf_u$-measurable random variable $\varphi$. Then we have
\begin{align*}
	\mathbb{E}_{\mathcal{F}_t^{\bar{W}}}\left[ \Phi_s \right] = \mathbb{E}_{\mathcal{F}_t^{\bar{W}}}[\varphi] \bbI_{(u,v]}(s),
	\end{align*}
	and hence
	\begin{align*}
	\mathbb{E}_{\mathcal{F}_t^{\bar{W}}}\left[ \int_{0}^{t} \Phi_s d \bar{W}_s \right] &= \mathbb{E}_{\mathcal{F}_t^{\bar{W}}}\left[ \varphi (\bar{W}_{v \wedge t} - \bar{W}_{u \wedge t}) \right] = \mathbb{E}_{\mathcal{F}_t^{\bar{W}}}\left[ \varphi \right] (\bar{W}_{v \wedge t} - \bar{W}_{u \wedge t})
	\\ &= \int_{0}^{t}  \mathbb{E}_{\mathcal{F}_t^{\bar{W}}}\left[ \Phi_s \right]d \bar{W}_s = \int_{0}^{t}  \mathbb{E}_{\mathcal{F}_s^{\bar{W}}}\left[ \Phi_s \right]d \bar{W}_s,
	\end{align*}
	where in the last step we have used Lemma \ref{lemma-cond-exp} below. Moreover, by the independence of $\bar{W}$ and $\bar{W}^{\perp}$ we obtain
	\begin{align*}
	\mathbb{E}_{\mathcal{F}_t^{\bar{W}}}\left[ \int_{0}^{t} \Phi_s d\bar{W}^{\perp}_s \right] &= \mathbb{E}_{\mathcal{F}_t^{\bar{W}}}\left[ \varphi (\bar{W}_{v \wedge t}^{\perp} - \bar{W}_{u \wedge t}^{\perp}) \right]
	\\ &= \mathbb{E}_{\mathcal{F}_t^{\bar{W}}}\left[ \mathbb{E}_{\mathcal{F}_t^{\bar{W}} \vee \sigma(\varphi)}\left[ \varphi (\bar{W}_{v \wedge t}^{\perp} - \bar{W}_{u \wedge t}^{\perp}) \right] \right]
	\\ &= \mathbb{E}_{\mathcal{F}_t^{\bar{W}}}\left[ \varphi \, \mathbb{E}_{\mathcal{F}_t^{\bar{W}} \vee \sigma(\varphi)}\left[ \bar{W}_{v \wedge t}^{\perp} - \bar{W}_{u \wedge t}^{\perp} \right] \right] = 0.
	\end{align*}
	Now, suppose that $\theta = \vartheta \bbI_{(u,v]}$ with $u < v$, and an $G$-valued $\calf_u$-measurable random variable $\vartheta$. Then we have
	\begin{align*}
	\mathbb{E}_{\mathcal{F}_t^{\bar{W}}}\left[ \theta_s \right] = \mathbb{E}_{\mathcal{F}_t^{\bar{W}}}[\vartheta] \bbI_{(u,v]}(s),
	\end{align*}
	and hence
	\begin{align*}
		\mathbb{E}_{\mathcal{F}_t^{\bar{W}}}\left[ \int_{0}^{t} \theta_s ds \right] &= \mathbb{E}_{\mathcal{F}_t^{\bar{W}}} [ \vartheta ( u \wedge t - u \wedge s ) ]
		\\ &= \mathbb{E}_{\mathcal{F}_t^{\bar{W}}} [ \vartheta ] \, ( u \wedge t - u \wedge s ) = \int_{0}^{t}  \mathbb{E}_{\mathcal{F}_s^{\bar{W}}}\left[ \theta_s \right]ds,
	\end{align*}
	completing the proof.
\end{proof}

\begin{lemma}\label{lemma-cond-exp}
Let $E$ be a separable Banach space, and let $\gamma$ be an $E$-valued $(\calf_t)_{t \geq 0}$-adapted, integrable process. Furthermore, let $\bar{W}$ be a trace class Wiener process. Then we have
\begin{align*}
\mathbb{E}_{\mathcal{F}_s^{\bar{W}}}[\gamma_s] = \mathbb{E}_{\mathcal{F}_t^{\bar{W}}}[\gamma_s] \quad \text{for all $s \leq t$,}
\end{align*}
where $(\calf_t^{\bar{W}})_{t \geq 0}$ denotes the completed filtration generated by $\bar{W}$.
\end{lemma}

\begin{proof}
Note that the filtration $(\calg_t^{\bar{W}})_{t \geq 0}$ generated by $\bar{W}$ is given by
\begin{align*}
\calg_t^{\bar{W}} = \sigma( \bar{W}_{t_1} - \bar{W}_{t_0}, \ldots, \bar{W}_{t_n} - \bar{W}_{t_{n-1}} : n \in \bbn, 0 = t_0 < t_1 < \ldots < t_n = t ).
\end{align*}
Setting $Y_s := \bbe_{\calf_s^{\bar{W}}} [ \gamma_s ]$, we claim that $\bbe_{\calf_t^{\bar{W}}}[ \gamma_s ] = Y_s$. Of course, the random variable $Y_s$ is $\calf_t^{\bar{W}}$-measurable. In order to verify the test equation, it suffices to consider sets $A \in \calf_t^{\bar{W}}$ of the form $A = B \cap C$ with $B \in \calg_s^{\bar{W}}$ and
\begin{align*}
C \in \sigma( \bar{W}_{t_1} - \bar{W}_{t_0}, \ldots, \bar{W}_{t_n} - \bar{W}_{t_{n-1}} : n \in \bbn, s = t_0 < t_1 < \ldots < t_n = t ).
\end{align*}
Since $W$ is an $(\calf_t)_{t \geq 0}$-Wiener process, it follows that $\calf_s$ and $C$ are independent, and we obtain
\begin{align*}
\bbe[ Y_s \bbI_A ] &= \bbe \big[ \bbI_{B \cap C} \bbe_{\calf_s^{\bar{W}}}[ \gamma_s ] \big] = \bbe \big[ \bbI_C \cdot \bbe_{\calf_s^{\bar{W}}}[ \bbI_B \gamma_s ] \big]
\\ &= \bbe[\bbI_C] \cdot \bbe \big[ \bbe_{\calf_s^{\bar{W}}}[ \bbI_B \gamma_s ] \big] = \bbe[\bbI_C] \cdot \bbe[ \bbI_B \gamma_s ] = \bbe[ \gamma_s \bbI_A ],
\end{align*}
finishing the proof.
\end{proof}

\subsubsection{The drift condition: The general case}\label{sec-sub-nec-2}

We now treat the general case, where no additional conditions on the eigenvalues $\mu_1(x) \geq \mu_2(x) \geq \ldots \geq \mu_r(x) > 0$ in the spectral decomposition \eqref{E:spectralC} are imposed. For this purpose, we prepare some auxiliary results. Let $A \in L(H)$ be a linear isomorphism. We define the closed subset $\cald^A := A \cald$. The new mapping $b_A : H \to H$ defined as
\begin{align}\label{b-A-def}
b_A(x) := A b (A^{-1} x), \quad x \in H
\end{align}
is continuous and satisfies the linear growth condition, due to Assumption \ref{ass-lin-growth}. Consider the mapping $C_A : H \to L_1(H)$ defined as
\begin{align}\label{C-A-def}
C_A(x) := A C(A^{-1} x) A^*, \quad x \in H.
\end{align}
As an immediate consequence of Lemma \ref{lemma-self-adjoint-transfer} we obtain:

\begin{lemma}\label{lemma-C-A-self-adjoint}
The linear operator $C_A(x)$ is self-adjoint and nonnegative definite for each $x \in \cald^A$.
\end{lemma}

Furthermore, by \eqref{C-lin-growth} there is a constant $N > 0$ such that
\begin{align*}
\| C_A(x) \|_{L_1(H)}^{1/2} \leq N ( 1 + \| x \| ) \quad \text{for all $x \in H$,}
\end{align*}
and $\ran(C_A(x))$ is finite dimensional for each $x \in \cald$, due to Assumption \ref{ass-closed-range}.

\begin{lemma}\label{lemma-C-eps-diff}
The mapping $C_A$ is of class $C^2$, and we have
\begin{align*}
D C_A(x) v = A ( DC(A^{-1} x) A^{-1} v ) A^*, \quad x,v \in H.
\end{align*}
\end{lemma}

\begin{proof}
We can express the mapping $C_A$ as $C_A = \Psi \circ \Phi$, where $\Phi : H \to L_1(H)$ is given by $\Phi = C \circ A^{-1}$, and where $\Psi : L_1(H) \to L_1(H)$ is the continuous linear operator given by $\Psi(T) = A T A^*$. Therefore, by Propositions \ref{prop-smooth-linear} and \ref{prop-smooth-linear-two} the mapping $C_A$ is of class $C^2$, and we have
\begin{align*}
D C_A(x) v = D(\Psi \circ \Phi)(x)v = \Psi( D \Phi(x) v ) = A ( D \Phi(x) v ) A^*
\end{align*}
as well as
\begin{align*}
D \Phi(x) v = D(C \circ A^{-1})(x)v = D C (A^{-1} x) A^{-1} v.
\end{align*}
This completes the proof.
\end{proof}

Now, we define $\bar{\sigma}_A : H \to L_2^0(H)$ as
\begin{align}\label{sigma-A-def}
\bar{\sigma}_A(x) := \bar{\Sigma}_A(x) Q^{-1/2}, \quad x \in H,
\end{align}
where $\bar{\Sigma}_A : H \to L_2^+(H)$ is given by
\begin{align*}
\bar{\Sigma}_A(x) := |C_A(x)|^{1/2}, \quad x \in H.
\end{align*}
Then we have $\bar{\Sigma}_A(x) = \bar{\sigma}_A(x) Q^{1/2}$ for all $x \in H$, and in view of Lemma \ref{lemma-C-A-self-adjoint} we have $C_A(x) = \bar{\Sigma}_A(x)^2$ for all $x \in \cald^A$. Furthermore, the linear operator $\bar{\Sigma}_A(x)$ is self-adjoint and nonnegative definite for each $x \in \cald^A$. Moreover, Lemma \ref{lemma-construct-sigma} ensures that $\bar{\sigma}_A$ is continuous and satisfies the linear growth condition. Consequently, Assumptions \ref{ass-lin-growth}, \ref{ass-extension}, \ref{ass-Sigma-self-adjoint}, \ref{ass-closed-range} are also fulfilled for the new diffusion
\begin{align}\label{SDE-Y-bar}
d \tilde{X}_t = b_A(\tilde{X}_t)dt + \bar{\sigma}_A(\tilde{X}_t) dW_t, \quad \tilde{X}_0 = \tilde{x}.
\end{align}

\begin{lemma}\label{lemma-inv-transfer}
Suppose that $\cald$ is stochastically invariant with respect to the diffusion \eqref{SDE}. Then $\cald^A$ is stochastically invariant with respect to the diffusion \eqref{SDE-Y-bar}.
\end{lemma}

\begin{proof}
Let $\tilde{x} \in \cald^A$ be arbitrary. We set $x := A^{-1} \tilde{x} \in \cald$. Then there exists a weak solution $X$ to \eqref{SDE}, starting at $X_0 = x$ such that $X_t \in \cald$ for all $t \geq 0$, almost surely. Now, we define the continuous mapping $\sigma_A : H \to L_2^0(H)$ as
\begin{align*}
\sigma_A(x) := A \sigma (A^{-1} x), \quad x \in H.
\end{align*}
Then the process $\tilde{X} := A X$ is a weak solution to the SDE
\begin{align}\label{SDE-Y}
d \tilde{X}_t = b_A(\tilde{X}_t)dt + \sigma_A(\tilde{X}_t) dW_t, \quad \tilde{X}_0 = \tilde{x},
\end{align}
and we have $\tilde{X}_t \in \cald^A$ for all $t \geq 0$, almost surely, showing that $\cald^A$ is invariant with respect to the diffusion \eqref{SDE-Y}. We define $\Sigma_A : H \to L_2^+(H)$ as $\Sigma_A(x) := \sigma_A(x) Q^{1/2}$ for each $x \in H$. Then for all $x \in \cald^A$ we have
\begin{align*}
C_A(x) &= A C (A^{-1} x) A^* = A \Sigma(A^{-1} x) \Sigma(A^{-1} x)^* A^*
\\ &= A \sigma(A^{-1} x) Q^{1/2} (\sigma(A^{-1} x) Q^{1/2})^* A^*
\\ &= A \sigma(A^{-1} x) Q^{1/2} (A \sigma(A^{-1} x) Q^{1/2})^*
\\ &= \sigma_A(x) Q^{1/2} (\sigma_A(x) Q^{1/2})^* = \Sigma_A(x) \Sigma_A(x)^*.
\end{align*}
Thus, by Lemma \ref{lemma-invariance-law} the set $\cald^A$ is also stochastically invariant with respect to the diffusion \eqref{SDE-Y-bar}.
\end{proof}

\begin{theorem}\label{thmnecessity}
	Suppose that Assumptions \ref{ass-lin-growth}, \ref{ass-extension}, \ref{ass-Sigma-self-adjoint}, \ref{ass-closed-range} are in force. Assume that  $\cald$ is stochastically invariant with respect to the diffusion \eqref{SDE}. Then conditions \eqref{cond-C-zero} and \eqref{cond-drift-N} hold for all $x \in \cald$ and $u \in \mathcal{N}^{1,\pr}_{\mathcal{D}}(x)$.
\end{theorem}

\begin{proof}  Let $(\mu_k(x))_{k=1,\ldots,r}$ and $(q_k(x))_{k=1,\ldots,r}$ be as in the spectral decomposition \eqref{E:spectralC} with $\mu_1(x)\ge \mu_2(x) \geq \ldots \geq \mu_r(x) > 0$. Note that the $(\mu_k(x))_{k=1,\ldots,r}$ are not necessarily all distinct.  {We shall perform a change of   variable to reduce to the case of simple spectrum  treated in Proposition \ref{propdistincteigen}.}  To do this, we fix $0 < \epsilon < 1$ and define the self-adjoint linear isomorphism $A^{\epsilon}_x \in L(H)$ given by
\begin{equation*}
 A^{\epsilon}_x  = \sum_{k=1}^{r} (1-\epsilon)^{k/2} \langle \cdot , q_k(x) \rangle \, q_k(x) + \Id_{U_x^{\perp}},
\end{equation*}
where $U_x := \lin \{ q_1(x), \ldots, q_r(x) \}$. Consider the new set $ \mathcal{D}^{\epsilon}:=A_x^{\epsilon} \mathcal{D}$. We define the mappings $b_{\epsilon} : H \to H$, $C_{\epsilon} : H \to L_1(H)$ and $\bar{\sigma}_{\epsilon} : H \to L_2^0(H)$ according to \eqref{b-A-def}, \eqref{C-A-def} and \eqref{sigma-A-def} with $A = A_x^{\epsilon}$. The discussion at the beginning of this section shows that Assumptions \ref{ass-lin-growth}, \ref{ass-extension}, \ref{ass-Sigma-self-adjoint}, \ref{ass-closed-range} are also fulfilled for the new diffusion
\begin{align}\label{SDE-epsilon}
dX_t^{\epsilon}=b_{\epsilon}(X_t^{\epsilon})dt+\bar{\sigma}_{\epsilon}(X_t^{\epsilon})dW_t, \quad X_0^{\epsilon} = x^{\epsilon},
\end{align}
and by Lemma \ref{lemma-inv-transfer} the set $\cald^{\epsilon}$ is invariant with respect to \eqref{SDE-epsilon}. Moreover, the positive eigenvalues of $C_{\epsilon}$ are all distinct at $x^{\epsilon}:=A_x^{\epsilon} x$, as by \eqref{E:spectralC} we have
\begin{align}\label{E:spectralC:eps}
C_{\epsilon}(x^{\epsilon}) = A_x^{\epsilon} C(x) A_x^{\epsilon} = \sum_{k=1}^r (1-\epsilon)^k \mu_k(x) \langle \cdot , q_k(x) \rangle \, q_k(x).
\end{align}
We can therefore apply Proposition \ref{propdistincteigen} to $(X^{\epsilon},\mathcal{D}^{\epsilon})$:
\begin{align}\label{eqcase(ii)}
C_{\epsilon}(x^{\epsilon})u_{\epsilon}=0  \quad \mbox{ and } \quad \langle  u_{\epsilon}, b_{\epsilon}(x^{\epsilon}) \rangle -\frac{1}{2} \sum_{j=1}^{\infty} \la u_{\epsilon}, D C_{\epsilon}^j(x^{\epsilon})(C_{\epsilon}C_{\epsilon}^+)^j(x^{\epsilon})    \rangle \leq 0
	\end{align}
for all  $u_{\epsilon} \in \mathcal{N}^{p}_{A_x^{\epsilon} \mathcal{D}}(x^{\epsilon})$, where we take into account Remark \ref{rem-main}.
Moreover, by Lemma \ref{lemma-normal-cones-transformation} we have $ \mathcal{N}^{p}_{A_x^{\epsilon}\mathcal{D}} (x^{\epsilon}) = (A_x^{\epsilon})^{-1} \mathcal{N}^{p}_{\mathcal{D}}(x)$. Now, let $u \in \mathcal{N}^{p}_{\mathcal{D}}(x)$ be arbitrary and set $u_{\epsilon} := (A_x^{\epsilon})^{-1} u \in \mathcal{N}^{p}_{A_x^{\epsilon}\mathcal{D}} (x^{\epsilon})$. Since $A_x^{\epsilon} \to \Id$ in $L(H)$, we also have $(A_x^{\epsilon})^{-1} \to \Id$ in $L(H)$, because the mapping $I(H) \to I(H)$, $A \mapsto A^{-1}$ is continuous, where $I(H) \subset L(H)$ denotes the open subset of linear isomorphisms; see \cite[Lemma 2.5.5]{Abraham}. Moreover, the mapping
\begin{align*}
\Phi : L(H) \times H \to H, \quad \Phi(T,x) = T x
\end{align*}
is a continuous bilinear operator. Therefore, we obtain
\begin{align*}
\lim_{\epsilon \to 0} u_{\epsilon} = \lim_{\epsilon \to 0} \Phi \big( (A_x^{\epsilon})^{-1},u \big) = \Phi \Big( \lim_{\epsilon \to 0} (A_x^{\epsilon})^{-1}, u \Big) = \Phi(\Id, u) = u.
\end{align*}
Furthermore, the mapping
\begin{align*}
\Psi : L(H) \times L_1(H) \times L(H) \to L_1(H), \quad \Psi(T,S,R) = TSR
\end{align*}
is a continuous three-linear operator. Hence, noting that $L_1(H) \hookrightarrow L(H)$ with continuous embedding, sending $\epsilon \rightarrow 0$ we obtain
\begin{align*}
C_{\epsilon}(x^{\epsilon}) u_{\epsilon} = \Phi(\Psi(A_x^{\epsilon},C(x),A_x^{\epsilon}),u_{\epsilon}) \longrightarrow \Phi(\Psi(\Id,C(x),\Id),u) = C(x)u,
\end{align*}
showing \eqref{cond-C-zero}. Furthermore, sending $\epsilon \rightarrow 0$ we obtain
\begin{align*}
\langle  u_{\epsilon}, b_{\epsilon}(x^{\epsilon}) \rangle = \la u_{\epsilon}, A_x^{\epsilon} b(x) \ra = \la u_{\epsilon}, \Phi(A_x^{\epsilon}, b(x)) \ra \longrightarrow \la u, \Phi(\Id, b(x)) \ra = \la u, b(x) \ra.
\end{align*}
Moreover, the mapping
\begin{align*}
\Pi : L(H,L_1(H)) \times H \to L_1(H), \quad \Pi(T,x) = T x
\end{align*}
is a continuous bilinear operator. By Lemma \ref{lemma-C-eps-diff} we have
\begin{align*}
D C_{\epsilon}(x^{\epsilon}) v = A_x^{\epsilon} ( DC(x) (A_x^{\epsilon})^{-1} v ) A_x^{\epsilon}, \quad v \in H.
\end{align*}
Hence, taking $v = P_C(x) u_{\epsilon}$ we obtain
\begin{align*}
D C_{\epsilon}(x^{\epsilon}) P_C(x) u_{\epsilon} &= A_x^{\epsilon} ( DC(x) (A_x^{\epsilon})^{-1} P_C(x) u_{\epsilon} ) A_x^{\epsilon}
\\ &= \Psi \big( A_x^{\epsilon}, DC(x) ( (A_x^{\epsilon})^{-1} P_C(x) u_{\epsilon} ), A_x^{\epsilon} \big)
\\ &= \Psi \Big( A_x^{\epsilon}, \Pi \big( DC(x), \Phi((A_x^{\epsilon})^{-1},\Phi(P_C(x),u_{\epsilon})) \big), A_x^{\epsilon} \Big).
\end{align*}
Noting that by \eqref{E:spectralC} and \eqref{E:spectralC:eps} we have $\ran(C_{\epsilon}(x^{\epsilon})) = U_x = \ran(C(x))$, it follows that
\begin{align*}
P_{C_{\epsilon}}(x^{\epsilon}) = P_C(x).
\end{align*}
Since by Lemma \ref{lemma-trace-functional} the trace is a continuous linear functional on $L_1(H)$, by Lemma \ref{lemma-series-1} we obtain
\begin{align*}
&\sum_{j=1}^{\infty} \la u_{\epsilon}, D C_{\epsilon}^j(x^{\epsilon})(C_{\epsilon}C_{\epsilon}^+)^j(x^{\epsilon}) \ra = \tr \big( D C_{\epsilon}(x^{\epsilon}) P_C(x) u_{\epsilon}) \big)
\\ &= \tr \, \Psi \Big( A_x^{\epsilon}, \Pi \big( DC(x), \Phi((A_x^{\epsilon})^{-1},\Phi(P_C(x),u_{\epsilon})) \big), A_x^{\epsilon} \Big)
\\ &\longrightarrow \tr \, \Psi \Big( \Id, \Pi \big( DC(x), \Phi(\Id,\Phi(P_C(x),u)) \big), \Id \Big)
\\ &= \tr \big( D C(x) P_C(x) u \big) = \sum_{j=1}^{\infty} \la u, D C^j(x)(C C^+)^j(x) \ra \quad \text{as $\epsilon \rightarrow 0$,}
\end{align*}
showing \eqref{cond-drift-N}. This ends the proof.
\end{proof}

\section{The positive maximum principle}\label{sec-maximum-principle}

In this section we prove that the invariance conditions from our two main results imply that the positive maximum principle is fulfilled. In Section \ref{sec-max-principle-suff-sigma} this is done for the invariance conditions from Theorem \ref{thm-main-sigma}, and in Section \ref{sec-max-principle-suff} this is done for the invariance conditions from Theorem \ref{thm-main}.

Throughout this section we consider the mathematical framework from Section \ref{sec-main-results}, and Assumptions \ref{ass-lin-growth}, \ref{ass-Sigma-self-adjoint} will always be in force. The Heine-Borel property (Assumption \ref{ass-Heine-Borel}) of the closed subset $\cald \subset H$ will not be required in this section. Recall from Definition \ref{def-pos-max} that the generator $\call$ satisfies the positive maximum principle if $\call \phi(x) \leq 0$ for any $x \in \cald$ and any function $\phi : H \to \bbr$ of class $C^2$ such that $\displaystyle\max_{\mathcal{D}} \phi = \phi(x) \geq 0$. The following auxiliary result will be useful; it is an immediate consequence of Lemma \ref{lemma-phi-in-N2}.

\begin{lemma}\label{lemma-pos-max}
Suppose that for all $x \in \cald$ and all $(u,v) \in \caln_{\cald}^2(x)$ we have
\begin{align}\label{tr-inequ-for-pos-max-prin}
\la u,b(x) \ra + \frac{1}{2} \tr ( v C(x) ) \leq 0.
\end{align}
Then the generator $\call$ satisfies the positive maximum principle.
\end{lemma}

\subsection{Positive maximum principle under smoothness on $\sigma$: Theorem \ref{thm-main-sigma}}\label{sec-max-principle-suff-sigma}

In this section we show that the invariance conditions from Theorem \ref{thm-main-sigma} imply that the positive maximum principle is fulfilled. Besides Assumptions \ref{ass-lin-growth}, \ref{ass-Sigma-self-adjoint}, we suppose that Assumption \ref{ass-sigma-smooth} is fulfilled, and that for all $x \in \cald$ and all $u \in \caln_{\cald}^{1,\pr}(x)$ we have \eqref{cond-sigma-zero} and \eqref{cond-drift-sigma-N}. For any $j \in \bbn$ let us consider the $H$-valued deterministic ordinary differential equation (ODE)
\begin{align}\label{ODE-sigma}
\left\{
\begin{array}{rcl}
y'(t) & = & \sigma^j(y(t))
\\ y(0) & = & x.
\end{array}
\right.
\end{align}

\begin{lemma}\label{lemma-det-system}
The subset $\cald$ is locally invariant for the ODE \eqref{ODE-sigma}.
\end{lemma}

\begin{proof}
Let $x \in \cald$ be arbitrary. Choosing an appropriate open and bounded neighborhood $N(x) \subset H$ of $x$, by Proposition \ref{prop-modify-fct-bdd} we may assume that $\sigma$ is of class $C_b^2$, which in particular implies that $\sigma^j$ is Lipschitz continuous; see Proposition \ref{prop-Cb1-Lipschitz}. Thus, taking into account condition \eqref{cond-sigma-zero} and Proposition \ref{prop-cond-vol}, the stated result is a consequence of Theorem \ref{thm-Nagumo-Banach}.
\end{proof}

\begin{lemma}\label{lemma-derivatives-ODE}
For each $x \in \cald$ we have
\begin{align*}
y(0) &= x,
\\ y'(0) &= \sigma^j(x),
\\ y''(0) &= D \sigma^j(x) \sigma^j(x),
\end{align*}
where $y$ denotes the local solution to the ODE \eqref{ODE-sigma}.
\end{lemma}

\begin{proof}
It is clear that $y(0) = x$ and $y'(0) = \sigma^j(y(0)) = \sigma^j(x)$. Moreover, we have
\begin{align*}
y''(0) &= \frac{d}{dt} y'(t)|_{t=0} = \frac{d}{dt} \sigma^j(y(t))|_{t=0} = D \sigma^j(y(t)) y'(t) |_{t=0}
\\ &= D \sigma^j(y(0)) y'(0) = D \sigma^j(x) \sigma^j(x),
\end{align*}
completing the proof.
\end{proof}

\begin{proposition}\label{prop-gen-u-v-sigma}
Suppose that conditions \eqref{cond-sigma-zero} and \eqref{cond-drift-sigma-N} are satisfied for all $x \in \cald$ and all $u \in \caln_{\cald}^{1,\pr}(x)$. Then for all $x \in \cald$ and all $(u,v) \in \caln_{\cald}^2(x)$ we have \eqref{tr-inequ-for-pos-max-prin}.
\end{proposition}

\begin{proof}
Let $x \in \cald$ and $(u,v) \in \caln_{\cald}^2(x)$ be arbitrary. By Lemma \ref{lemma-cones-1-2} we have $u \in \caln_{\cald}^{1}(x)$. Let us first assume that $u \in \caln_{\cald}^{1,\pr}(x)$. Furthermore, let $j \in \bbn$ be arbitrary. By Lemma \ref{lemma-det-system} there exists $T > 0$ such that
\begin{align}\label{invariant-y-sigma}
y(\sqrt{h}) \in \cald \quad \text{for all $h \in [0,T]$,}
\end{align}
where $y : [0,T] \to H$ denotes the local solution to the ODE \eqref{ODE-sigma}. Furthermore, by Taylor's theorem (see, e.g. \cite[Thm. 2.4.15]{Abraham}) we have
\begin{equation}\label{Taylor-ODE}
\begin{aligned}
y(\sqrt{h}) &= y(0) + D y(0) \sqrt{h} + \frac{1}{2} D^2 y(0) h + R(h) h^2
\\ &= y(0) + \sqrt{h} y'(0) + \frac{h}{2} y''(0) + R(h) h^2, \quad h \in [0,T],
\end{aligned}
\end{equation}
where the remainder term $R : [0,T] \to H$ is a continuous mapping with $R(0) = 0$. By Lemma \ref{lemma-derivatives-ODE} we obtain
\begin{align}\label{eqn-Taylor-sigma}
y(\sqrt{h}) - x = \sqrt{h} \sigma^j(x) + \frac{h}{2} D \sigma^j(x) \sigma^j(x) + R(h) h^2, \quad h \in [0,T].
\end{align}
In view of Lemma \ref{lemma-adjoint-ker-ran}, condition \eqref{cond-sigma-zero} implies $\la u, \sigma^j(x) \ra = 0$. Therefore, we obtain
\begin{align*}
\la u, y(\sqrt{h}) - x \ra = \frac{h}{2} \la u, D \sigma^j(x) \sigma^j(x) \ra + h^2 \la u, R(h) \ra, \quad h \in [0,T],
\end{align*}
and hence
\begin{align*}
\lim_{h \downarrow 0} \frac{\la u, y(\sqrt{h}) - x \ra}{h} = \frac{1}{2} \la u, D \sigma^j(x) \sigma^j(x) \ra.
\end{align*}
Furthermore, evaluating $\la v(y(\sqrt{h}) - x), y(\sqrt{h}) - x \ra$ using \eqref{eqn-Taylor-sigma} we obtain
\begin{align*}
\lim_{h \downarrow 0} \frac{\la v(y(\sqrt{h}) - x), y(\sqrt{h}) - x \ra}{h} = \la v \sigma^j(x), \sigma^j(x) \ra.
\end{align*}
In view of the last two identities, we arrive at
\begin{align*}
&\lim_{h \downarrow 0} \frac{1}{h} \bigg( \la u, y(\sqrt{h}) - x \ra + \frac{1}{2} \la v(y(\sqrt{h}) - x), y(\sqrt{h}) - x \ra \bigg)
\\ &= \frac{1}{2} \Big( \la u, D \sigma^j(x) \sigma^j(x) \ra + \la v \sigma^j(x), \sigma^j(x) \ra \Big).
\end{align*}
Noting that $(u,v) \in \caln_{\cald}^2(x)$, by Definition \ref{def-normal-cones}, the continuity of $y$ and \eqref{invariant-y-sigma} we obtain
\begin{align}\label{limit-ODE}
\lim_{h \downarrow 0} \frac{1}{h} \Big( \la u, y(\sqrt{h}) - x \ra + \frac{1}{2} \la v(y(\sqrt{h}) - x), y(\sqrt{h}) - x \ra \Big) \leq 0,
\end{align}
and it follows that
\begin{align*}
\la u, D \sigma^j(x) \sigma^j(x) \ra + \la v \sigma^j(x), \sigma^j(x) \ra \leq 0.
\end{align*}
Therefore, using Lemma \ref{lemma-trace-sigma-part-2} and \eqref{cond-drift-sigma-N} we arrive at
\begin{align*}
\la u,b(x) \ra + \frac{1}{2} \tr ( v C(x) ) &= \la u, b(x) \ra - \frac{1}{2} \sum_{j=1}^{\infty} \la u, D \sigma^j(x) \sigma^j(x) \ra
\\ &\quad + \frac{1}{2} \sum_{j=1}^{\infty} \la u, D \sigma^j(x) \sigma^j(x) \ra + \frac{1}{2} \sum_{j=1}^{\infty} \la v \sigma^j(x), \sigma^j(x) \ra \leq 0,
\end{align*}
which proves \eqref{tr-inequ-for-pos-max-prin} in case $u \in \caln_{\cald}^{1,\pr}(x)$. Since the mapping
\begin{align*}
L(H) \times L_1(H) \to L_1(H), \quad (T,S) \mapsto TS
\end{align*}
is a continuous bilinear operator, and the trace is a continuous linear functional on $L_1(H)$ due to Lemma \ref{lemma-trace-functional}, the mapping
\begin{align*}
\cald \to H, \quad x \mapsto \tr ( v C(x) )
\end{align*}
is continuous. Consequently, using Proposition \ref{prop-normal-inward-b-cont} completes the proof.
\end{proof}

\begin{proposition}\label{prop-pos-max-prin-sigma}
Suppose that Assumptions \ref{ass-lin-growth}, \ref{ass-Sigma-self-adjoint}, \ref{ass-sigma-smooth} are in force, and that conditions \eqref{cond-sigma-zero} and \eqref{cond-drift-sigma-N} are satisfied for all $x \in \cald$ and all $u \in \caln_{\cald}^{1,\pr}(x)$. Then the generator $\call$ satisfies the positive maximum principle.
\end{proposition}

\begin{proof}
This is an immediate consequence of Lemma \ref{lemma-pos-max} and Proposition \ref{prop-gen-u-v-sigma}.
\end{proof}

\subsection{Positive maximum principle under smoothness on $C$: Theorem \ref{thm-main}}\label{sec-max-principle-suff}

In this section we show that the invariance conditions from Theorem \ref{thm-main} imply that the positive maximum principle is fulfilled. Besides Assumptions \ref{ass-lin-growth}, \ref{ass-Sigma-self-adjoint}, we suppose that Assumptions \ref{ass-extension}, \ref{ass-closed-range} are in force, and that for all $x \in \cald$ and all $u \in \caln_{\cald}^{1,\pr}(x)$ we have \eqref{cond-C-zero} and \eqref{cond-drift-N}. For any $j \in \bbn$ let us consider the $H$-valued deterministic control system
\begin{align}\label{det-control-system}
\left\{
\begin{array}{rcl}
y'(t) & = & a_j(x,y(t))
\\ y(0) & = & x,
\end{array}
\right.
\end{align}
where for $x \in H$ and $j \in \bbn$ the mapping $a_j(x,\cdot) : H \to H$ is defined as
\begin{align*}
a_j(x,y) := C(y) \Sigma(x)^+ e_j, \quad y \in H.
\end{align*}

\begin{lemma}\label{lemma-det-control-system}
The subset $\cald$ is locally invariant for the deterministic control system \eqref{det-control-system}.
\end{lemma}

\begin{proof}
Let $x \in \cald$ be arbitrary. Choosing an appropriate open and bounded neighborhood $N(x) \subset H$ of $x$, by Proposition \ref{prop-modify-fct-bdd} we may assume that $C : H \to L_1(H)$ is of class $C_b^2$, which in particular implies that $a_j(x,\cdot)$ is Lipschitz continuous; see Proposition \ref{prop-Cb1-Lipschitz}. In view of Proposition \ref{prop-cond-vol}, condition \eqref{cond-C-zero} implies that for all $y \in \cald$ we have
\begin{align*}
a_j(x,y) = \Sigma(y) \big( \Sigma(y) \Sigma(x)^+ e_j \big) \in T_{\cald}^b(y).
\end{align*}
Therefore, by Theorem \ref{thm-Nagumo-Banach} the subset $\cald$ is locally invariant for the ODE
\begin{align*}
\left\{
\begin{array}{rcl}
y'(t) & = & a_j(x,y(t))
\\ y(0) & = & y,
\end{array}
\right.
\end{align*}
which gives the stated result by choosing the initial point $y=x$.
\end{proof}

\begin{lemma}\label{lemma-derivatives-controll}
For each $x \in \cald$ we have
\begin{align*}
y(0) &= x,
\\ y'(0) &= C(x) \Sigma(x)^+ e_j,
\\ y''(0) &= D C(x) [\Sigma(x)^+ e_j] [ C(x) \Sigma(x)^+ e_j ],
\end{align*}
where $y$ denotes the local solution to the deterministic control system \eqref{det-control-system}.
\end{lemma}

\begin{proof}
It is clear that $y(0) = x$ and
\begin{align*}
y'(0) = a_j(x,y(0)) = a_j(x,x) = C(x) \Sigma(x)^+ e_j.
\end{align*}
Moreover, we have
\begin{align*}
y''(0) &= \frac{d}{dt} y'(t) |_{t=0} = \frac{d}{dt} a_j(x,y(t)) |_{t=0} = D_y a_j(x,y(t)) y'(t) |_{t=0}
\\ &= D_y a_j(x,y(0)) y'(0) = D C(x) [\Sigma(x)^+ e_j] [ C(x) \Sigma(x)^+ e_j ],
\end{align*}
completing the proof.
\end{proof}

\begin{lemma}\label{lemma-series-2}
For each $x \in \cald$ and each self-adjoint operator $v \in L(H)$ we have
\begin{align*}
\tr \big( v C(x) \big) = \sum_{j=1}^{\infty} \la v C(x) \Sigma(x)^+ e_j, C(x) \Sigma(x)^+ e_j \ra.
\end{align*}
\end{lemma}

\begin{proof}
Recalling that $C(x) = \Sigma(x)^2$, by Lemma \ref{lemma-inverse-rules} and Lemma \ref{lemma-trace-commute} we have
\begin{align*}
\tr \big( v C(x) \big) &= \tr \big( v C(x) (\Sigma(x)^+)^2 C(x) \big)
\\ &= \tr \big( v C(x) \Sigma(x)^+ (C(x) \Sigma(x)^+)^* \big)
\\ &= \tr \big( (C(x) \Sigma(x)^+)^* v C(x) \Sigma(x)^+ \big)
\\ &= \sum_{j=1}^{\infty} \la ( C(x) \Sigma(x)^+ )^* v C(x) \Sigma(x)^+ e_j, e_j \ra
\\ &= \sum_{j=1}^{\infty} \la v C(x) \Sigma(x)^+ e_j, C(x) \Sigma(x)^+ e_j \ra,
\end{align*}
completing the proof.
\end{proof}

\begin{proposition}\label{prop-gen-u-v}
Suppose that conditions \eqref{cond-sigma-zero} and \eqref{cond-drift-sigma-N} are satisfied for all $x \in \cald$ and all $u \in \caln_{\cald}^{1,\pr}(x)$. Then for all $x \in \cald$ and all $(u,v) \in \caln_{\cald}^2(x)$ we have \eqref{tr-inequ-for-pos-max-prin}.
\end{proposition}

\begin{proof}
Let $x \in \cald$ and $(u,v) \in \caln_{\cald}^2(x)$ be arbitrary. By Lemma \ref{lemma-cones-1-2} we have $u \in \caln_{\cald}^{1}(x)$. Let us first assume that $u \in \caln_{\cald}^{1,\pr}(x)$. Furthermore, let $j \in \bbn$ be arbitrary. By Lemma \ref{lemma-det-control-system} there exists $T > 0$ such that we have \eqref{invariant-y-sigma}, where $y : [0,T] \to H$ denotes the local solution to the deterministic control system \eqref{det-control-system}. As in the proof of Proposition \ref{prop-gen-u-v-sigma}, we perform the second order Taylor expectation \eqref{Taylor-ODE}, and by Lemma \ref{lemma-derivatives-controll} we obtain
\begin{equation}\label{eqn-Taylor}
\begin{aligned}
y(\sqrt{h}) - x &= \sqrt{h} C(x) \Sigma(x)^+ e_j + \frac{h}{2} DC(x) [ \Sigma(x)^+ e_j ] [ C(x) \Sigma(x)^+ e_j ]
\\ &\quad + R(h) h^2, \quad h \in [0,T].
\end{aligned}
\end{equation}
By Lemma \ref{lemma-inverse-rules} and \eqref{cond-C-zero} we have
\begin{align*}
\la u, C(x) \Sigma(x)^+ e_j \ra = \la (C(x) \Sigma(x)^+)^* u, e_j \ra = \la \Sigma(x)^+ C(x) u, e_j \ra = 0.
\end{align*}
Therefore, we obtain
\begin{align*}
\la u, y(\sqrt{h}) - x \ra = \frac{h}{2} \la u, DC(x) [ \Sigma(x)^+ e_j ] [ C(x) \Sigma(x)^+ e_j ] \ra + h^2 \la u, R(h) \ra, \quad h \in [0,T],
\end{align*}
and hence
\begin{align*}
\lim_{h \downarrow 0} \frac{\la u, y(\sqrt{h}) - x \ra}{h} = \frac{1}{2} \la u, DC(x) [ \Sigma(x)^+ e_j ] [ C(x) \Sigma(x)^+ e_j ] \ra.
\end{align*}
Furthermore, evaluating $\la v(y(\sqrt{h}) - x), y(\sqrt{h}) - x \ra$ using \eqref{eqn-Taylor} we obtain
\begin{align*}
\lim_{h \downarrow 0} \frac{\la v(y(\sqrt{h}) - x), y(\sqrt{h}) - x \ra}{h} = \la v(C(x) \Sigma(x)^+ e_j), C(x) \Sigma(x)^+ e_j \ra.
\end{align*}
In view of the last two identities, we arrive at
\begin{align*}
&\lim_{h \downarrow 0} \frac{1}{h} \bigg( \la u, y(\sqrt{h}) - x \ra + \frac{1}{2} \la v(y(\sqrt{h}) - x), y(\sqrt{h}) - x \ra \bigg)
\\ &= \frac{1}{2} \Big( \la u, DC(x) [ \Sigma(x)^+ e_j ] [ C(x) \Sigma(x)^+ e_j ] \ra + \la v(C(x) \Sigma(x)^+ e_j), C(x) \Sigma(x)^+ e_j \ra \Big).
\end{align*}
As in the proof of Proposition \ref{prop-gen-u-v-sigma}, we have \eqref{limit-ODE}, and it follows that
\begin{align*}
\la u, DC(x) [ \Sigma(x)^+ e_j ] [ C(x) \Sigma(x)^+ e_j ] \ra + \la v(C(x) \Sigma(x)^+ e_j), C(x) \Sigma(x)^+ e_j \ra \leq 0.
\end{align*}
Therefore, using Lemma \ref{lemma-series-1}, Lemma \ref{lemma-series-2} and \eqref{cond-drift-N} we arrive at
\begin{align*}
\la u,b(x) \ra + \frac{1}{2} \tr ( v C(x) ) &= \la u, b(x) \ra - \frac{1}{2} \sum_{j=1}^{\infty} \la u, D C^j(x) (C C^+)^j(x) \ra
\\ &\quad + \frac{1}{2} \sum_{j=1}^{\infty} \la u, DC(x) [ \Sigma(x)^+ e_j ] [ C(x) \Sigma(x)^+ e_j ] \ra
\\ &\quad + \frac{1}{2} \sum_{j=1}^{\infty} \la v C(x) \Sigma(x)^+ e_j, C(x) \Sigma(x)^+ e_j \ra \leq 0,
\end{align*}
which proves \eqref{tr-inequ-for-pos-max-prin} in case $u \in \caln_{\cald}^{1,\pr}(x)$. Now, using Proposition \ref{prop-normal-inward-b-cont} as in the proof of Proposition \ref{prop-gen-u-v-sigma} completes the proof.
\end{proof}

\begin{proposition}\label{prop-pos-max-prin}
Suppose that Assumptions \ref{ass-lin-growth},  \ref{ass-extension}, \ref{ass-Sigma-self-adjoint}, \ref{ass-closed-range} are in force, and that conditions \eqref{cond-C-zero} and \eqref{cond-drift-N} are satisfied for all $x \in \cald$ and all $u \in \caln_{\cald}^{1,\pr}(x)$. Then the generator $\call$ satisfies the positive maximum principle.
\end{proposition}

\begin{proof}
This is an immediate consequence of Lemma \ref{lemma-pos-max} and Proposition \ref{prop-gen-u-v}.
\end{proof}

\section{The sufficiency proof}\label{sec-sufficiency}

In this section we prove that the positive maximum principle implies stochastic invariance. Throughout this section we consider the mathematical framework from Section \ref{sec-main-results}, and Assumptions \ref{ass-lin-growth}, \ref{ass-Heine-Borel} will always be in force. Moreover, we assume that the generator $\call$ satisfies the positive maximum principle. Let us point out that Assumptions \ref{ass-extension}, \ref{ass-Sigma-self-adjoint}, \ref{ass-closed-range} will not be required in this section.

Note that the Heine-Borel property (Assumption \ref{ass-Heine-Borel}) implies that $\cald$ is locally compact, and that the sets
\begin{align*}
\cald_n := \{ x \in \cald : \| x \| \leq n \}, \quad n \in \bbn
\end{align*}
are compact subsets of $H$. Recall that a continuous function $f : \cald \to \bbr$ \emph{vanishes at infinity} if for each $\epsilon > 0$ there exists a compact subset $K \subset \cald$ such that $| f(x) | < \epsilon$ for all $x \in \cald \setminus K$. We denote by $C_0(\cald)$ the Banach space of real-valued continuous functions $f : \cald \to \bbr$ vanishing at infinity, equipped with the supremum norm
\begin{align}\label{sup-norm}
\| f \|_{\infty} = \sup_{x \in \cald} | f(x) |, \quad f \in C_0(\cald).
\end{align}
For $k \in \bbn \cup \{ \infty \}$ we denote by $C_0^k(\cald)$ the space of all functions $\phi \in C_0(\cald)$ which can be extended to a function $\Phi \in C^k(H)$.

\begin{lemma}\label{lemma-domain-dense-pre}
The space $C_0^{\infty}(\cald)$ is dense in $C_0(\cald)$ with respect to the supremum norm \eqref{sup-norm}.
\end{lemma}

\begin{proof}
Let $f \in C_0(\cald)$ be arbitrary. By the Tietze extension theorem (see, e.g. \cite[Kor. B.1.6]{Werner}) there exists a bounded extension $F \in C(H)$ such that $\| f \|_{\infty} = \| F \|_{\infty}$. Let $n \in \bbn$ be arbitrary. By \cite[Thm. 1.1]{Azagra} there is a function $\Psi_n \in C^{\infty}(H)$ such that
\begin{align*}
| F(x) - \Psi_n(x) | \leq \frac{1}{n (1 + \| x \|)}, \quad x \in H.
\end{align*}
Setting $\psi_n := \Psi_n|_{\cald}$, we obtain
\begin{align*}
| f(x) - \psi_n(x) | \leq \frac{1}{n (1 + \| x \|)}, \quad  x \in \cald.
\end{align*}
It remains to show that $\psi_n \in C_0^{\infty}(\cald)$. For this purpose, let $\epsilon > 0$ be arbitrary. There is a compact subset $K \subset \cald$ such that
\begin{align*}
| f(x) | < \frac{\epsilon}{2}, \quad x \in \cald \setminus K.
\end{align*}
Furthermore, there exists an index $m \in \bbn$ such that
\begin{align*}
\frac{1}{n(1+m)} \leq \frac{\epsilon}{2}.
\end{align*}
We define the compact subset $C := K \cup \cald_m$. Then for all $x \in \cald \setminus C$ we obtain
\begin{align*}
| \psi_n(x) | \leq | f(x) | + | f(x) - \psi_n(x) | < \frac{\epsilon}{2} + \frac{1}{n (1 + \| x \|)} \leq \epsilon,
\end{align*}
completing the proof.
\end{proof}

Recall that the \emph{support} of a function $f : \cald \to \bbr$ is defined as
\begin{align*}
{\rm supp}(f) := \overline{\{ x \in \cald : f(x) \neq 0 \}}.
\end{align*}
We denote by $C_c(\cald)$ the space of all continuous functions $f : \cald \to \bbr$ with compact support. For $k \in \bbn \cup \{ \infty \}$ we denote by $C_c^k(\cald)$ the space of all functions $\phi \in C_c(\cald)$ which can be extended to a function $\Phi \in C^k(H)$. Obviously, we have $C_c(\cald) \subset C_0(\cald)$ and $C_c^k(\cald) \subset C_0^k(\cald)$ for each $k \in \bbn \cup \{ \infty \}$.

\begin{lemma}\label{lemma-domain-dense}
The space $C_c^{\infty}(\cald)$ is dense in $C_0(\cald)$ with respect to the supremum norm \eqref{sup-norm}.
\end{lemma}

\begin{proof}
In view of Lemma \ref{lemma-domain-dense-pre} it suffices to prove that $C_c^{\infty}(\cald)$ is dense in $C_0^{\infty}(\cald)$. Let $f \in C_0^{\infty}(\cald)$ be arbitrary. Furthermore, let $n \in \bbn$ be arbitrary. By Lemma \ref{lemma-bump-r-R} there is a function $\varphi_n : H \to [0,1]$ of class $C^{\infty}$ such that $\varphi_n(x) = 1$ for all $x \in H$ with $\| x \| \leq n$ and $\varphi_n(x) = 0$ for all $x \in H$ with $\| x \| \geq n+1$. We define $\Phi_n := \varphi_n \cdot f : H \to \bbr$, which is also of class $C^{\infty}$ due to Proposition \ref{prop-Leibniz}. Moreover, we set $\phi_n := \Phi_n|_{\cald} : \cald \to \bbr$. Then we have $\phi_n(x) = 0$ for all $x \in \cald \setminus \cald_{n+1}$, showing that $\phi_n$ has compact support.

Now, let $\epsilon > 0$ be arbitrary. There exists a compact subset $K \subset \cald$ such that
\begin{align}\label{f-on-comp}
|f(x)| < \epsilon, \quad x \in \cald \setminus K.
\end{align}
Since $K$ is compact, it is also bounded. Hence, there exists $n \in \bbn$ such that $K \subset \cald_{n}$. Let $x \in \cald$ be arbitrary. If $\| x \| > n$, then we have $x \in \cald \setminus \cald_n$, and by \eqref{f-on-comp} we obtain
\begin{align*}
| f(x) - \phi_n(x) | = |f(x) - \varphi_n(x) f(x)| = |1 - \varphi_n(x)| \cdot |f(x)| < \epsilon.
\end{align*}
Furthermore, if $\| x \| \leq n$, then we have $| f(x) - \phi_n(x) | = 0$, completing the proof.
\end{proof}

Recall that the generator $\call : C^2(H) \to C(H)$ is given by \eqref{generator-def}. In what follows, we will consider the generator $\call$ on the domain $\cald(\call) := C_c^2(\cald)$. This is done as follows. We define the linear operator $\call : \cald(\call) \to C(\cald)$ as
\begin{align*}
\call \phi := \call \Phi|_{\cald}, \quad \phi \in  \cald(\call),
\end{align*}
where $\Phi \in C^2(H)$ denotes any extension of $\phi$. Since $\call$ satisfies the positive maximum principle, this definition does not depend on the choice of the extension $\Phi$. More precisely, we have the following auxiliary result.

\begin{lemma}
Let $\phi \in \cald(\call)$ be arbitrary, and let $\Phi, \Psi \in C^2(H)$ be two extensions of $\phi$. Then we have $\call \Phi|_{\cald} = \call \Psi|_{\cald}$.
\end{lemma}

\begin{proof}
Define $f,g \in C^2(H)$ as $f := \Phi - \Psi$ and $g := \Psi - \Phi$. Then we have $f|_{\cald} = 0$ and $g|_{\cald} = 0$. Now, let $x \in \cald$ be arbitrary. Then we have $\displaystyle\max_{\mathcal{D}} f = f(x) \geq 0$ and $\displaystyle\max_{\mathcal{D}} g = g(x) \geq 0$. Since the generator $\call$ satisfies the positive maximum principle, we obtain $\call f(x) \leq 0$ and $\call g(x) \leq 0$, and hence $\call \Phi(x) = \call \Psi(x)$.
\end{proof}

The next auxiliary result shows that the generator is a linear operator $\call : \cald(\call) \to C_c(\cald)$.

\begin{lemma}\label{lemma-L-in-C-c}
We have $\call(\cald(\call)) \subset C_c(\cald)$.
\end{lemma}

\begin{proof}
Let $\phi \in \cald(\call)$ be arbitrary, and let $\Phi \in C^2(H)$ be an extension of $\phi$. Since $\supp(\phi)$ is compact, it is also bounded. Hence, there exists an index $n \in \bbn$ such that $\supp(\phi) \subset \cald_n$. By Lemma \ref{lemma-bump-r-R} there is a function $\varphi_n : H \to [0,1]$ of class $C^{\infty}$ such that $\varphi_n(x) = 1$ for all $x \in H$ with $\| x \| \leq n$ and $\varphi_n(x) = 0$ for all $x \in H$ with $\| x \| \geq n+1$. We define $\Psi := \varphi_n \cdot \Phi : H \to \bbr$, which is of class $C^2(H)$ due to Proposition \ref{prop-Leibniz}. The mapping $\Psi$ is also and extension of $\phi$, and the support of $\call \phi = \call \Psi|_{\cald}$ is contained in $\cald_{n+1}$.
\end{proof}

Since $\cald$ is locally compact, we can consider the one-point compactification $\cald^{\Delta} := \cald \cup  \{ \Delta \}$, where $\Delta$ denotes the point at infinity. Then $\cald^{\Delta}$ is a compact Hausdorff space; see \cite[p. 185, 186]{Cohn}. Furthermore, by \cite[Prop. 7.1.13 and Lemma 7.1.14]{Cohn} the one-point compactification $\cald^{\Delta}$ is metrizable. For a function $\cald \to \bbr$ we have $f \in C_0(\cald)$ if and only if the extension $f^{\Delta} : \cald^{\Delta} \to \bbr$ defined by
\begin{align*}
f^{\Delta}(x) =
\begin{cases}
f(x), & \text{if } x \in \cald,
\\ 0, & \text{if } x = \Delta,
\end{cases}
\end{align*}
is continuous; see p. 206, Exercise 3 in \cite{Cohn}. Therefore the following auxiliary result holds true.

\begin{lemma}\label{lemma-compact-continuous}
Let $f \in C_0(\cald)$ be arbitrary, and let $f^{\Delta} : \cald^{\Delta} \to \bbr$ be an extension of $f$. Then we have $f^{\Delta} \in C(\cald^{\Delta})$ if and only if $f^{\Delta}(\Delta) = 0$.
\end{lemma}

Now we extend the generator $\call$ to a linear operator $\call^{\Delta} : \cald(\call^{\Delta}) \to C(\cald^{\Delta})$ as follows. First, we introduce the domain
\begin{align*}
\cald(\call^{\Delta}) := \{ f \in C(\cald^{\Delta}) : f|_{\cald} \in \cald(\call) \}.
\end{align*}
Note that by Lemma \ref{lemma-compact-continuous} we have
\begin{align}\label{domain-Delta}
\cald(\call^{\Delta}) = \{ f : \cald^{\Delta} \to \bbr : f|_{\cald} \in \cald(\call) \text{ and } f(\Delta) = 0 \}.
\end{align}
Furthermore, for each $f \in \cald(\call^{\Delta})$ we set
\begin{align*}
(\call^{\Delta} f)(x) :=
\begin{cases}
(\call f)(x), & x \in \cald,
\\ 0, & x = \Delta.
\end{cases}
\end{align*}
Then by Lemma \ref{lemma-compact-continuous} and Lemma \ref{lemma-L-in-C-c} we have indeed $\call^{\Delta}(\cald(\call^{\Delta})) \subset C(\cald^{\Delta})$.

Recall that a $\cald^{\Delta}$-valued adapted process $X$ on some stochastic basis is called a \emph{solution of the martingale problem for $\call^{\Delta}$} if for each $f \in \cald(\call^{\Delta})$ the process
\begin{align*}
f(X_t) - \int_0^t (\call^{\Delta} f)(X_s) ds, \quad t \in \bbr_+
\end{align*}
is a martingale. If additionally $\bbp \circ X_0^{-1} = \nu$ for some probability measure $\nu \in \calp(\cald^{\Delta})$, then we say that $X$ is a \emph{solution of the martingale problem for $(\call^{\Delta},\nu)$}.

\begin{proposition}\label{prop-mart-sol-1}
For every probability measure $\nu \in \calp(\cald^{\Delta})$ there exists a $\cald^{\Delta}$-valued c\`{a}dl\`{a}g solution of the martingale problem for $(\call^{\Delta},\nu)$.
\end{proposition}

\begin{proof}
By Lemma \ref{lemma-domain-dense} the domain $\cald(\call)$ is dense in $C_0(\cald)$. Moreover, the linear operator $\call$ satisfies the positive maximum principle. Therefore, taking account Lemma \ref{lemma-compact-continuous}, by \cite[Thm. 4.5.4]{EK} the statement follows.
\end{proof}

For convenience of the reader, we will briefly recall some further notions. For this purpose, let $S$ be a metric space. A sequence $(f_n)_{n \in \bbn} \subset B(S)$ of bounded functions is said to converge \emph{boundedly and pointwise} to a bounded function $f \in B(S)$ if $\sup_{n \in \bbn} \| f_n \|_{\infty} < \infty$ and $\lim_{n \to \infty} f_n(x) = f(x)$ for every $x \in S$; we denote this by
\begin{align*}
\text{bp-lim}_{n \to \infty} f_n = f.
\end{align*}
A subset $M \subset B(S) \times B(S)$ is called \emph{bp-closed} if for all sequences $(f_n,g_n)_{n \in \bbn}$ and all $(f,g) \in B(S) \times B(S)$ such that $\text{bp-lim}_{n \to \infty} f_n = f$ and $\text{bp-lim}_{n \to \infty} g_n = g$ we have $(f,g) \in M$. The \emph{bp-closure} of $M \subset B(S) \times B(S)$ is the smallest bp-closed subset of $B(S) \times B(S)$ that contains $M$.

As in \cite[Sec. 1.4]{EK}, we will consider $\call^{\Delta}$ as the multivalued linear operator given by its graph
\begin{align*}
\call^{\Delta} = \{ (f,\call^{\Delta} f) : f \in \cald(\call^{\Delta}) \}.
\end{align*}

\begin{lemma}\label{lemma-bp-closure}
The pair $(\bbI_{\cald},0)$ is in the bp-closure of $\call^{\Delta} \cap (C(\cald^{\Delta}) \times B(\cald^{\Delta}))$.
\end{lemma}

\begin{proof}
There is a function $f \in C^{\infty}(\bbr;[0,1])$ such that $f(y) = 1$ for all $y \leq 1$ and $f(y) = 0$ for all $y \geq 2$. For each $n \in \bbn$ we define $\ell_n \in L(\bbr)$ as
\begin{align*}
\ell_n(y) := \frac{y}{n}, \quad y \in \bbr
\end{align*}
and $f_n := f \circ \ell_n$. Then we have $f_n \in C^{\infty}(\bbr;[0,1])$ with $f_n(y) = 1$ for all $y \leq n$ and $f_n(y) = 0$ for all $y \geq 2n$ as well as
\begin{align}\label{fn-diff}
f_n(y) = f \Big( \frac{y}{n} \Big), \quad f_n'(y) = \frac{1}{n} \cdot f' \Big( \frac{y}{n} \Big), \quad f_n''(y) = \frac{1}{n^2} \cdot f'' \Big( \frac{y}{n} \Big)
\end{align}
for all $y \in \bbr$. The norm function $\eta : H \setminus \{ 0 \} \to (0,\infty)$ given by
\begin{align*}
\eta(x) := \| x \|, \quad x \in H \setminus \{ 0 \}
\end{align*}
is of class $C^{\infty}$ with first and second order derivatives
\begin{align*}
D \eta(x)v &= \frac{\la x,v \ra}{\| x \|}, \quad x \in H \setminus \{ 0 \} \text{ and } v \in H,
\\ D^2 \eta(x)(v,w) &= \frac{\la v,w \ra}{\| x \|} - \frac{\la x,v \ra \la x,w \ra}{\| x \|^3}, \quad x \in H \setminus \{ 0 \} \text{ and } v,w \in H.
\end{align*}
Therefore, for all $x \in H \setminus \{ 0 \}$ we have
\begin{align}\label{eta-est}
\| D \eta(x) \| \leq 1 \quad \text{and} \quad \| D^2 \eta(x) \| \leq \frac{2}{\| x \|}.
\end{align}
For each $n \in \bbn$ we define $\Phi_n : H \to \bbr$ as
\begin{align*}
\Phi_n(x) := f_n(\| x \|), \quad x \in H.
\end{align*}
Then we have $\Phi_n \in C^{\infty}(H)$ with $\Phi_n(x) = 1$ for all $x \in H$ with $\| x \| \leq n$ and $\Phi_n(x) = 0$ for all $x \in H$ with $\| x \| \geq 2n$. In particular, it follows that
\begin{align}\label{bp-Phi-1}
\| \Phi_n \|_{\infty} \leq 1, \quad  n \in \bbn.
\end{align}
Furthermore, by the first and second order chain rules, for all $x \in H \setminus \{ 0 \}$ we obtain
\begin{align*}
D(f_n \circ \eta)(x) &= Df_n(\eta(x)) D \eta(x) = f_n'(\eta(x)) D \eta(x),
\\ D^2(f_n \circ \eta)(x) &= D^2 f_n(\eta(x)) \circ ( D \eta(x), D \eta(x) ) + Df_n(\eta(x)) \circ D^2 \eta(x).
\end{align*}
Therefore, by \eqref{eta-est} and \eqref{fn-diff}, for all $x \in H \setminus \{ 0 \}$ we have the estimates
\begin{align*}
\| D(f_n \circ \eta)(x) \| &\leq | f_n'(\eta(x)) | = \frac{1}{n} \bigg| f' \bigg( \frac{\| x \|}{n} \bigg) \bigg|,
\\ \| D^2(f_n \circ \eta)(x) \| &\leq | f_n''(\eta(x)) | + | f_n'(\eta(x)) | \cdot \frac{2}{\| x \|}
\\ &= \frac{1}{n^2} \bigg| f'' \bigg( \frac{\| x \|}{n} \bigg) \bigg| + \frac{2}{n \| x \|} \bigg| f' \bigg( \frac{\| x \|}{n} \bigg) \bigg|.
\end{align*}
Thus, using the linear growth condition \eqref{linear-growth-2}, for all $x \in H$ with $n \leq \| x \| \leq 2n$ we obtain
\begin{align*}
\| D\Phi_n(x)b(x) \| \leq \| D\Phi_n(x) \| \, \| b(x) \| \leq \frac{\| f' \|_{\infty}}{n} L (1 + 2n) \leq 3 L \| f' \|_{\infty}
\end{align*}
as well as
\begin{align*}
\| D^2 \Phi_n(x) C(x) \| &\leq \| D^2 \Phi_n(x) \| \, \| C(x) \| \leq \bigg( \frac{\| f'' \|_{\infty}}{n^2} + \frac{2\| f' \|_{\infty}}{n \| x \|} \bigg) \| \Sigma(x) \|^2
\\ &\leq \bigg( \frac{\| f'' \|_{\infty}}{n^2} + \frac{2\| f' \|_{\infty}}{n^2} \bigg) L^2(1 + 2n)^2 \leq 9L^2 \big( \| f'' \|_{\infty} + 2\| f' \|_{\infty} \big).
\end{align*}
Consequently, there is a constant $M > 0$ such that
\begin{align}\label{bp-Phi-2}
\| \call \Phi_n \|_{\infty} \leq M, \quad n \in \bbn.
\end{align}
Now, we set $\phi_n := \Phi_n|_{\cald} : \cald \to \bbr$ for each $n \in \bbn$. Then we have $\phi_n(x) = 0$ for all $x \in \cald \setminus \cald_{2n}$, and hence $\phi_n \in \cald(\call)$. Finally, we define the extension $\phi_n^{\Delta} : \cald^{\Delta} \to \bbr$ as
\begin{align*}
\phi_n^{\Delta}(x) :=
\begin{cases}
\phi_n(x), & \text{if $x \in \cald$,}
\\ 0, & \text{if $x = \Delta$.}
\end{cases}
\end{align*}
Then by \eqref{domain-Delta} we have $\phi_n^{\Delta} \in \cald(\call^{\Delta})$ for each $n \in \bbn$. Moreover, by \eqref{bp-Phi-1} and \eqref{bp-Phi-2} we obtain
\begin{align*}
\text{bp-}\lim_{n \to \infty} \phi_n^{\Delta} = \bbI_{\cald} \quad \text{and} \quad \text{bp-}\lim_{n \to \infty} \call^{\Delta} \phi_n^{\Delta} = 0,
\end{align*}
completing the proof.
\end{proof}

\begin{proposition}\label{prop-mart-sol-2}
For every probability measure $\nu \in \calp(\cald)$ there exists a $\cald$-valued c\`{a}dl\`{a}g solution of the martingale problem for $(\call,\nu)$.
\end{proposition}

\begin{proof}
Recall that $\cald$ is open in $\cald^{\Delta}$, and that $\cald^{\Delta}$ is metrizable. Therefore, the statement is a consequence of Proposition \ref{prop-mart-sol-1}, Lemma \ref{lemma-bp-closure} and \cite[Thm. 4.3.8]{EK}.
\end{proof}

\begin{proposition}\label{prop-mart-sol-3}
For every probability measure $\nu \in \calp(\cald)$ there exists a $\cald$-valued continuous solution of the martingale problem for $(\call,\nu)$.
\end{proposition}

\begin{proof}
There exists a function $\varphi \in C_c^{\infty}(\bbr)$ such that $\varphi(0) = 1$, $\varphi'(0) = \varphi''(0) = 0$ and $\varphi(x) < 1$ for all $x \in \bbr$ with $x \neq 0$. We define the symmetric function $f : \cald \times \cald \to \bbr$ as
\begin{align*}
f(x,y) := -\varphi(\| x-y \|^2), \quad x,y \in \cald.
\end{align*}
Let $\theta : H \to \bbr_+$ be the square of the norm function; that is
\begin{align}\label{theta-norm-2}
\theta(z) := \| z \|^2, \quad z \in H.
\end{align}
Then $\theta$ is of class $C^{\infty}$ and we can express $f$ as
\begin{align*}
f(x,y) = -(\varphi \circ \theta)(x-y), \quad x,y \in \cald.
\end{align*}
Let $y \in \cald$ be arbitrary. We define $\Phi_y : H \to \bbr$ as
\begin{align*}
\Phi_y(x) := -(\varphi \circ \theta)(x-y), \quad x \in H.
\end{align*}
Then we have $\Phi_y \in C_c^{\infty}(H)$ and $f(\cdot,y) = \Phi_y|_{\cald}$, showing that $f(\cdot,y) \in \cald(\call)$.

Let $\epsilon > 0$ and $K \subset \cald$ be compact. Then for all $x,y \in K$ with $\| x-y \| \geq \epsilon$ we have
\begin{align*}
f(x,y) - f(y,y) = 1 - \varphi(\| x-y \|^2) \geq 1 - \varphi(\epsilon) > 0,
\end{align*}
because $\varphi(\epsilon) < 1$. Therefore, we obtain
\begin{align*}
\inf \{ f(x,y) - f(y,y) : x,y \in K \text{ with } \| x-y \| \geq \epsilon \} > 0.
\end{align*}
Now, we define $g : \cald \times \cald \to \bbr$ as
\begin{align*}
g(x,y) := \call f(\cdot,y)(x), \quad x,y \in \cald.
\end{align*}
Then we have $(f(\cdot,y),g(\cdot,y)) \in \call$ for each $y \in \cald$, where we recall that we consider $\call$ as the multivalued linear operator given by its graph
\begin{align*}
\call = \{ (f, \call f) : f \in \cald(\call) \}.
\end{align*}
Let $z \in H$ be arbitrary. By the definition \eqref{theta-norm-2} we obtain the first and second order derivatives
\begin{align*}
D \theta(z) v &= 2 \la z,v \ra, \quad v \in H,
\\ D^2 \theta(z)(v,w) &= 2 \la v,w \ra, \quad v,w \in H.
\end{align*}
Therefore, by the first and second order chain rules we obtain
\begin{align*}
D (\varphi \circ \theta)(z)v &= D \varphi(\theta(z)) D \theta(z)v = 2 \varphi'(\| z \|^2) \la z,v \ra, \quad v \in H,
\\ D^2 (\varphi \circ \theta)(z)(v,w) &= D^2 \varphi(\theta(z)) \circ ( D \theta(z)v, D\theta(z)w ) + D \varphi(\theta(z)) \circ D^2 \theta(z)(v,w)
\\ &= 4 \varphi''(\| z \|^2) \la z,v \ra \la z,w \ra + 2 \varphi'(\| z \|^2) \la v,w \ra, \quad v,w \in H.
\end{align*}
Let $x \in \cald$ be arbitrary. Then we have
\begin{align*}
g(x,y) = \call \Phi_y(x) = \la D \Phi_y(x), b(x) \ra + \frac{1}{2} \tr \big( D^2 \Phi_y(x) C(x) \big), \quad y \in \cald.
\end{align*}
Taking into account Remark \ref{remark-derivatives} we have
\begin{align*}
\la D \Phi_y(x), v \ra = -D_x (\varphi \circ \theta)(x-y)v = -2 \varphi'(\| x-y \|^2) \la x-y,v \ra, \quad v \in H
\end{align*}
as well as
\begin{align*}
\la D^2 \Phi_y(x)v,w \ra &= - \la D_x^2 (\varphi \circ \theta)(x-y)v,w \ra
\\ &= -4 \varphi''(\| x-y \|^2) \la x-y,v \ra \la x-y,w \ra
\\ &\quad - 2 \varphi'(\| x-y \|^2) \la v,w \ra, \quad v,w \in H.
\end{align*}
Since $\varphi'(0) = \varphi''(0) = 0$, it follows that
\begin{align*}
\lim_{y \to x} g(x,y) = g(x,x) = 0.
\end{align*}
Finally, note that the metric space $\cald$ is separable, because it is a subset of the separable Hilbert space $H$. Now, the statement is a consequence of Proposition \ref{prop-mart-sol-2} and Problem 19 on page 265 in \cite{EK}.
\end{proof}

\begin{lemma}\label{lemma-loc-MT-problem}
Let $X$ be a $\cald$-valued continuous solution of the martingale problem for $\call$ such that $X_0 = \xi$ for some $\xi \in \cald$. Then for each $\phi \in C^2(\cald)$ the process $M$ given by
\begin{align*}
M_t := \phi(X_t) - \int_0^t (\call \phi)(X_s)ds, \quad t \in \bbr_+
\end{align*}
is a local martingale.
\end{lemma}

\begin{proof}
We define the localizing sequence $(\tau_n)_{n \in \bbn}$ of stopping times as
\begin{align*}
\tau_n := \inf \{ t \in \bbr_+ : \| X_t \| \geq n \}, \quad n \in \bbn.
\end{align*}
There exists a function $\Phi \in C^2(H)$ such that $\phi = \Phi|_{\cald}$. Let $n \in \bbn$ with $n \geq \| \xi \|$ be arbitrary. By Lemma \ref{lemma-bump-r-R} there is a function $\varphi_n : H \to [0,1]$ of class $C^{\infty}$ such that $\varphi_n(x) = 1$ for all $x \in H$ with $\| x \| \leq n$ and $\varphi_n(x) = 0$ for all $x \in H$ with $\| x \| \geq n+1$. We define the function $\Phi_n := \varphi_n \cdot \Phi$, which belongs to $C_c^2(H)$ by virtue of Proposition \ref{prop-Leibniz}. Therefore, setting $\phi_n := \Phi_n|_{\cald}$, we have $\phi_n \in \cald(\call)$. Since $X$ is a solution of the martingale problem for $\call$, the process $M^n$ given by
\begin{align*}
M_t^n := \phi_n(X_t) - \int_0^t (\call \phi_n)(X_s)ds, \quad t \in \bbr_+
\end{align*}
is a martingale. Therefore, the stopped process $(M^n)^{\tau_n}$ is also a martingale. Furthermore, we have $M^{\tau_n} = (M^n)^{\tau_n}$, proving that $M$ is a local martingale.
\end{proof}

\begin{theorem}\label{thm-suff}
Suppose that Assumptions \ref{ass-lin-growth}, \ref{ass-extension}, \ref{ass-Sigma-self-adjoint}, \ref{ass-Heine-Borel} are in force, and that the generator $\call$ satisfies the positive maximum principle. Then $\cald$ is stochastically invariant with respect to the diffusion \eqref{SDE}.
\end{theorem}

\begin{proof}
We will show that for each $\xi \in \cald$ there exists a $\cald$-valued weak solution $X$ to the SDE \eqref{SDE} with $X_0 = \xi$. Indeed, by Proposition \ref{prop-mart-sol-3} there exists a $\cald$-valued continuous solution $X$ of the martingale problem for $\call$ with $X_0 = \xi$. By the continuity of $b$ we can define the adapted process $M$ as
\begin{align}\label{M-MT-problem}
M_t := X_t - \int_0^t b(X_s) ds, \quad t \in \bbr_+.
\end{align}
Let $h \in H$ be arbitrary. Moreover, let $\Phi^h : H \to \bbr$ be the continuous linear functional given by $\Phi^h := \la h,\cdot \ra$, and set $\phi^h := \Phi^h|_{\cald} \in C^{\infty}(\cald)$. Since $D \Phi^h(x) = \la h,\cdot \ra$ for each $x \in H$ and $D^2 \Phi^h = 0$, we have
\begin{align*}
(\call \phi^h)(x) = \la h,b(x) \ra, \quad x \in \cald.
\end{align*}
Hence, by Lemma \ref{lemma-loc-MT-problem} the process $\la h,M \ra$ is a local martingale for each $h \in H$. Consequently, by Lemma \ref{lemma-loc-MT-functionals} the process $M$ is a local martingale.

Now, let $h,g \in H$ be arbitrary. Moreover, let $\Phi^{h,g} : H \to \bbr$ be the function given by the product $\Phi^{h,g} := \Phi^h \cdot \Phi^g$, and set $\phi^{h,g} := \Phi^{h,g}|_{\cald} \in C^{\infty}(\cald)$. Let $x \in H$ be arbitrary. Taking into account Remark \ref{remark-derivatives}, by the product rule (Proposition \ref{prop-Leibniz}) we obtain the first and second order derivatives
\begin{align*}
\la D \Phi^{h,g}(x), v \ra &= \la h,x \ra \la g,v \ra + \la h,v \ra \la g,x \ra, \quad v \in H,
\\ \la D^2 \Phi^{h,g}(x) v,w \ra &= 2 \la h,v \ra \la g,w \ra, \quad v,w \in H.
\end{align*}
Since $\Sigma(x) \Sigma(x)^*$ is self-adjoint, we have
\begin{align*}
\frac{1}{2} \tr \big( D^2 \Phi^{h,g}(x) \Sigma(x) \Sigma(x)^* \big) &= \frac{1}{2} \sum_{j=1}^{\infty} \la D^2 \Phi^{h,g}(x) \Sigma(x) \Sigma(x)^* e_j,e_j \ra
\\ &= \sum_{j=1}^{\infty} \la h, \Sigma(x) \Sigma(x)^* e_j \ra \la g,e_j \ra
\\ &= \sum_{j=1}^{\infty} \la \Sigma(x) \Sigma(x)^* h,e_j \ra \la e_j,g \ra = \la \Sigma(x) \Sigma(x)^* h,g \ra.
\end{align*}
Therefore, we obtain
\begin{align*}
(\call \phi^{h,g})(x) = \la h,x \ra \la g,b(x) \ra + \la g,x \ra \la h,b(x) \ra + \la C(x) h,g \ra, \quad x \in \cald.
\end{align*}
Hence, by Lemma \ref{lemma-loc-MT-problem} the process $N^{h,g}$ given by
\begin{align*}
N_t^{h,g} &:= \la h,X_t \ra \la g,X_t \ra
\\ &\quad - \int_0^t \Big( \la h,X_s \ra \la g,b(X_s) \ra + \la g,X_s \ra \la h,b(X_s) \ra + \la C(X_s) h,g \ra \Big) ds
\end{align*}
for each $t \in \bbr_+$ is a local martingale. Performing analogous calculations as in \cite[p. 315]{Karatzas-Shreve} it follows that the process $\widetilde{N}^{h,g}$ given by
\begin{align*}
\widetilde{N}_t^{h,g} := \langle h, M_t \rangle \langle g, M_t \rangle - \int_0^t \la C(X_s)h, g \ra ds, \quad t \in \bbr_+
\end{align*}
is a local martingale. Therefore, the quadratic covariation of the two local martingales $\la M,h \ra$ and $\la M,g \ra$ is given by
\begin{align}\label{co-var-h-g}
\la \, \la M,h \ra, \la M,g \ra \, \ra_t = \int_0^t \la C(X_s)h, g \ra ds, \quad t \in \bbr_+.
\end{align}
Now, we wish to show that the local martingale $M$ is actually a square-integrable martingale. Let $T \in \bbr_+$ be an arbitrary finite time horizon. We define the localizing sequence $(\tau_n)_{n \in \bbn}$ of stopping times as
\begin{align*}
\tau_n := \inf \{ t \in [0,T] : \| X_t \| \geq n \}.
\end{align*}
Let $n \in \bbn$ with $n \geq \| \xi \|$ be arbitrary. Then we have
\begin{align*}
\| X^{\tau_n} \| \leq n.
\end{align*}
Furthermore, by \eqref{M-MT-problem} and the linear growth condition \eqref{linear-growth} we have
\begin{align*}
\| M^{\tau_n} \| \leq \| X^{\tau_n} \| + TL (1 + \| X^{\tau_n} \| ) \leq n + TL(1+n).´
\end{align*}
Thus $M^{\tau_n}$ is a continuous, bounded local martingale, and hence by Lemma \ref{lemma-bounded-loc-MT} it is a square-integrable martingale. Now, we also fix an arbitrary $t \in [0,T]$. By \eqref{M-MT-problem} we have
\begin{align*}
\bbe \bigg[ \sup_{s \in [0,t]} \| X_s^{\tau_n} \|^2 \bigg] \leq 2 \, \bbe \bigg[ \sup_{s \in [0,t]} \| M_s^{\tau_n} \|^2 \bigg] + 2 \, \bbe \Bigg[ \sup_{s \in [0,t]} \bigg\| \int_0^s b(X_u^{\tau_n}) du \bigg\|^2 \Bigg].
\end{align*}
By the linear growth condition \eqref{linear-growth} we obtain
\begin{equation}\label{b-calculation}
\begin{aligned}
&\bbe \Bigg[ \sup_{s \in [0,t]} \bigg\| \int_0^s b(X_u^{\tau_n}) du \bigg\|^2 \Bigg] \leq t \, \bbe \bigg[ \int_0^t \| b(X_u^{\tau_n}) \|^2 du \bigg]
\\ &\leq t L^2 \, \bbe \bigg[ \int_0^t \big( 1 + \| X_u^{\tau_n} \| \big)^2 du \bigg] \leq 2 t L^2 \, \bbe \bigg[ \int_0^t \big( 1 + \| X_u^{\tau_n} \|^2 \big) du \bigg]
\\ &\leq 2t L^2 \bigg( t + \int_0^t \bbe \bigg[ \sup_{u \in [0,s]} \| X_u^{\tau_n} \|^2 \bigg] ds \bigg).
\end{aligned}
\end{equation}
Moreover, noting that $M_0 = \xi$, by Doob's maximal inequality for martingales in Banach spaces (see, e.g. \cite[Thm. 2.2.7]{Liu-Roeckner}), the monotone convergence theorem, identity \eqref{co-var-h-g} and the linear growth condition \eqref{linear-growth-2} we have
\begin{align*}
&\bbe \bigg[ \sup_{s \in [0,t]} \| M_s^{\tau_n} \|^2 \bigg] \leq 4 \, \bbe \big[ \| M_t^{\tau_n} \|^2 \big] = 4 \, \bbe \bigg[ \sum_{j=1}^{\infty} | \la M_t^{\tau_n},e_j \ra |^2 \bigg] = 4 \sum_{j=1}^{\infty} \bbe \big[ | \la M_t^{\tau_n},e_j \ra |^2 \big]
\\ &= 4 \sum_{j=1}^{\infty} \Big( \bbe \big[ | \la \xi,e_j \ra |^2 \big] + \bbe \big[ \la \, \la M^{\tau_n},e_j \ra, \la M^{\tau_n},e_j \ra \, \ra_t \big] \Big)
\\ &= 4 \| \xi \|^2 + 4 \sum_{j=1}^{\infty} \bbe \bigg[ \int_0^t \la C(X_s^{\tau_n}) e_j,e_j \ra ds \bigg] = 4 \| \xi \|^2 + 4 \, \bbe \bigg[ \int_0^t \tr ( C(X_s^{\tau_n}) ) ds \bigg]
\\ &= 4 \| \xi \|^2 + 4 \, \bbe \bigg[ \int_0^t \| C(X_s^{\tau_n}) \|_{L_1(H)} ds \bigg] \leq 4 \| \xi \|^2 + 4 \, \bbe \bigg[ \int_0^t \| \Sigma(X_s^{\tau_n}) \|_{L_2(H)}^2 ds \bigg]
\\ &\leq 4 \| \xi \|^2 + 4 L^2 \, \bbe \bigg[ \int_0^t \big( 1 + \| X_s^{\tau_n} \| \big)^2 ds \bigg]
\\ &\leq 4 \| \xi \|^2 + 8 L^2 \bigg( t + \int_0^t \bbe \bigg[ \sup_{u \in [0,s]} \| X_u^{\tau_n} \|^2 \bigg] ds \bigg).
\end{align*}
Summing up, there is a constant $C > 0$, only depending on $\xi$, $T$ and $L$, such that
\begin{align*}
\bbe \bigg[ \sup_{s \in [0,t]} \| X_s^{\tau_n} \|^2 \bigg] \leq C \bigg( 1 + \int_0^t \bbe \bigg[ \sup_{u \in [0,s]} \| X_u^{\tau_n} \|^2 \bigg] ds \bigg), \quad t \in [0,T].
\end{align*}
Therefore, by Gronwall's inequality we deduce that
\begin{align*}
\bbe \bigg[ \sup_{t \in [0,T]} \| X_t^{\tau_n} \|^2 \bigg] \leq C e^{CT}.
\end{align*}
Thus, using Fatou's lemma we obtain
\begin{align*}
\bbe \bigg[ \sup_{t \in [0,T]} \| X_t \|^2 \bigg] = \bbe \bigg[ \lim_{n \to \infty} \sup_{t \in [0,T]} \| X_t^{\tau_n} \|^2 \bigg] \leq \liminf_{n \to \infty} \bbe \bigg[ \sup_{t \in [0,T]} \| X_t^{\tau_n} \|^2 \bigg] \leq C e^T.
\end{align*}
Therefore, by \eqref{M-MT-problem} and \eqref{b-calculation} with $t = T$ and $\tau_n = T$ we have
\begin{align*}
\bbe \bigg[ \sup_{t \in [0,T]} \| M_t \|^2 \bigg] \leq 2 \, \bbe \bigg[ \sup_{t \in [0,T]} \| X_t \|^2 \bigg] + 2 \, \bbe \Bigg[ \sup_{t \in [0,T]} \bigg\| \int_0^t b(X_s) ds \bigg\|^2 \Bigg] < \infty.
\end{align*}
Consequently, by Proposition \ref{prop-Banach-square-MT} the local martingale $M$ is a square-integrable martingale. Moreover, by \eqref{co-var-h-g} the quadratic variation $\langle\!\langle M, M \rangle\!\rangle$ is given by
\begin{align*}
\langle\!\langle M, M \rangle\!\rangle_t = \int_0^t C(X_s) ds, \quad t \in \bbr_+.
\end{align*}
Therefore, by the martingale representation theorem (see, e.g. \cite[Thm. 2.7]{Atma-book}) there exists an $H$-valued $Q$-Wiener process $W$ on an extended stochastic basis such that
\begin{align*}
M_t = \int_0^t \sigma(X_s) d W_s, \quad t \in \bbr_+.
\end{align*}
Therefore, by \eqref{M-MT-problem} the $\cald$-valued process $X$ is a weak solution to the SDE \eqref{SDE}.
\end{proof}

\section{Proof of Proposition~\ref{P:C=sigma2}}\label{S:proofofequalitydriftseries}

In this section we provide the proof of Proposition~\ref{P:C=sigma2}.

\begin{proof}[Proof of Proposition~\ref{P:C=sigma2}]
By Lemma \ref{lemma-HS-norm} the linear mapping $\Phi_Q : L_2^0(H) \to L_2(H)$ given by $\Phi_Q(T) = T Q^{1/2}$ is an isometric isomorphism. Noting that $\Sigma = \Phi_Q \circ \sigma$, by Proposition \ref{prop-smooth-linear} the mapping $\Sigma : H \to L_2(H)$ is of class $C^1$ (and even of class $C_b^1$, provided $\sigma$ is of class $C_b^1$), and we have
\begin{align}\label{expr-tr-2}
D \Sigma(x) v = ( D \sigma(x) v ) Q^{1/2}, \quad x,v \in H.
\end{align}
Furthermore, note that $C = B(\Sigma,\Sigma^*)$, where $B : L_2(H) \times L_2(H) \to L_1(H)$ denotes the continuous bilinear operator given by $B(T,S) := TS$. Moreover, by Lemma \ref{lemma-adjoint-isometry} the linear mapping
\begin{align*}
L_1(H) \to L_1(H), \quad T \mapsto T^*
\end{align*}
is an isometry. Therefore, by Propositions \ref{prop-Leibniz} and \ref{prop-smooth-linear} the mapping $C : H \to L_1^+(H)$ is of class $C^1$ (and even of class $C_b^1$, provided $\sigma$ is of class $C_b^1$), and we have
\begin{equation}\label{expr-C-diff}
 \begin{aligned}
DC(x)v &= B(D \Sigma(x) v, \Sigma(x)^*) + B(\Sigma(x), D \Sigma^*(x) v)
\\ &= (D \Sigma(x)v) \Sigma(x)^* + \Sigma(x) (D \Sigma(x)v)^*, \quad x,v \in H.
\end{aligned}
\end{equation}
For any $u \in H$ we denote by $\Phi_u : L_1(H) \to H$ the continuous linear operator given by $\Phi_u(S) := S u$. Then we have $C^j = \Phi_{e_j} \circ C$ for each $j \in \bbn$. Hence, applying Proposition \ref{prop-series-app} we obtain that the series \eqref{series-sigma2} is weakly convergent, and that for each $u \in H$ we have \eqref{series-sigma2-u}. Let us fix an arbitrary $x \in H$. By Proposition \ref{prop-smooth-linear} and \eqref{expr-C-diff} we obtain
\begin{align*}
D C^j(x) P_C^j(x) &= ( D C(x) P_C^j(x) ) e_j
\\ &= (D \Sigma(x) P_C^j(x)) \Sigma(x)^* e_j + \Sigma(x) (D \Sigma(x) P_C^j(x))^* e_j, \quad j \in \bbn.
\end{align*}
After noticing that for all $u \in \ker \Sigma(x)^*$,
$$\langle u, \Sigma(x) (D \Sigma(x) P_C^j(x))^* e_j \rangle=\langle \Sigma(x)^* u, (D \Sigma(x) P_C^j(x))^* e_j  \rangle =0, \quad j \in \bbn,$$
we see that
\begin{align}\label{E:DC3}
\sum_{j=1}^{\infty} \la u, D C^j(x) P_C^j(x) \ra = \sum_{j=1}^{\infty} \la u, (D \Sigma(x) P_C^j(x)) \Sigma(x)^* e_j \ra, \quad u \in \ker \Sigma(x)^*.
\end{align}
For any $u \in H$ we denote by $\Sigma^* u : H \to H$ the mapping $\Sigma^* u := \Phi_u \circ \Sigma^*$. Then, applying Proposition \ref{prop-smooth-linear} twice we have
\begin{align}\label{expr-tr-0}
D ( \Sigma^* u )(x)v = (D \Sigma^*(x) v) u = ( D \Sigma(x) v )^* u, \quad x,u,v \in H.
\end{align}
In particular, for $v=P_C^j(x)$ we obtain
\begin{align}\label{expr-tr-1}
D(\Sigma^* u) (x) P_C(x) e_j = ( D \Sigma(x)P_C^j(x) )^* u, \quad x,u \in H.
\end{align}
Note that $\sigma^j = \Psi_{f_j} \circ \sigma$ for all $j \in \bbn$, where for any $u \in H_0$ the continuous linear operator $\Psi_u : L_2(H_0,H) \to H$ is given by $\Psi_u(T) := Tu$. Hence, by Proposition \ref{prop-smooth-linear} we obtain
\begin{align}\label{expr-sigma-j-diff}
D \sigma^j(x)v = (D \sigma(x) v)f_j, \quad x,v \in H \text{ and } j \in \bbn.
\end{align}
Moreover, by Lemma \ref{lemma-orth-proj-range} we have
\begin{align}\label{proj-range}
P_C(x) \Sigma(x) = \Sigma(x), \quad x \in H.
\end{align}
Thus, plugging the expressions \eqref{expr-tr-1}, \eqref{proj-range}, \eqref{expr-tr-0}, \eqref{expr-tr-2} and \eqref{expr-sigma-j-diff} back in \eqref{E:DC3} yields
\begin{align*}
&\sum_{j=1}^{\infty} \la u, DC^j(x) P_C^j(x) \ra = \sum_{j=1}^{\infty} \la u, ( D \Sigma(x) P_C^j(x) ) \Sigma(x)^* e_j \ra
\\ &= \sum_{j=1}^{\infty} \la \Sigma(x) ( D \Sigma(x) P_C^j(x) )^* u, e_j \ra = \sum_{j=1}^{\infty} \la \Sigma(x) D(\Sigma^* u) (x) P_C(x) e_j, e_j \ra
\\ &= \tr \big( \Sigma(x) D(\Sigma^* u) (x) P_C(x) \big) = \tr \big( D(\Sigma^* u) (x) P_C(x) \Sigma(x) \big) = \tr \big( D(\Sigma^* u) (x) \Sigma(x) \big)
\\ &= \sum_{j=1}^{\infty} \la D(\Sigma^* u) (x) \Sigma(x) e_j, e_j \ra = \sum_{j=1}^{\infty} \la (D \Sigma(x) \Sigma(x) e_j )^* u, e_j \ra
\\ &= \sum_{j=1}^{\infty} \la u, (D \Sigma(x) \Sigma(x) e_j ) e_j \ra = \sum_{j=1}^{\infty} \la u, (D \sigma(x) \Sigma(x) e_j ) Q^{1/2} e_j \ra
\\ &= \sum_{j=1}^{\infty} \la u, (D \sigma(x) \sigma(x) Q^{1/2} e_j ) Q^{1/2} e_j \ra = \sum_{j=1}^{\infty} \la u, (D \sigma(x) \sigma(x) f_j ) f_j \ra
\\ &= \sum_{j=1}^{\infty} \la u, (D \sigma(x) \sigma^j(x)) f_j \ra = \sum_{j=1}^{\infty} \la u, D \sigma^j(x) \sigma^j(x) \ra
\end{align*}
for all $x \in H$ and $u \in \ker \Sigma(x)^*$, where we have used Lemma \ref{lemma-trace-commute} in the third line of the calculation. This proves \eqref{series-sigma2-Stratonovich}.
\end{proof}

\section{Proof of Lemma~\ref{lemmadoubleintstohilbert}}\label{S:proofdoublestochasticinequality}

In this section we provide the proof of Lemma~\ref{lemmadoubleintstohilbert}. For this purpose, recall that the sequence $(W^i)_{i \in \bbn}$ defined as
\begin{align*}
W^i := \frac{1}{\sqrt{\lambda_i}} \la W,e_i \ra
\end{align*}
is a sequence of independent real-valued standard Wiener processes; see \cite[Prop. 4.3]{Da_Prato}. Here $\{ e_i \}_{i \in \bbn}$ denotes the orthonormal basis of $H$ and $(\lambda_i)_{i \in \bbn} \subset (0,\infty)$ the sequence such that \eqref{diagonal} is fulfilled.

\begin{remark}\label{rem-Wiener-series-1}
Let $Y$ be a predictable $H_0$-valued process. Using the identification $H_0 \cong L_2(H_0,\bbr)$ from Lemma \ref{lemma-Riesz-Parseval}, we can consider the $\bbr$-valued It\^{o} integral
\begin{align*}
\int_0^t Y_s dW_s = \sum_{i=1}^{\infty} \int_0^t \la Y_s,f_i \ra_{H_0} dW_s^i,
\end{align*}
where the series converges unconditionally. Let us also recall from Proposition \ref{prop-Riesz-Parseval-1} that $H_0 \to \ell^2(\bbn)$, $y \mapsto ( \la y,f_i \ra_{H_0} )_{i \in \bbn}$ is an isometric isomorphism.
\end{remark}

\begin{remark}\label{rem-Wiener-series-2}
Let $Z$ be a predictable $L_2(H_0)$-valued process. Using the identification $L_2(H_0) \cong L_2(H_0,L_2(H_0,\bbr))$ from Proposition \ref{prop-Riesz-Parseval-2}, we can consider the $\bbr$-valued double It\^{o} integral
\begin{align*}
\int_0^t \bigg( \int_0^s Z_r dW_r \bigg) dW_s &= \int_0^t \bigg( \sum_{i=1}^{\infty} \int_0^s Z_r f_i dW_r^i \bigg) dW_s
\\ &= \sum_{j=1}^{\infty} \int_0^t \bigg\la \sum_{i=1}^{\infty} \int_0^s Z_r f_i dW_r^i, f_j \bigg\ra_{H_0} dW_s^j
\\ &= \sum_{j=1}^{\infty} \int_0^t \bigg( \sum_{i=1}^{\infty} \int_0^s \la Z_r f_i,f_j \ra_{H_0} dW_r^i \bigg) dW_s^j
\\ &= \sum_{i,j=1}^{\infty} \int_0^t \int_0^s \la Z_r f_i,f_j \ra_{H_0} dW_r^i dW_s^j,
\end{align*}
where the series converges unconditionally. Let us also recall from Proposition \ref{prop-Riesz-Parseval-2} that
\begin{align}\label{isom-H0-N-N}
L_2(H_0) \to \ell^2(\bbn \times \bbn), \quad T \mapsto \big( \la T f_i, f_j \ra_{H_0} \big)_{i,j \in \bbn}
\end{align}
is an isometric isomorphism.
\end{remark}

\begin{remark}\label{rem-Wiener-series-3}
Suppose that $\gamma \in L_2(H_0)$ is constant. Setting $\gamma^{ij} := \la \gamma f_i,f_j \ra_{H_0}$ for all $i,j \in \bbn$, by Remark \ref{rem-Wiener-series-2} we obtain
\begin{equation}\label{series-split}
\begin{aligned}
\int_0^t \bigg( \int_0^s \gamma dW_r \bigg) dW_s &= \sum_{i,j=1}^{\infty} \gamma^{ij} \int_0^t W_s^i dW_s^j
\\ &= \sum_{i=1}^{\infty} \frac{\gamma^{ii}}{2} \big( (W_t^i)^2 - t \big) + \sum_{i \neq j=1}^{\infty} \gamma^{ij} \int_0^t W_s^i dW_s^j.
\end{aligned}
\end{equation}
In the last step, we have used that
$$(W^i_t)^2 = 2 \int_0^t W^i_s d W^i_s + t, \quad i \in \bbn,$$
which follows from integration by parts. Moreover, note that the two series in the last line of \eqref{series-split} are also unconditionally convergent. For example, using the isometric isomorphism \eqref{isom-H0-N-N}, let us define $\widetilde{\gamma} \in L_2(H_0)$ as $\widetilde{\gamma}^{ii} := \gamma^{ii}$ for all $i \in \bbn$ and $\widetilde{\gamma}^{ij} := 0$ for $i \neq j$. Then we have
\begin{align*}
\int_0^t \bigg( \int_0^s \widetilde{\gamma} dW_r \bigg) dW_s = \sum_{i=1}^{\infty} \gamma^{ii} \int_0^t W_s^i dW_s^i = \sum_{i=1}^{\infty} \frac{\gamma^{ii}}{2} \big( (W_t^i)^2 - t \big).
\end{align*}
\end{remark}

\begin{proof}[Proof of Lemma~\ref{lemmadoubleintstohilbert}]
Let us agree on the notation
\begin{align}\label{notation-H0-1}
\alpha^i &:= \la \alpha,f_i \ra_{H_0}, \quad i \in \bbn,
\\ \label{notation-H0-2} \beta^i &:= \la \beta,f_i \ra_{H_0}, \quad i \in \bbn,
\\ \label{notation-H0-3} \gamma^{ij} &:= \la \gamma f_i, f_j \ra_{H_0}, \quad i,j \in \bbn.
\end{align}
Then, in view of Remarks \ref{rem-Wiener-series-1} and \ref{rem-Wiener-series-2} we can write \eqref{eqintegraledoublehilbert} as
\begin{align*}
\int_{0}^{t} \theta_s ds + \sum_{i=1}^{\infty} \alpha^i W_t^i + \sum_{i=1}^{\infty} \int_0^t \int_0^s \beta_r^i dr dW_s^i + \sum_{i,j=1}^{\infty} \int_0^t \int_0^s \gamma_r^{ij} d W_r^i d W_s^j \leq 0.
\end{align*}
This reduces to
\begin{align}\label{ineqn-short-time-proof}
\theta_0 t + \sum_{i=1}^{\infty} \alpha^i W^i_t + \int_{0}^{t} \bigg( \int_{0}^{s} \gamma_0  dW_r \bigg)  dW_s + R_t \leq 0,
\end{align}
where the remainder is given by
	\begin{eqnarray*}
		R_t &=& \int_{0}^{t} (\theta_s - \theta_0) ds + \int_{0}^{t}   \int_{0}^{s} \beta_rdr   dW_s  +  \int_{0}^{t}\int_{0}^{s} (\gamma_r -\gamma_0) dW_r   dW_s \\
		&=:&  R^1_t + R^2_t + R^3_t.
	\end{eqnarray*}
	In view of Lemma \ref{lemma-short-time-0} below, for the first conclusion $\alpha = 0$ it suffices to show that  $R_t/t \overset{\bbp}{\rightarrow} 0$ when $t \to 0$. To see this, first note that using L'H\^{o}pital's rule we have $R^1_t = o(t)$  a.s.~since $\theta$ is continuous at $0$. Moreover, since $\beta$ is bounded, there is a constant $C > 0$ such that $\| \beta_t(\omega) \|_{H_0} \leq C$ for all $(\omega,t) \in \Omega \times \bbr_+$. Hence, recalling Remark \ref{rem-Wiener-series-1}, by the It\^{o} isometry we obtain
	\begin{align*}
	\bbe \Bigg[ \bigg| \frac{R_t^2}{t} \bigg|^2 \Bigg] &= \frac{1}{t^2} \bbe \Bigg[ \bigg| \int_0^t \int_0^s \beta_r dr dW_s \bigg|^2 \Bigg] = \frac{1}{t^2} \bbe \Bigg[ \int_0^t \bigg\| \int_0^s \beta_r dr \bigg\|_{H_0}^2 ds \Bigg]
	\\ &\leq \frac{1}{t^2} \bbe \Bigg[ \int_0^t \bigg( \int_0^s \| \beta_r \|_{H_0} dr \bigg)^2 ds \Bigg] \leq \frac{C^2}{t^2} \bbe \bigg[ \int_0^t s^2 ds \bigg]
	\\ &= \frac{C^2}{t^2} \cdot \frac{t^3}{3} = \frac{C^2}{3} \cdot t \to 0 \quad \text{for $t \to 0$.}
	\end{align*}
	Therefore, we have $R^2_t = o(t)$ in $L^2$, and hence in probability.  Finally, recalling Remarks \ref{rem-Wiener-series-1} and \ref{rem-Wiener-series-2}, by the It\^{o} isometry and \eqref{eqgammacontinuity} we derive
	\begin{align*}
	\bbe \Bigg[ \bigg| \frac{R_t^3}{t} \bigg|^2 \Bigg] &= \frac{1}{t^2} \bbe \Bigg[ \bigg| \int_{0}^{t}\int_{0}^{s} (\gamma_r -\gamma_0) dW_r   dW_s \bigg|^2 \Bigg]
	\\ &= \frac{1}{t^2} \bbe \Bigg[ \int_{0}^{t} \bigg\| \int_{0}^{s} (\gamma_r -\gamma_0) dW_r \bigg\|_{H_0}^2 ds  \Bigg]
	\\ &= \frac{1}{t^2} \bbe \Bigg[ \int_{0}^{t} \int_{0}^{s} \| \gamma_r -\gamma_0 \|_{L_2(H_0)}^2 dr ds  \Bigg] \to 0 \quad \text{for $t \to 0$.}
	\end{align*}
	Thus, we have $R^3_t = o(t)$ in $L^2$, and hence in probability. By applying Lemma \ref{lemma-short-time-0} we conclude that $\alpha = 0$.

	Now suppose that additionally the series \eqref{series-gamma-0} converges. By Remark \ref{rem-Wiener-series-3} we can also express \eqref{ineqn-short-time-proof} as
	\begin{equation*}
	\bigg( \theta_0- \frac{1}{2} \mbox{Tr}(\gamma_0) \bigg) t + \sum_{i=1}^{\infty} \alpha^i W^i_t + \sum_{i=1}^{\infty} \frac{\gamma^{ii}_0}{2} (W^i_t)^2 + \sum_{1 \leq i \neq j}^{\infty} \gamma^{ij}_0 \int_0^t W^i_s d W^j_s + R_t \leq 0.
	\end{equation*}
	Consequently, applying Lemma \ref{lemma-short-time} with $\delta = \theta_0 - \frac{1}{2} \tr(\gamma_0)$ concludes the proof.
\end{proof}

\begin{lemma}\label{lemma-short-time-0}
Let $R$ be a real-valued process such that $\lim_{t \downarrow 0} \frac{R_t}{t} = 0$ in probability. Let $\alpha \in H_0$, $\widetilde{\gamma} \in L_2(H_0)$, and $\delta \in \bbr$ be such that for all $t \geq 0$
\begin{align*}
\sum_{i=1}^{\infty} \alpha^i W_t^i + \int_0^t \bigg( \int_0^s \widetilde{\gamma} \, dW_r \bigg) dW_s + \delta t + R_t \leq 0,
\end{align*}
where the notation is according to \eqref{notation-H0-1}. Then we have $\alpha = 0$.
\end{lemma}

\begin{proof}
Let us define the processes $I^1$, $I^2$ and $L$ as
\begin{align}\label{I-1-def}
I_t^1 &:= \sum_{i=1}^{\infty} \alpha^i W_t^i, \quad t \geq 0,
\\ \label{I-2-def} I_t^2 &:= \int_0^t \bigg( \int_0^s \widetilde{\gamma} \, dW_r \bigg) dW_s, \quad t \geq 0,
\\ \label{L-def} L_t &:= I_t^1 + I_t^2 + \delta t + R_t, \quad t \geq 0.
\end{align}
Then we have $L_t \leq 0$ for all $t \geq 0$, and by Remark \ref{rem-Wiener-series-1} we obtain
\begin{align}\label{I-1-identity}
I_t^1 = \int_0^t \alpha \, dW_s, \quad t \geq 0.
\end{align}
By Lemma \ref{lemma-short-time-gamma} below we have
\begin{align*}
\frac{I_t^2 + \delta t + R_t}{\sqrt{t}} \overset{\bbp}{\to} 0 \quad \text{as $t \to 0$.}
\end{align*}
Moreover, by the scaling property of Brownian motion we have
\begin{align*}
\frac{I_t^1}{\sqrt{t}} = \sum_{i=1}^{\infty} \alpha^i \frac{W_t^i}{\sqrt{t}} \overset{d}{=} \sum_{i=1}^{\infty} \alpha^i W_1^i = I_1^1.
\end{align*}
We have $L_t / \sqrt{t} \leq 0$ for all $t > 0$. This implies $I_1^1 \leq 0$ almost surely, and hence $\alpha = 0$.
\end{proof}

\begin{lemma}\label{lemma-short-time-gamma}
Let $(\gamma_t)_{t \geq 0}$ be a bounded, predictable $L_2(H_0)$-valued process. We define the real-valued process $I$ as
\begin{align*}
I_t := \int_0^t \int_0^s \gamma_r d W_r dW_s, \quad t \geq 0.
\end{align*}
Then for each $\delta \in (0,1)$ we have
\begin{align*}
\frac{I_t}{t^{\delta}} \overset{L^2}{\to} 0 \quad \text{for $t \to 0$.}
\end{align*}
\end{lemma}

\begin{proof}
Since $\gamma$ is bounded, there is a constant $C > 0$ such that $\| \gamma_t(\omega) \|_{L_2(H_0)} \leq C$ for all $(\omega,t) \in \Omega \times \bbr_+$. Recalling Remarks \ref{rem-Wiener-series-1} and \ref{rem-Wiener-series-2},
by the It\^{o} isometry we obtain
	\begin{align*}
	\bbe \Bigg[ \bigg| \frac{I_t}{t^{\delta}} \bigg|^2 \Bigg] &= \frac{1}{t^{2 \delta}} \bbe \Bigg[ \bigg| \int_{0}^{t}\int_{0}^{s} \gamma_r dW_r   dW_s \bigg|^2 \Bigg]
	= \frac{1}{t^{2 \delta}} \bbe \Bigg[ \int_{0}^{t} \bigg\| \int_{0}^{s} \gamma_r dW_r \bigg\|_{H_0}^2 ds  \Bigg]
	\\ &= \frac{1}{t^{2 \delta}} \bbe \Bigg[ \int_{0}^{t} \int_{0}^{s} \| \gamma_r \|_{L_2(H_0)}^2 dr ds  \Bigg] \leq \frac{C^2}{t^{2 \delta}} \int_0^t s ds = \frac{C^2}{t^{2 \delta}} \cdot \frac{t^2}{2} \to 0
	\end{align*}
as $t \to 0$, completing the proof.
\end{proof}

Now we present an infinite dimensional version of \cite[Lemma 2.1]{buc}. For the specification \eqref{specification-operator} of the linear operator $\widetilde{\gamma} \in L_2(H_0)$ below we use that \eqref{isom-H0-N-N} is an isometric isomorphism.

\begin{lemma}\label{lemma-short-time}
Let $R$ be a real-valued process such that $\lim_{t \downarrow 0} \frac{R_t}{t} = 0$ in probability. Let $\alpha, \beta \in H_0$, $\gamma \in L_2(H_0)$, and $\delta \in \bbr$ be such that for all $t \geq 0$
\begin{align*}
\sum_{i=1}^{\infty} \alpha^i W_t^i + \sum_{i=1}^{\infty} \beta^i (W_t^i)^2 + \sum_{i \neq j} \gamma^{ij} \int_0^t W_s^i d W_s^j + \delta t + R_t \leq 0,
\end{align*}
where the notation is according to \eqref{notation-H0-1}--\eqref{notation-H0-3}. Then the following statements are true:
\begin{enumerate}
\item We have $\alpha = 0$.

\item We have $\beta^i \leq 0$ for all $i \in \bbn$.

\item The operator $\gamma$ is self-adjoint.

\item Consider the operator $\widetilde{\gamma} \in L_2(H_0)$, specified as
\begin{align}\label{specification-operator}
\widetilde{\gamma}^{ij} := \la \widetilde{\gamma} f_i,f_j \ra_{H_0} :=
\begin{cases}
\gamma^{ij}, & \text{if $i \neq j$,}
\\ 2 \beta^i, & \text{if $i = j$.}
\end{cases}
\end{align}
Then $\widetilde{\gamma}$ is self-adjoint, and we have $-\widetilde{\gamma} \in L_2^+(H_0)$.

\item We have $\delta \leq 0$.
\end{enumerate}
\end{lemma}

\begin{proof}
Let us define the processes $I^1$, $I^2$ and $L$ by \eqref{I-1-def},
\begin{align*}
I_t^2 := \sum_{i=1}^{\infty} \beta^i (W_t^i)^2 + \sum_{i \neq j} \gamma^{ij} \int_0^t W_s^i d W_s^j, \quad t \geq 0
\end{align*}
and \eqref{L-def}. Then we have $L_t \leq 0$ for all $t \geq 0$, by Remark \ref{rem-Wiener-series-1} we obtain \eqref{I-1-identity}, and by Remark \ref{rem-Wiener-series-3} the process $I^2$ coincides with the right-hand side of \eqref{I-2-def}. Therefore, by Lemma \ref{lemma-short-time-0} we deduce that $\alpha = 0$. Furthermore, we have $L_t / t \leq 0$ for all $t > 0$. Since $\frac{1}{t} \int_0^t W_s^i dW_s^j \overset{d}{=} \int_0^1 W_s^i dW_s^j$ for all $i,j \in \bbn$, this implies
\begin{align}\label{inequ-bcd}
\sum_{i=1}^{\infty} \beta^i (W_1^i)^2 + \sum_{i \neq j} \gamma^{ij} \int_0^1 W_s^i d W_s^j + \delta \leq 0.
\end{align}
Let us fix arbitrary $i,j \in \bbn$ with $i \neq j$. Taking conditional expectation with respect to $\sigma( W_s^i, W_s^j : s \geq 0 )$, we obtain
\begin{align*}
\beta^i (W_1^i)^2 + \beta^j (W_1^j)^2 + \gamma^{ij} \int_0^1 W_s^i dW_s^j + \gamma^{ji} \int_0^1 W_s^j dW_s^i + \delta + \sum_{k \neq i,j} |\beta^k|^2 \leq 0.
\end{align*}
Introducing the L\'{e}vy area
\begin{align*}
L^{ij} := \int_0^1 W_s^i dW_s^j - \int_0^1 W_s^j dW_s^i,
\end{align*}
using integration by parts we have
\begin{align*}
\gamma^{ij} \int_0^1 W_s^i dW_s^j + \gamma^{ji} \int_0^1 W_s^j dW_s^i = \frac{\gamma^{ij} + \gamma^{ji}}{2} W_1^i W_1^j + \frac{\gamma^{ij} - \gamma^{ji}}{2} L^{ij}.
\end{align*}
Thus, we obtain $\bbp$-almost surely
\begin{align}\label{inequ-W1}
\beta^i (W_1^i)^2 + \beta^j (W_1^j)^2 + \frac{\gamma^{ij} + \gamma^{ji}}{2} W_1^i W_1^j + \frac{\gamma^{ij} - \gamma^{ji}}{2} L^{ij} + \delta + \sum_{k \neq i,j} |\beta^k|^2 \leq 0.
\end{align}
Let $\epsilon > 0$ be arbitrary, and consider the event $B_{\epsilon} := \{ W_1^i,W_1^j \in (-\epsilon,\epsilon) \}$. Then we have $B_{\epsilon} \in \calf$ with $\bbp(B_{\epsilon}) > 0$. The inequality \eqref{inequ-W1} also holds $\bbp( \,\cdot\, | B_{\epsilon} )$-almost surely. Moreover, also under $\bbp( \,\cdot\, | B_{\epsilon} )$ the distribution of $L^{ij}$ is symmetric and has unbounded support. Therefore, we have $\gamma^{ij} = \gamma^{ji}$. Using Proposition \ref{prop-Riesz-Parseval-2}, this shows that the operators $\gamma$ and $\widetilde{\gamma}$ are self-adjoint. Furthermore, using integration by parts we have
\begin{align*}
W_1^i W_1^j = \int_0^1 W_s^i dW_s^j + \int_0^1 W_s^j dW_s^i
\end{align*}
for all $i,j \in \bbn$ with $i \neq j$. Therefore, the inequality \eqref{inequ-bcd} becomes
\begin{align*}
\sum_{i=1}^{\infty} \beta^i (W_1^i)^2 + \sum_{i < j}^{\infty} \gamma^{ij} W_1^i W_1^j + \delta \leq 0.
\end{align*}
We conclude that $\beta^i \leq 0$ for each $i \in \bbn$, and that $\delta \leq 0$. Moreover, in view of Propositions \ref{prop-Riesz-Parseval-1} and \ref{prop-Riesz-Parseval-2} we have $\widetilde{\gamma} \in L_2(H_0)$, because $\beta \in H_0$, and we have $- \widetilde{\gamma} \in L_2^+(H_0)$, because $\beta^i \leq 0$ for each $i \in \bbn$.
\end{proof}

\newpage

\begin{appendix}

\begin{center}

{\Large{\textbf{APPENDIX}}}

\vspace{0.5cm}

{\Large{\textbf{Geometry, stochastic processes and smooth functions in infinite dimension}}}

\vspace{1.0cm}

\end{center}

In this appendix we provide the required background about several topics related to this paper. More precisely, we present the required results about geometry in Hilbert spaces in Appendix \ref{app-geometry}, about submanifolds with boundary in Appendix \ref{app-submanifolds}, about martingales in Banach spaces in Appendix \ref{app-martingales-Banach}, about smooth functions in Banach spaces in Appendix \ref{app-functions-Banach}, and about linear operators in Hilbert spaces in Appendix \ref{app-operators-Hilbert}.

\section{Geometry in Hilbert spaces}\label{app-geometry}

In this appendix we provide the required results about geometry in Hilbert spaces. We start with various notions of tangent cones, which can also be defined for closed subsets in Banach spaces. Thus, for the beginning of this section let $E$ be a Banach space, and let $\cald \subset E$ be a closed subset. Recall that the distance function $d_{\cald} : E \to \bbr_+$ is defined as
\begin{align*}
d_{\cald}(x) := \inf_{y \in \cald} \| x-y \|, \quad x \in E.
\end{align*}
For $x \in E$ and $r > 0$ we denote by $K(x,r) \subset X$ the closed ball
\begin{align*}
K(x,r) := \{ y \in X : \| y-x \| \leq r \}
\end{align*}
around $x$ with radius $r$. The following auxiliary result will be useful. Its proof is straightforward, and therefore omitted.

\begin{lemma}\label{lemma-distance}
Let $x \in E$ and $r > 0$ be arbitrary. Then we have
\begin{align*}
d_{\cald}(y) = d_{\cald \cap K(x,r)}(y)
\end{align*}
for all $y \in E$ with $\| y-x \| < \frac{r}{2}$.
\end{lemma}

\begin{definition}\label{def-cones}
Let $x \in \cald$ be arbitrary.
\begin{enumerate}
\item The \emph{Clarke tangent cone} to $\cald$ at $x$ is defined as
\begin{align*}
T_{\cald}^c(x) := C_{\cald}(x) := \bigg\{ v \in E : \lim_{t \to 0^+ \atop \cald \ni x' \to x} \frac{d_{\cald}(x'+tv)}{t} = 0 \bigg\}.
\end{align*}

\item The \emph{adjacent cone} (or \emph{intermediate cone}) to $\cald$ at $x$ is defined as
\begin{align*}
T_{\cald}^a(x) := T_{\cald}^{\flat}(x) := \bigg\{ v \in E : \lim_{t \to 0^+} \frac{d_{\cald}(x+tv)}{t} = 0 \bigg\}.
\end{align*}

\item The \emph{Bouligand tangent cone} (or \emph{contingent cone}) to $\cald$ at $x$ is defined as
\begin{align*}
T_{\cald}^b(x) := T_{\cald}(x) := \bigg\{ v \in E : \liminf_{t \to 0^+} \frac{d_{\cald}(x+tv)}{t} = 0 \bigg\}.
\end{align*}

\item The \emph{weak Bouligand tangent cone} (or \emph{weak contingent cone}) to $\cald$ at $x$, denoted by $T_{\cald}^{\sigma}(x)$, consists of all $v \in E$ such that there are sequences $(t_n)_{n \in \bbn} \subset (0,\infty)$ with $t_n \to 0^+$ and $(v_n)_{n \in \bbn} \subset E$ with $v_n \overset{\sigma}{\to} v$ such that $x + t_n v_n \in \cald$ for each $n \in \bbn$.
\end{enumerate}
\end{definition}

Recall that a subset $C \subset E$ is called a \emph{cone} if $\lambda x \in C$ for all $\lambda \geq 0$ and $x \in C$. It is easily checked that for each $x \in \cald$ the sets $T_{\cald}^{c}(x)$, $T_{\cald}^{a}(x)$, $T_{\cald}^b(x)$ and $T_{\cald}^{\sigma}(x)$ are indeed cones.

\begin{remark}
There are various notations of these cones in the literature. The notations $T_{\cald}^c(x)$ and $T_{\cald}^b(x)$ are used in \cite{Colombo}, whereas the notations $C_{\cald}(x)$, $T_{\cald}^{\flat}(x)$ and $T_{\cald}(x)$ are used in \cite{Aubin_Fr}. The notation $T_{\cald}^{\sigma}(x)$ is used in \cite{Colombo} and \cite{Aubin_Fr}.
\end{remark}

The following result provides characterizations of the elements of $T_{\cald}^{c}(x)$, $T_{\cald}^{a}(x)$ and $T_{\cald}^b(x)$.

\begin{proposition}\cite[p. 122, 128]{Aubin_Fr}\label{prop-cones}
Let $x \in \cald$ and $v \in E$ be arbitrary.
\begin{enumerate}
\item We have $v \in T_{\cald}^c(x)$ if and only if for every sequence $(t_n)_{n \in \bbn} \subset (0,\infty)$ with $t_n \to 0^+$ and every sequence $(x_n)_{n \in \bbn} \subset \cald$ with $x_n \to x$ there exists a sequence $(v_n)_{n \in \bbn} \subset E$ with $v_n \to v$ such that $x_n + t_n v_n \in \cald$ for each $n \in \bbn$.

\item We have $v \in T_{\cald}^{a}(x)$ if and only if for every sequence $(t_n)_{n \in \bbn} \subset (0,\infty)$ with $t_n \to 0^+$ there exists a sequence $(v_n)_{n \in \bbn} \subset E$ with $v_n \to v$ such that $x + t_n v_n \in \cald$ for each $n \in \bbn$.

\item We have $v \in T_{\cald}^b(x)$ if and only if there are sequences $(t_n)_{n \in \bbn} \subset (0,\infty)$ with $t_n \to 0^+$ and $(v_n)_{n \in \bbn} \subset E$ with $v_n \to v$ such that $x + t_n v_n \in \cald$ for each $n \in \bbn$.
\end{enumerate}
\end{proposition}

\begin{lemma}\label{lemma-cones-inclusions}
Let $x \in \cald$ be arbitrary. Then the following statements are true:
\begin{enumerate}
\item We have the inclusions $T_{\cald}^{c}(x) \subset T_{\cald}^{a}(x) \subset T_{\cald}^b(x) \subset T_{\cald}^{\sigma}(x)$.

\item If $\dim E < \infty$, then we have $T_{\cald}^b(x) = T_{\cald}^{\sigma}(x)$.
\end{enumerate}
\end{lemma}

\begin{proof}
This is an immediate consequence of Proposition \ref{prop-cones}.
\end{proof}

\begin{lemma}\label{lemma-tangent-cones}
For each $x \in \cald$ the following statements are true:
\begin{enumerate}
\item $T_{\cald}^c(x)$ is a closed convex cone.

\item $T_{\cald}^{a}(x)$ and $T_{\cald}^b(x)$ are closed cones.
\end{enumerate}
\end{lemma}

\begin{proof}
By \cite[Prop. 4.1.6]{Aubin_Fr} the cone $T_{\cald}^c(x)$ is closed and convex. Furthermore, by \cite[p. 127, 121]{Aubin_Fr} the cones $T_{\cald}^{a}(x)$ and $T_{\cald}^b(x)$ are closed.
\end{proof}

The following result is easily verified. Therefore, we omit the proof.

\begin{lemma}
For each $x \in \cald$ and each $v \in E$ the following statements are true:
\begin{enumerate}
\item We have $v \in T_{\cald}^b(x)$ if and only if there are sequences $(t_n)_{n \in \bbn} \subset (0,\infty)$ with $t_n \to 0^+$ and $(x_n)_{n \in \bbn} \subset \cald$ such that
\begin{align*}
v = \lim_{n \to \infty} \frac{x_n - x}{t_n}.
\end{align*}

\item We have $v \in T_{\cald}^{\sigma}(x)$ if and only if there are sequences $(t_n)_{n \in \bbn} \subset (0,\infty)$ with $t_n \to 0^+$ and $(x_n)_{n \in \bbn} \subset \cald$ such that
\begin{align*}
v = \sigma \text{-} \lim_{n \to \infty} \frac{x_n - x}{t_n}.
\end{align*}
\end{enumerate}
\end{lemma}

The following result shows that all these tangent cones are local objects.

\begin{lemma}\label{lemma-tangent-local}
Let $x \in \cald$ be arbitrary. Then for any closed neighborhood $C \subset E$ of $x$ we have $T_{\cald}^c(x) = T_{\cald \cap C}^c(x)$, $T_{\cald}^a(x) = T_{\cald \cap C}^a(x)$, $T_{\cald}^b(x) = T_{\cald \cap C}^b(x)$ and $T_{\cald}^{\sigma}(x) = T_{\cald \cap C}^{\sigma}(x)$.
\end{lemma}

\begin{proof}
The proof is a consequence of Proposition \ref{prop-cones}. For the tangent cones $T_{\cald}^c(x)$, $T_{\cald}^a(x)$ and $T_{\cald}^b(x)$, the statement also follows from Lemma \ref{lemma-distance}.
\end{proof}

The following is a rotated version of the example on page 161 in \cite{Aubin_Fr}. It illustrates that the inclusions in Lemma \ref{lemma-cones-inclusions} can be strict.

\begin{example}
Consider the state space $E = \bbr^2$ and the closed subset
\begin{align*}
\cald = \Big\{ \Big( \frac{1}{n},0 \Big) : n \in \bbn \Big\} \cup \{ (0,y) : y \in \bbr_+ \}.
\end{align*}
Then we have
\begin{align*}
T_{\cald}^c(0) &= \{ 0 \},
\\ T_{\cald}^a(0) &= \{ (0,y) : y \in \bbr_+ \},
\\ T_{\cald}^b(0) &= T_{\cald}^{\sigma}(0) = \{ (x,0) : x \in \bbr_+ \} \cup \{ (0,y) : y \in \bbr_+ \}.
\end{align*}
\end{example}

The following example can be found in \cite[Example 7.1.c]{Colombo}. It shows that the inclusion $T_{\cald}^b(x) \subset T_{\cald}^{\sigma}(x)$ can be strict in infinite dimension.

\begin{example}
Suppose that $E$ is an infinite dimensional separable Hilbert space with orthonormal basis $\{ e_n : n \in \bbn_0 \}$. We define the closed subset
\begin{align*}
\cald := \{ 0 \} \cup \bigg\{ \frac{e_n + e_0}{n} : n \in \bbn \bigg\}.
\end{align*}
Then we have
\begin{align*}
T_{\cald}^b(0) = \{ 0 \} \quad \text{and} \quad T_{\cald}^{\sigma}(0) = \{ \lambda e_0 : \lambda \geq 0 \}.
\end{align*}
To illustrate this, take $v := e_0$. We choose $t_n := \frac{1}{n}$ and $v_n := e_n + e_0$ for each $n \in \bbn$. Then we have $t_n v_n \in \cald$ for each $n \in \bbn$, and by Parseval's identity we have $v_n \overset{\sigma}{\to} v$, but of course $\| v_n - v \| \not \to 0$.
\end{example}

In what follows, let $\cald \subset H$ be a closed subset of a Hilbert space $H$.

\begin{definition}\label{def-polar}
For a subset $A \subset H$ the \emph{polar} of $A$ is defined as
\begin{align*}
A^{\circ} := \{ y \in H : \la y,x \ra \leq 1 \text{ for all } x \in A \}.
\end{align*}
\end{definition}

Note that for two subsets $A_1,A_2 \subset H$ with $A_1 \subset A_2$ we have $A_2^{\circ} \subset A_1^{\circ}$.

\begin{remark}\label{rem-polar-cone}
For a cone $C \subset H$ we have
\begin{align*}
C^{\circ} = \{ y \in H : \la y,x \ra \leq 0 \text{ for all } x \in C \} = \bigcap_{x \in C} \{ y \in H : \la y,x \ra \leq 0 \},
\end{align*}
showing that $C^{\circ}$ is a closed convex cone; the so-called \emph{polar cone} of $C$.
\end{remark}

\begin{definition}
Let $A \subset H$ be a subset.
\begin{enumerate}
\item The \emph{convex hull} ${\rm co} \, A$ consists of all $x \in H$ of the form $x = \sum_{i=1}^n \lambda_i x_i$ with $x_1,\ldots,x_n \in A$ and $\lambda_1,\ldots,\lambda_n \geq 0$ for some $n \in \bbn$ such that $\sum_{i=1}^n \lambda_i = 1$.

\item The \emph{closed convex hull} $\overline{\rm co} \, A$ is the smallest closed convex set including $A$.
\end{enumerate}
\end{definition}

\begin{remark}
According to \cite[Lemma 5.27 (6)]{Aliprantis-Border} the closed convex hull is the closure of ${\rm co} \, A$; that is $\overline{\rm co} \, A = \overline{{\rm co} \, A}$.
\end{remark}

\begin{lemma}\label{lemma-double-polar}
For a cone $C \subset H$ the following statements are true:
\begin{enumerate}
\item We have $C \subset C^{\circ \circ}$.

\item If $C$ is a closed convex cone, then we have $C = C^{\circ \circ}$.
\end{enumerate}
\end{lemma}

\begin{proof}
The first statement is a consequence of the Bipolar Theorem; see \cite[Thm. 5.103 (2)]{Aliprantis-Border}. For the second statement we have to show the converse inclusion $C^{\circ \circ} \subset C$. Let $x \in H$ with $x \notin C$ be arbitrary. By the geometric Hahn-Banach theorem there exists $y \in H$ such that
\begin{align*}
\la y,x \ra > \sup \{ \la y,z \ra : z \in C \}.
\end{align*}
Since $C$ is a cone, we have $\sup \{ \la y,z \ra : y \in C \} \leq 0$, and hence
\begin{align*}
\la y,x \ra > 0 \geq \la y,z \ra \quad \quad \text{for all $z \in C$.}
\end{align*}
Now we obtain $y \in C^{\circ}$, and hence $x \notin C^{\circ \circ}$.
\end{proof}

\begin{definition}
Let $x \in \cald$ be arbitrary.
\begin{enumerate}
\item The \emph{Clarke normal cone} to $\cald$ at $x$ is defined as $\caln_{\cald}^{c}(x) := T_{\cald}^{c}(x)^{\circ}$.

\item The \emph{adjacent normal cone} (or \emph{intermediate normal cone}) to $\cald$ at $x$ is defined as $\caln_{\cald}^{a}(x) := T_{\cald}^{a}(x)^{\circ}$.

\item The \emph{Bouligand normal cone} (or \emph{contingent normal cone}) to $\cald$ at $x$ is defined as $\caln_{\cald}^b(x) := T_{\cald}^b(x)^{\circ}$.

\item The \emph{weak Bouligand normal cone} (or \emph{weak contingent normal cone}) to $\cald$ at $x$ is defined as $\caln_{\cald}^{\sigma}(x) := T_{\cald}^{\sigma}(x)^{\circ}$.
\end{enumerate}
\end{definition}

\begin{lemma}\label{lemma-normal-cones}
For each $x \in \cald$ the following statements are true:
\begin{enumerate}
\item $\caln_{\cald}^{\sigma}(x)$, $\caln_{\cald}^b(x)$, $\caln_{\cald}^{a}(x)$ and $\caln_{\cald}^{c}(x)$ are closed convex cones, and we have
\begin{align*}
\caln_{\cald}^{\sigma}(x) \subset \caln_{\cald}^b(x) \subset \caln_{\cald}^{a}(x) \subset \caln_{\cald}^{c}(x).
\end{align*}
\item We have $T_{\cald}^{c}(x) = \caln_{\cald}^{c}(x)^{\circ}$,  $T_{\cald}^{a}(x) \subset \caln_{\cald}^{a}(x)^{\circ}$, $T_{\cald}^b(x) \subset \caln_{\cald}^b(x)^{\circ}$ and $T_{\cald}^{\sigma}(x) \subset \caln_{\cald}^{\sigma}(x)^{\circ}$.
\end{enumerate}
\end{lemma}

\begin{proof}
The first statement is a consequence of Lemma \ref{lemma-cones-inclusions} and Remark \ref{rem-polar-cone}, and the second statement follows from Lemma \ref{lemma-tangent-cones} and Lemma \ref{lemma-double-polar}.
\end{proof}

The following result shows that all these normal cones are local objects. It is an immediate consequence of Lemma \ref{lemma-tangent-local}.

\begin{lemma}\label{lemma-normal-local-1}
Let $x \in \cald$ be arbitrary. Then for any closed neighborhood $C \subset H$ of $x$ we have $\caln_{\cald}^c(x) = \caln_{\cald \cap C}^c(x)$, $\caln_{\cald}^a(x) = \caln_{\cald \cap C}^a(x)$, $\caln_{\cald}^b(x) = \caln_{\cald \cap C}^b(x)$ and $\caln_{\cald}^{\sigma}(x) = \caln_{\cald \cap C}^{\sigma}(x)$.
\end{lemma}

\begin{definition}\label{def-normal-cones}
Let $x \in \cald$ be arbitrary.
\begin{enumerate}
\item An element $u \in H$ is called a \emph{first order Fr\'{e}chet normal} to $\cald$ at $x$ if for each $\epsilon > 0$ there exists $\delta > 0$ such that
\begin{align}\label{normal-order-1}
\la u,y-x \ra \leq \epsilon \| y-x \|
\end{align}
for all $y \in \cald$ with $\| y-x \| < \delta$.

\item The \emph{first order normal cone} to $\cald$ at $x$, denoted by $\caln_{\cald}^1(x)$, is defined as the set of all first order Fr\'{e}chet normals to $\cald$ at $x$.

\item A pair $(u,v) \in H \times L(H)$ with a self-adjoint operator $v$ is called a \emph{second order Fr\'{e}chet normal} to $\cald$ at $x$ if for each $\epsilon > 0$ there exists $\delta > 0$ such that
\begin{align}\label{normal-order-2}
\la u,y-x \ra + \frac{1}{2} \la v(y-x), y-x \ra \leq \epsilon \| y-x \|^2
\end{align}
for all $y \in \cald$ with $\| y-x \| < \delta$.

\item The \emph{second order normal cone} to $\cald$ at $x$, denoted by $\caln_{\cald}^2(x)$, is defined as the set of all second order Fr\'{e}chet normals to $\cald$ at $x$.
\end{enumerate}
\end{definition}

\begin{remark}
It is easy to check that $\caln_{\cald}^1(x)$ and $\caln_{\cald}^2(x)$ are convex cones for each $x \in \cald$.
\end{remark}

The following result shows that also the just introduced normal cones are also local objects. It is immediately verified by checking the properties of first and second order Fr\'{e}chet normals from Definition \ref{def-normal-cones}.

\begin{lemma}\label{lemma-normal-local-2}
Let $x \in \cald$ be arbitrary. Then for any closed neighborhood $C \subset H$ of $x$ we have $\caln_{\cald}^1(x) = \caln_{\cald \cap C}^1(x)$ and $\caln_{\cald}^2(x) = \caln_{\cald \cap C}^2(x)$.
\end{lemma}

\begin{proposition}\cite[Prop. 3.1]{Borwein}\label{prop-normal-1}
We have $\caln_{\cald}^1(x) = \caln_{\cald}^{\sigma}(x)$ for each $x \in \cald$.
\end{proposition}

\begin{lemma}\label{lemma-cones-1-2}
Let $x \in \cald$ be arbitrary. Then for all $(u,v) \in \caln_{\cald}^2(x)$ we have $u \in \caln_{\cald}^1(x)$.
\end{lemma}

\begin{proof}
Let $\epsilon > 0$ be arbitrary. There is $\delta > 0$ such that \eqref{normal-order-2} for all $y \in \cald$ with $\| y-x \| < \delta$. With the convention $\frac{\epsilon}{0} := \infty$, we define $\eta > 0$ as
\begin{align*}
\eta := \min \bigg\{ \delta, \frac{1}{2}, \frac{\epsilon}{\| v \|} \bigg\}.
\end{align*}
Then for all $y \in \cald$ with $\| y-x \| < \eta$ we obtain
\begin{align*}
\la u,y-x \ra &\leq \epsilon \| y-x \|^2 - \frac{1}{2} \la v(y-x), y-x \ra
\\ &\leq \epsilon \| y-x \|^2 + \frac{1}{2} | \la v(y-x), y-x \ra |
\\ &\leq \epsilon \| y-x \|^2 + \frac{1}{2} \| v \| \| y-x \|^2
\\ &\leq \bigg( \epsilon \eta + \frac{1}{2} \| v \| \eta \bigg) \| y-x \| \leq \epsilon \| y-x \|,
\end{align*}
showing \eqref{normal-order-1}.
\end{proof}

In the following two results we use the conventions from Remark \ref{remark-derivatives}.

\begin{lemma}\label{lemma-phi-in-N1}
Let $\phi : H \to \bbr$ be of class $C^1$, and let $x \in \cald$ be such that $\displaystyle\max_{\mathcal{D}} \phi = \phi(x)$. Then we have
\begin{align}\label{der-in-normal-cone-1}
D \phi(x) \in \caln_{\cald}^1(x).
\end{align}
\end{lemma}

\begin{proof}
By Taylor's theorem (see, e.g. \cite[Thm. 2.4.15]{Abraham}) we have
\begin{align*}
\phi(x+z) = \phi(x) + \la D\phi(x),z \ra + R(z) \| z \|, \quad z \in H,
\end{align*}
where the remainder term $R : H \to \bbr$ is a continuous mapping with $R(0) = 0$. Since $\displaystyle\max_{\mathcal{D}} \phi = \phi(x)$, we obtain
\begin{align*}
\la D\phi(x),y-x \ra &= \phi(y) - \phi(x) - R(y-x) \| y-x \|
\\ &\leq | R(y-x) | \cdot \| y-x \|.
\end{align*}
By the continuity of $R$, this proves \eqref{der-in-normal-cone-1}.
\end{proof}

\begin{lemma}\label{lemma-phi-in-N2}
Let $\phi : H \to \bbr$ be of class $C^2$, and let $x \in \cald$ be such that $\displaystyle\max_{\mathcal{D}} \phi = \phi(x)$. Then we have
\begin{align}\label{der-in-normal-cone}
(D \phi(x), D^2 \phi(x)) \in \caln_{\cald}^2(x).
\end{align}
\end{lemma}

\begin{proof}
By Taylor's theorem (see, e.g. \cite[Thm. 2.4.15]{Abraham}) we have
\begin{align*}
\phi(x+z) = \phi(x) + \la D\phi(x),z \ra + \frac{1}{2} \la D^2 \phi(x)z,z \ra + R(z) \| z \|^2, \quad z \in H,
\end{align*}
where the remainder term $R : H \to \bbr$ is a continuous mapping with $R(0) = 0$. Since $\displaystyle\max_{\mathcal{D}} \phi = \phi(x)$, we obtain
\begin{align*}
\la D\phi(x),y-x \ra + \frac{1}{2} \la D^2 \phi(x)(y-x),y-x \ra &= \phi(y) - \phi(x) - R(y-x) \| y-x \|^2
\\ &\leq | R(y-x) | \cdot \| y-x \|^2.
\end{align*}
By the continuity of $R$, this proves \eqref{der-in-normal-cone}.
\end{proof}

In view of the upcoming definition, note that for each $x \in \cald$ we have
\begin{align*}
d_{\cald}(x+u) \leq \| u \|, \quad u \in H.
\end{align*}

\begin{definition}
Let $x \in \cald$ be arbitrary.
\begin{enumerate}
\item A vector $u \in H$ is called a \emph{proximal normal} to $\cald$ at $x$ if $d_{\cald}(x+u) = \| u \|$.

\item We denote by $\caln_{\cald}^{1,\pr}(x)$ the set of all proximal normals to $\cald$ at $x$; that is
\begin{align*}
\caln_{\cald}^{1,\pr}(x) := \{ u \in H : d_{\cald}(x+u) = \| u \| \}.
\end{align*}

\item We denote by $\caln_{\cald}^p(x)$ the cone generated by all proximal normals to $\cald$ at $x$; that is
\begin{align*}
\caln_{\cald}^p(x) := \{ \lambda u : \lambda \geq 0 \text{ and } u \in \caln_{\cald}^{1,\pr}(x) \}.
\end{align*}
\end{enumerate}
\end{definition}

\begin{lemma}\label{lemma-prox-normal}
For each $x \in \cald$ the following statements are true:
\begin{enumerate}
\item We have $\mathcal{N}^{1,\pr}_{\mathcal{D}}(x) \subset \caln_{\cald}^p(x)$.

\item For each $u \in \caln_{\cald}^{1,\pr}(x)$ and each $t \in [0,1]$ we have $tu \in \caln_{\cald}^{1,\pr}(x)$.

\item The set $\caln_{\cald}^p(x)$ is a cone having the representations
\begin{align*}
\caln_{\cald}^p(x) &= \{ u \in H : tu \in \caln_{\cald}^{1,\pr}(x) \text{ for some } t > 0 \}
\\ &= \{ u \in H : d_{\cald}(x+tu) = t \| u \| \text{ for some } t > 0 \}.
\end{align*}
Moreover, for all $u \in \caln_{\cald}^p(x)$ and $t > 0$ such that $d_{\cald}(x+tu) = t \| u \|$ we have
\begin{align*}
d_{\cald}(x+su) = s \| u \|, \quad s \in [0,t].
\end{align*}
\end{enumerate}
\end{lemma}

\begin{proof}
The first statement is obvious. For the second statement, suppose there is $t \in (0,1)$ such that $d_{\cald}(x+tu) < t \| u \|$. Then we arrive at the contradiction
\begin{align*}
d_{\cald}(x+u) \leq d_{\cald}(x+tu) + (1-t) \| u \| < \| u \|.
\end{align*}
Now, the third statement is an immediate consequence.
\end{proof}

\begin{proposition}\label{prop-normal-2}
For all $x \in \cald$ and $u \in H$ the following statements are equivalent:
\begin{enumerate}
\item[(i)] We have $u \in \caln_{\cald}^p(x)$.

\item[(ii)] There exists $t = t(x,u) > 0$ such that
\begin{align}\label{inner-prod-normal-prox}
\langle u,y-x \rangle \leq \frac{1}{2t} \| y-x \|^2 \quad \text{for all $y \in \cald$.}
\end{align}
\end{enumerate}
\end{proposition}

\begin{proof}
Let $t > 0$ be arbitrary. Then we have $d_{\cald}(x + tu) = t \| u \|$ if and only if $d_{\cald}(x + tu) \geq t \| u \|$, and this is equivalent to
\begin{align}\label{dist-normal-prox}
\| tu \| \leq \| (x+tu) - y \| \quad \text{for all $y \in \cald$.}
\end{align}
Furthermore, for each $y \in \cald$ we have
\begin{align*}
\| (x+tu) - y \|^2 = \| (x-y) + tu \|^2 = \| x-y \|^2 + 2t \la x-y, u \ra + t^2 \| u \|^2
\end{align*}
Therefore, we have \eqref{dist-normal-prox} if and only if we have \eqref{inner-prod-normal-prox}. Taking into account Lemma \ref{lemma-prox-normal}, this completes the proof.
\end{proof}

For the upcoming result recall that the adjoint $A^*$ of a linear isomorphism $A \in L(H)$ is also invertible with $(A^*)^{-1} = (A^{-1})^*$.

\begin{lemma}\label{lemma-normal-cones-transformation}
Let $A \in L(H)$ be a linear isomorphism. Then we have
\begin{align*}
\caln_{A \cald}^p(Ax) = (A^*)^{-1} \caln_{\cald}^p(x) \quad \text{for each $x \in \cald$.}
\end{align*}
\end{lemma}

\begin{proof}
We fix an arbitrary $x \in \cald$. Let $u \in \caln_{A \cald}^p(Ax)$ be arbitrary. By Proposition \ref{prop-normal-2} there exists $t = t(x,u) > 0$ such that
\begin{align}\label{inner-prod-n-1}
\la u, A(y-x) \ra \leq \frac{1}{2t} \| A(y-x) \|^2 \quad \text{for all $y \in \cald$.}
\end{align}
Defining $s = s(x,u) > 0$ as $s := t / \| A \|^2$, we obtain
\begin{align}\label{inner-prod-n-2}
\la A^* u, y-x \ra \leq \frac{1}{2s} \| y-x \|^2 \quad \text{for all $y \in \cald$,}
\end{align}
and hence, by Proposition \ref{prop-normal-2} we have $u \in (A^*)^{-1} \caln_{\cald}^p(x)$.

Conversely, let $u \in (A^*)^{-1} \caln_{\cald}^p(x)$ be arbitrary. By Proposition \ref{prop-normal-2} there exists $s = s(x,u) > 0$ such that \eqref{inner-prod-n-2} is fulfilled. Defining $t = t(x,u) > 0$ as $t := s / \| A^{-1} \|^2$, we obtain \eqref{inner-prod-n-1}. Thus, by Proposition \ref{prop-normal-2} we have $u \in \caln_{A \cald}^p(Ax)$.
\end{proof}

\begin{lemma}\label{lemma-N-p-1}
We have $\caln_{\cald}^p(x) \subset \caln_{\cald}^1(x)$ for each $x \in \cald$.
\end{lemma}

\begin{proof}
Let $u \in \caln_{\cald}^p(x)$ be arbitrary. By Proposition \ref{prop-normal-2} there exists $t = t(x,u) > 0$ such that \eqref{inner-prod-normal-prox} is satisfied. Let $\epsilon > 0$ be arbitrary. Setting $\delta := 2t \epsilon$, we obtain
\begin{align*}
\la u,y-x \ra \leq \frac{1}{2t} \| y-x \|^2 \leq \frac{\delta}{2t} \| y-x \| = \epsilon \| y-x \|
\end{align*}
for all $y \in \cald$ with $\| x-y \| < \delta$, showing \eqref{normal-order-1}.
\end{proof}

\begin{proposition}\label{prop-normal-cones-inclusion}
For each $x \in \cald$ the following statements are true:
\begin{enumerate}
\item We have the inclusions
\begin{align*}
\mathcal{N}^{1,\pr}_{\mathcal{D}}(x) \subset \caln_{\cald}^p(x) \subset \caln_{\cald}^{1}(x) = \caln_{\cald}^{\sigma}(x) \subset \caln_{\cald}^b(x) \subset \caln_{\cald}^{a}(x) \subset \caln_{\cald}^{c}(x).
\end{align*}
\item We have $t u \in \caln_{\cald}^{1,\pr}(x)$ for all $u \in \caln_{\cald}^{1,\pr}(x)$ and $t \in [0,1]$.

\item $\caln_{\cald}^p(x)$ is a cone.

\item $\caln_{\cald}^{1}(x), \caln_{\cald}^{\sigma}(x), \caln_{\cald}^b(x), \caln_{\cald}^{a}(x)$ and $\caln_{\cald}^{c}(x)$ are closed convex cones.
\end{enumerate}
\end{proposition}

\begin{proof}
This is a consequence of Lemma \ref{lemma-normal-cones}, Proposition \ref{prop-normal-1}, Lemma \ref{lemma-prox-normal} and Lemma \ref{lemma-N-p-1}.
\end{proof}

We refer to Example \ref{example-graph} below for an example which shows that the inclusion $\caln_{\cald}^p(x) \subset \caln_{\cald}^{1}(x)$ can be strict.

\begin{proposition}\label{prop-normal-local}
Let $x \in \cald$ be arbitrary, and let $C \subset H$ be a closed neighborhood of $x$. Then we have $\caln_{\cald}^p(x) = \caln_{\cald \cap C}^p(x)$, $\caln_{\cald}^1(x) = \caln_{\cald \cap C}^1(x)$, $\caln_{\cald}^{\sigma}(x) = \caln_{\cald \cap C}^{\sigma}(x)$, $\caln_{\cald}^b(x) = \caln_{\cald \cap C}^b(x)$, $\caln_{\cald}^a(x) = \caln_{\cald \cap C}^a(x)$ and $\caln_{\cald}^c(x) = \caln_{\cald \cap C}^c(x)$.
\end{proposition}

\begin{proof}
In view of Lemma \ref{lemma-normal-local-1} and Lemma \ref{lemma-normal-local-2}, we only have to show $\caln_{\cald}^p(x) = \caln_{\cald \cap C}^p(x)$, which is a consequence of Lemma \ref{lemma-distance} and Lemma \ref{lemma-prox-normal}.
\end{proof}

\begin{definition}
For each $x \in \cald$ we define $R_{\cald}(x)$ as the set consisting of all $u \in H$ such that there are sequences $(x_n)_{n \in \bbn} \subset \cald$ and $(u_n)_{n \in \bbn} \subset H$ with $u_n \in \caln_{\cald}^p(x_n)$ for each $n \in \bbn$ as well as $x_n \to x$ and $u_n \overset{\sigma}{\to} u$.
\end{definition}

\begin{proposition}\label{prop-normal-cone-limits}
We have $\caln_{\cald}^c(x) = \overline{\rm co} \, R_{\cald}(x)$ for each $x \in \cald$.
\end{proposition}

\begin{proof}
This follows from the proximal normal formula in Hilbert spaces (see \cite[Prop. 3.1 and Thm. 3.7]{Loewen}).
\end{proof}

\begin{proposition}\label{prop-C-N}
Let $C : \cald \to L(H)$ be a continuous mapping such that $C(x)$ is self-adjoint for each $x \in H$. Then the following statements are equivalent:
\begin{enumerate}
\item[(i)] We have $C(x)u = 0$ for all $x \in \cald$ and $u \in \caln_{\cald}^{1,\pr}(x)$.

\item[(ii)] We have $C(x)u = 0$ for all $x \in \cald$ and $u \in \caln_{\cald}^{c}(x)$.
\end{enumerate}
\end{proposition}

\begin{proof}
By Proposition \ref{prop-normal-cones-inclusion} we only need to show (i) $\Rightarrow$ (ii). Let $x \in \cald$ be arbitrary. Furthermore, let $u \in \caln_{\cald}^{p}(x)$ be arbitrary. Then there are $\lambda \geq 0$ and $v \in \caln_{\cald}^{1,\pr}(x)$ such that $u = \lambda v$, and we obtain
\begin{align*}
C(x) u = C(x)(\lambda v) = \lambda C(x)v = 0.
\end{align*}
Now, assume that $u \in R_{\cald}(x)$. Then there are sequences $(x_n)_{n \in \bbn} \subset \cald$ and $(u_n)_{n \in \bbn} \subset H$  such that $u_n \in \caln_{\cald}^p(x_n)$ for all $n \in \bbn$ and we have $x_n \to x$ as well as $u_n \overset{\sigma}{\to} u$. Furthermore, let $v \in H$ be arbitrary. Then we have
\begin{align*}
&| \langle C(x_n) u_n - C(x) u, v \rangle | = | \langle u_n, C(x_n)v \rangle - \langle u, C(x) v \rangle |
\\ &\leq | \langle u_n, (C(x_n)-C(x))v \rangle | + |\langle u_n-u,C(x)v \rangle|
\\ &\leq \| u_n \| \, \| (C(x_n) - C(x))v \| + |\langle u_n-u,C(x)v \rangle| \to 0,
\end{align*}
because the sequence $(u_n)_{n \in \bbn}$ is bounded by the uniform boundedness principle. Therefore, we obtain
\begin{align*}
\langle C(x)u,v \rangle = \lim_{n \to \infty} \langle C(x_n)u_n,v \rangle = 0.
\end{align*}
Since $v \in H$ was arbitrary, we deduce that $C(x) u = 0$. Next, let $u \in {\rm co} \, R_{\cald}(x)$ be arbitrary. Then we have $u = \sum_{i=1}^n \lambda_i u_i$ with $u_1,\ldots,u_n \in R_{\cald}(x)$ and $\lambda_1,\ldots,\lambda_n \geq 0$ such that $\sum_{i=1}^n \lambda_i = 1$. Therefore, we obtain
\begin{align*}
C(x)u = C(x) \bigg( \sum_{i=1}^n \lambda_i u_i \bigg) = \sum_{i=1}^n \lambda_i C(x) u_i = 0.
\end{align*}
Now, let $u \in \caln_{\cald}^{c}(x)$ be arbitrary. By Proposition \ref{prop-normal-cone-limits} there is a sequence $(u_n)_{n \in \bbn} \subset {\rm co} \, R_{\cald}(x)$ such that $u_n \to u$, and we obtain
\begin{align*}
C(x)u = \lim_{n \to \infty} C(x) u_n = 0,
\end{align*}
completing the proof.
\end{proof}

\begin{proposition}\label{prop-normal-inward-a-cont}
Let $a : \cald \to H$ be a mapping such that for all elements $x \in \cald$ and $u \in H$, and all sequences $(x_n)_{n \in \bbn} \subset \cald$ and $(u_n)_{n \in \bbn} \subset H$  such that $u_n \in \caln_{\cald}^p(x_n)$ for all $n \in \bbn$ and $x_n \to x$ as well as $u_n \overset{\sigma}{\to} u$ it follows that $\la u_n,a(x_n) \ra \to \la u,a(x) \ra$. Then the following statements are equivalent:
\begin{enumerate}
\item[(i)] We have $\la u,a(x) \ra \leq 0$ for all $x \in \cald$ and $u \in \caln_{\cald}^{1,\pr}(x)$.

\item[(ii)] We have $\la u,a(x) \ra \leq 0$ for all $x \in \cald$ and $u \in \caln_{\cald}^{c}(x)$.

\item[(iii)] We have $a(x) \in T_{\cald}^c(x)$ for all $x \in \cald$.

\item[(iv)] We have $a(x) \in T_{\cald}^{\sigma}(x)$ for all $x \in \cald$.
\end{enumerate}
\end{proposition}

\begin{proof}
The equivalence (ii) $\Leftrightarrow$ (iii) and the implication (iv) $\Rightarrow$ (i) are a consequence of Lemma \ref{lemma-normal-cones}, and the implication (iii) $\Rightarrow$ (iv) follows from Lemma \ref{lemma-cones-inclusions}. Thus, it remains to show (i) $\Rightarrow$ (ii): Let $x \in \cald$ be arbitrary. Furthermore, let $u \in \caln_{\cald}^{p}(x)$ be arbitrary. Then there are $\lambda \geq 0$ and $v \in \caln_{\cald}^{1,\pr}(x)$ such that $u = \lambda v$, and we obtain
\begin{align*}
\la u,a(x) \ra = \la \lambda v,a(x) \ra = \lambda \la v,a(x) \ra \leq 0.
\end{align*}
Now, assume that $u \in R_{\cald}(x)$. Then there are sequences $(x_n)_{n \in \bbn} \subset \cald$ and $(u_n)_{n \in \bbn} \subset H$  such that $u_n \in \caln_{\cald}^p(x_n)$ for all $n \in \bbn$ and we have $x_n \to x$ as well as $u_n \overset{\sigma}{\to} u$. By assumption, we obtain
\begin{align*}
\langle u,a(x) \rangle = \lim_{n \to \infty} \langle u_n,a(x_n) \rangle \leq 0.
\end{align*}
Next, let $u \in {\rm co} \, R_{\cald}(x)$ be arbitrary. Then we have $u = \sum_{i=1}^n \lambda_i u_i$ with $u_1,\ldots,u_n \in R_{\cald}(x)$ and $\lambda_1,\ldots,\lambda_n \geq 0$ such that $\sum_{i=1}^n \lambda_i = 1$. Therefore, we obtain
\begin{align*}
\la u,a(x) \ra = \bigg\la \sum_{i=1}^n \lambda_i u_i, a(x) \bigg\ra = \sum_{i=1}^n \lambda_i \la u_i, a(x) \ra \leq 0.
\end{align*}
Now, let $u \in \caln_{\cald}^{c}(x)$ be arbitrary. By Proposition \ref{prop-normal-cone-limits} there is a sequence $(u_n)_{n \in \bbn} \subset {\rm co} \, R_{\cald}(x)$ such that $u_n \to u$, and we obtain
\begin{align*}
\la u,a(x) \ra = \lim_{n \to \infty} \la u_n,a(x) \ra \leq 0,
\end{align*}
completing the proof.
\end{proof}

\begin{remark}\label{rem-more-general-inward-a}
Note that Proposition \ref{prop-normal-inward-a-cont} in particular applies if the mapping $a : \cald \to H$ is continuous. Indeed, then for elements $x \in \cald$ and $u \in H$, and sequences $(x_n)_{n \in \bbn} \subset \cald$ and $(u_n)_{n \in \bbn} \subset H$ such that $x_n \to x$ as well as $u_n \overset{\sigma}{\to} u$ we have
\begin{align*}
| \la u_n,a(x_n) \ra - \la u,a(x) \ra | &\leq | \la u_n,a(x_n)-a(x) \ra | + | \la u_n-u,a(x) \ra |
\\ &\leq \| u_n \| \, \| a(x_n)-a(x) \| + | \la u_n-u,a(x) \ra | \to 0,
\end{align*}
because the sequence $(u_n)_{n \in \bbn}$ is bounded by the uniform boundedness principle.
\end{remark}

In the following result we denote by $P_C(x)$ the orthogonal projection on the closure of the range of $C(x)$ for each $x \in \cald$.

\begin{lemma}\label{lemma-proj-continuous}
Let $C : H \to L_1(H)$ be a mapping of class $C^1$ such that the mapping $\cald \to L(H)$, $x \mapsto P_C(x)$ is continuous. Let $x \in \cald$, $u \in H$ and $(x_n)_{n \in \bbn} \subset \cald$, $(u_n)_{n \in \bbn} \subset H$ be such that $\ran (C(x))$ is finite dimensional and $x_n \to x$ as well as $u_n \overset{\sigma}{\to} u$. Then we have
\begin{align*}
\tr \big( DC(x_n) P_C(x_n) u_n \big) \to \tr \big( DC(x) P_C(x) u \big) \quad \text{for $n \to \infty$.}
\end{align*}
\end{lemma}

\begin{proof}
For each $n \in \bbn$ we have
\begin{align*}
\| P_C(x_n) u_n - P_C(x) u \| &\leq \| P_C(x_n) u_n - P_C(x) u_n \| + \| P_C(x) u_n - P_C(x) u \|
\\ &\leq \| P_C(x_n) - P_C(x) \|_{L(H)} \| u_n \| + \| P_C(x) (u_n-u) \|.
\end{align*}
Since the weakly convergent sequence $(u_n)_{n \in \bbn}$ is bounded as a consequence of the uniform boundedness principle, by the continuity of the mapping $\cald \to L(H)$, $x \mapsto P_C(x)$ we obtain
\begin{align*}
\| P_C(x_n) - P_C(x) \|_{L(H)} \| u_n \| \to 0 \quad \text{for $n \to \infty$.}
\end{align*}
Now, set $d := \dim \ran(C(x))$ and let $\{ e_1,\ldots,e_d \}$ be an orthonormal basis of $\ran(C(x))$. Then we have
\begin{align*}
P_C(x) v = \sum_{k=1}^d \la v, e_k \ra e_k \quad \text{for all $v \in H$.}
\end{align*}
We define the sequence $(v_n)_{n \in \bbn} \subset H$ as $v_n := u_n - u$ for each $n \in \bbn$. Then we have $v_n \overset{\sigma}{\to} 0$, and hence
\begin{align*}
\| P_C(x) v_n \| = \sum_{k=1}^d \la v_n, e_k \ra e_k \to 0 \quad \text{for $n \to \infty$.}
\end{align*}
Consequently, we obtain
\begin{align*}
P_C(x_n) u_n \to P_C(x) u \quad \text{for $n \to \infty$.}
\end{align*}
Now, we introduce the continuous bilinear operator
\begin{align*}
\Phi : L(H,L_1(H)) \times H \to L_1(H), \quad \Phi(T,x) := Tx.
\end{align*}
Since the trace is a continuous linear functional on $L_1(H)$ due to Lemma \ref{lemma-trace-functional}, we arrive at
\begin{align*}
\tr \big( \Phi(D C(x_n), P_C(x_n)u_n ) \big) \to \tr \big( \Phi(D C(x), P_C(x)u ) \big) \quad \text{for $n \to \infty$,}
\end{align*}
completing the proof.
\end{proof}

\begin{proposition}\label{prop-normal-inward-b-cont}
Let $b : \cald \to H$ and $c : \cald \to \bbr$ be continuous mappings. Then the following statements are equivalent:
\begin{enumerate}
\item[(i)] We have $\la u,b(x) \ra + c(x) \leq 0$ for all $x \in \cald$ and $u \in \caln_{\cald}^{1,\pr}(x)$.

\item[(ii)] We have $\la u,b(x) \ra + c(x) \leq 0$ for all $x \in \cald$ and $u \in \caln_{\cald}^{c}(x)$.
\end{enumerate}
\end{proposition}

\begin{proof}
By Proposition \ref{prop-normal-cones-inclusion} we only need to show (i) $\Rightarrow$ (ii). Let $x \in \cald$ be arbitrary. Furthermore, let $u \in \caln_{\cald}^{p}(x)$ be arbitrary. Then there are $\lambda \geq 0$ and $v \in \caln_{\cald}^{1,\pr}(x)$ such that $u = \lambda v$, and we obtain
\begin{align*}
\la u,b(x) \ra + c(x) = \la \lambda v,b(x) \ra + c(x) = \lambda \la v,b(x) \ra + c(x) \leq 0.
\end{align*}
Now, assume that $u \in R_{\cald}(x)$. Then there are sequences $(x_n)_{n \in \bbn} \subset \cald$ and $(u_n)_{n \in \bbn} \subset H$  such that $u_n \in \caln_{\cald}^p(x_n)$ for all $n \in \bbn$ and we have $x_n \to x$ as well as $u_n \overset{\sigma}{\to} u$. By the continuity of $b$ and $c$, we obtain
\begin{align*}
\langle u,b(x) \rangle + c(x) = \lim_{n \to \infty} \big( \langle u_n,b(x_n) \rangle + c(x_n) \big) \leq 0.
\end{align*}
Indeed, we have
\begin{align*}
| \la u_n,b(x_n) \ra - \la u,b(x) \ra | &\leq | \la u_n,b(x_n)-b(x) \ra | + | \la u_n-u,b(x) \ra |
\\ &\leq \| u_n \| \, \| b(x_n)-b(x) \| + | \la u_n-u,b(x) \ra | \to 0,
\end{align*}
because the sequence $(u_n)_{n \in \bbn}$ is bounded by the uniform boundedness principle.

Next, let $u \in {\rm co} \, R_{\cald}(x)$ be arbitrary. Then we have $$u = \sum_{i=1}^n \lambda_i u_i$$ with $u_1,\ldots,u_n \in R_{\cald}(x)$ and $\lambda_1,\ldots,\lambda_n \geq 0$ such that $\sum_{i=1}^n \lambda_i = 1$. Therefore, we obtain
\begin{align*}
\la u,b(x) \ra + c(x) = \bigg\la \sum_{i=1}^n \lambda_i u_i, b(x) \bigg\ra + c(x) = \sum_{i=1}^n \lambda_i \big( \la u_i, b(x) \ra + c(x) \big) \leq 0.
\end{align*}
Now, let $u \in \caln_{\cald}^{c}(x)$ be arbitrary. By Proposition \ref{prop-normal-cone-limits} there is a sequence $(u_n)_{n \in \bbn} \subset {\rm co} \, R_{\cald}(x)$ such that $u_n \to u$, and we obtain
\begin{align*}
\la u,b(x) \ra + c(x) = \lim_{n \to \infty} \big( \la u_n,b(x) \ra + c(x) \big) \leq 0.
\end{align*}
This completes the proof.
\end{proof}

\begin{proposition}\label{prop-cond-vol}
Let $H_0$ be another Hilbert space, let $R \in L(H,H_0)$ be a linear isomorphism, and let $\sigma : \cald \to L(H_0,H)$ be a continuous mapping. We define $\Sigma : \cald \to L(H)$ as $\Sigma(x) := \sigma(x) R$ for each $x \in \cald$, and $C : \cald \to L(H)$ as $C(x) := \Sigma(x) \Sigma(x)^*$ for each $x \in \cald$. Then the following statements are equivalent:
\begin{enumerate}
\item[(i)] We have $C(x)u = 0$ for all $x \in \cald$ and $u \in \caln_{\cald}^{1,\pr}(x)$.

\item[(ii)] We have $C(x)u = 0$ for all $x \in \cald$ and $u \in \caln_{\cald}^{c}(x)$.

\item[(iii)] We have $\Sigma(x)^* u = 0$ for all $x \in \cald$ and $u \in \caln_{\cald}^{1,\pr}(x)$.

\item[(iv)] We have $\Sigma(x)^* u = 0$ for all $x \in \cald$ and $u \in \caln_{\cald}^{c}(x)$.

\item[(v)] We have $\sigma(x)^* u = 0$ for all $x \in \cald$ and $u \in \caln_{\cald}^{1,\pr}(x)$.

\item[(vi)] We have $\sigma(x)^* u = 0$ for all $x \in \cald$ and $u \in \caln_{\cald}^{c}(x)$.

\item[(vii)] We have $\sigma(x)w \in T_{\cald}^c(x)$ for all $x \in \cald$ and $w \in H_0$.

\item[(viii)] We have $\sigma(x)w \in T_{\cald}^{\sigma}(x)$ for all $x \in \cald$ and $w \in H_0$.

\item[(ix)] We have $\Sigma(x) v \in T_{\cald}^c(x)$ for all $x \in \cald$ and $v \in H$.

\item[(x)] We have $\Sigma(x) v \in T_{\cald}^{\sigma}(x)$ for all $x \in \cald$ and $v \in H$.
\end{enumerate}
\end{proposition}

\begin{proof}
(i) $\Leftrightarrow$ (ii) $\Leftrightarrow$ (iii) $\Leftrightarrow$ (iv) $\Leftrightarrow$ (v) $\Leftrightarrow$ (vi): Note that $R^*$ is also an isomorphism. Taking also into account Lemma \ref{lemma-kernels}, for each $x \in \cald$ we have
\begin{align*}
\ker(C(x)) = \ker(\Sigma(x)^*) = \ker ( R^* \sigma(x)^* ) = \ker(\sigma(x)^*),
\end{align*}
and hence, these equivalences are a consequence of Proposition \ref{prop-C-N}.

\noindent(vii) $\Leftrightarrow$ (viii) $\Leftrightarrow$ (ix) $\Leftrightarrow$ (x): Recalling that $R$ is an isomorphism, the stated equivalences are a consequence of Proposition \ref{prop-normal-inward-a-cont}.

\noindent(iv) $\Leftrightarrow$ (ix): Let $x \in \cald$ be arbitrary. By Lemma \ref{lemma-adjoint-ker-ran} we have $$\ker( \Sigma(x)^* ) = \ran( \Sigma(x) )^{\circ}.$$ Therefore we have $\caln_{\cald}^c(x) \subset \ker( \Sigma(x)^* )$ if and only if $\ker( \Sigma(x)^* )^{\circ} \subset \caln_{\cald}^c(x)^{\circ}$,  and by Lemma \ref{lemma-normal-cones} this is equivalent to $\ran( \Sigma(x) ) \subset T_{\cald}^c(x)$.
\end{proof}

Now, let us recall the definition of the contingent curvature. For this purpose, it will be convenient to introduce the notation
\begin{align*}
{\rm Graph}(\caln_{\cald}^b) := \{ (x,u) : x \in \cald \text{ and } u \in \caln_{\cald}^b(x) \}.
\end{align*}

\begin{definition}
For $(x,u) \in {\rm Graph}(\caln_{\cald}^b)$ the \emph{contingent derivative} $D \caln_{\cald}^b(x,u)v$ of $\caln_{\cald}^b$ at $(x,u)$ in direction $v \in H$ is defined as the set of all $\mu \in H$ such that there are sequences $(t_n)_{n \in \bbn} \subset (0,\infty)$ with $t_n \to 0^+$ and $(v_n,\mu_n)_{n \in \bbn} \subset H \times H$ with $(v_n,\mu_n) \to (v,\mu)$ such that $x + t_n v_n \in \cald$ and $u + t_n \mu_n \in \caln_{\cald}^b(x+t_n v_n)$ for all $n \in \bbn$.
\end{definition}

\begin{remark}
Note that $D \caln_{\cald}^b(x,u)v = \emptyset$ for $v \notin T_{\cald}^b(x)$, which is a consequence of Proposition \ref{prop-cones}.
\end{remark}

\begin{definition}\label{def-curvature}
The \emph{contingent curvature} of $\cald$ at $(x,u) \in {\rm Graph}(\caln_{\cald}^b)$ is defined as
\begin{align*}
{\rm Curv}_{\cald}(x,u)(v,w) := \sup \{ \la \mu,w \ra : \mu \in D \caln_{\cald}^b(x,u)v \} \quad \forall v,w \in T_{\cald}^b(x).
\end{align*}
\end{definition}

\begin{remark}\label{rem-curvature-arguments}
Let $(x,u) \in {\rm Graph}(\caln_{\cald}^b)$ be arbitrary. In the sequel, we will often consider ${\rm Curv}_{\cald}(x,u)(v,\tau(x))$ for $v \in \caln_{\cald}^c(x)^{\perp}$ and a mapping $\tau : \cald \to H$ such that condition \eqref{orthogonal-tau} below is fulfilled. In this regard, note the following:
\begin{itemize}
\item By Lemma \ref{lemma-normal-cones} and Lemma \ref{lemma-cones-inclusions} we have
\begin{align}\label{inclusions-curvature}
\caln_{\cald}^c(x)^{\perp} \subset \caln_{\cald}^c(x)^{\circ} = T_{\cald}^c(x) \subset T_{\cald}^b(x) \quad \forall x \in \cald.
\end{align}
\item If the mapping $\tau$ is continuous, then condition \eqref{orthogonal-tau} below implies that $\tau(x) \in T_{\cald}^b(x)$ for all $x \in \cald$. This is a consequence of Proposition \ref{prop-normal-inward-a-cont} and Remark \ref{rem-more-general-inward-a},
\end{itemize}
\end{remark}

The following result generalizes the first part of \cite[Lemma 2.3]{DP-Frankowska}.

\begin{proposition}\label{prop-curvature}
Let $\tau : H \to H$ be a mapping of class $C^1$ such that
\begin{align}\label{orthogonal-tau}
\la u,\tau(x) \ra = 0 \quad \forall x \in \cald \quad \forall u \in \caln_{\cald}^b(x).
\end{align}
Then the following statements are true:
\begin{enumerate}
\item For all $(x,u) \in {\rm Graph}(\caln_{\cald}^b)$, $v \in T_{\cald}^b(x)$ and $\mu \in D \caln_{\cald}^b(x,u)v$ we have
\begin{align*}
{\rm Curv}_{\cald}(x,u)(v,\tau(x)) = \la \mu, \tau(x) \ra.
\end{align*}
\item For all $(x,u) \in {\rm Graph}(\caln_{\cald}^b)$ and $v \in T_{\cald}^b(x)$ we have
\begin{align*}
{\rm Curv}_{\cald}(x,u)(v,\tau(x)) = -\la u, D \tau(x) v \ra.
\end{align*}
\end{enumerate}
\end{proposition}

\begin{proof}
Let us fix arbitrary $(x,u) \in {\rm Graph}(\caln_{\cald}^b)$, $v \in T_{\cald}^b(x)$ and $\mu \in D \caln_{\cald}^b(x,u)v$. Then there are sequences $(t_n)_{n \in \bbn} \subset (0,\infty)$ with $t_n \to 0^+$ and $(v_n,\mu_n)_{n \in \bbn} \subset H \times H$ with $(v_n,\mu_n) \to (v,\mu)$ such that $x + t_n v_n \in \cald$ and $u + t_n \mu_n \in \caln_{\cald}^b(x + t_n v_n)$ for all $n \in \bbn$. By \eqref{orthogonal-tau} we have
\begin{align*}
\la u + t_n \mu_n, \tau(x + t_n v_n) \ra = 0 \quad \forall n \in \bbn,
\end{align*}
and hence
\begin{align*}
\la \mu_n, \tau(x + t_n v_n) \ra + \bigg\la u, \frac{\tau(x + t_n v_n) - \tau(x)}{t_n} \bigg\ra = 0 \quad \forall n \in \bbn.
\end{align*}
Sending $n \to \infty$, by Lemma \ref{lemma-derivative-sequence} we obtain
\begin{align*}
\la \mu, \tau(x) \ra = - \la u, D \tau(x) v \ra.
\end{align*}
Since $\mu \in D \caln_{\cald}^b(x,u)v$ was arbitrary, this completes the proof.
\end{proof}

Now, we will compute the contingent curvature for closed convex cones. For this purpose, we prepare two auxiliary results.

\begin{lemma}\label{lemma-ccc-tangent-normal}\cite[Lemma 4.2.5]{Aubin_Fr}
Suppose that $\cald$ is a closed convex cone. Then we have
\begin{align*}
T_{\cald}^b(x) = \cald + \lin \{ x \} \quad \text{and} \quad \caln_{\cald}^b(x) = \cald^{\circ} \cap \{ x \}^{\perp} \quad \forall x \in \cald.
\end{align*}
\end{lemma}

For the next auxiliary result recall the inclusions \eqref{inclusions-curvature}.

\begin{lemma}\label{lemma-ccc}
Suppose that $\cald$ is a closed convex cone. Let $x \in \cald$ and $v \in \caln_{\cald}^c(x)^{\perp}$ be arbitrary. Then there exists $\epsilon > 0$ such that $x + tv \in \cald$ and $\caln_{\cald}^b(x) \subset \caln_{\cald}^b(x + tv)$ for all $t \in [0,\epsilon]$.
\end{lemma}

\begin{proof}
By Lemma \ref{lemma-ccc-tangent-normal}, for $x \in \cald$ and $v \in \caln_{\cald}^c(x)^{\perp}$ we have $v = \lambda x + y$ for some $\lambda \in \bbr$ and $y \in \cald$. If $\lambda = 0$, then we have $x + tv = x + ty \in \cald$ for all $t \geq 0$. Therefore, we may assume $\lambda \neq 0$. Setting $\epsilon := \frac{1}{2 |\lambda|}$, we obtain $|t \lambda| \leq \frac{1}{2}$, and hence
\begin{align*}
x + tv = (1 + t \lambda) x + t y \in \cald \quad \forall t \in [0,\epsilon].
\end{align*}
Now, let $u \in \caln_{\cald}^b(x)$ be arbitrary. By Lemma \ref{lemma-normal-cones} we have $v \in \caln_{\cald}^b(x)^{\perp}$, and hence $\la u,v \ra = 0$. Moreover, by Lemma \ref{lemma-ccc-tangent-normal} we have $u \in \cald^{\circ}$ and $\la u,x \ra = 0$, and it follows that $\la u,x+tv \ra = 0$. Thus, using Lemma \ref{lemma-ccc-tangent-normal} again we arrive at $u \in \caln_{\cald}^b(x + tv)$, completing the proof.
\end{proof}

\begin{proposition}\label{prop-curvature-cone}
Suppose that $\cald$ is a closed convex cone. Let $\tau : H \to H$ be a mapping of class $C^1$ such that \eqref{orthogonal-tau} is fulfilled. Then for all $(x,u) \in {\rm Graph}(\caln_{\cald}^b)$ and $v \in \caln_{\cald}^c(x)^{\perp}$ we have
\begin{align*}
{\rm Curv}_{\cald}(x,u)(v,\tau(x)) = 0.
\end{align*}
\end{proposition}

\begin{proof}
By Lemma \ref{lemma-ccc} we have $0 \in D \caln_{\cald}^b(x,u)v$. Thus, by Proposition \ref{prop-curvature} the assertion follows.
\end{proof}

Now, we will provide the required results about $\varphi$-convex sets.

\begin{definition}\label{def-phi-convex}
The closed set $\cald$ is called \emph{$\varphi$-convex} if there exists a continuous function $\varphi : \cald \times \cald \to \bbr_+$ such that for all $x,y \in \cald$ and $u \in \caln_{\cald}^{\sigma}(x)$ we have
\begin{align*}
\la u,y-x \ra \leq \varphi(x,y) \, \| u \| \, \| y-x \|^2.
\end{align*}
\end{definition}

The following characterization of $\varphi$-convexity will be useful.

\begin{lemma}\label{lemma-phi-conv-char}\cite[Prop. 6.2]{Colombo}
The closed set $\cald$ is $\varphi$-convex if and only if there exists a continuous function $\psi : \cald \to \bbr_+$ such that for all $x,y \in \cald$ and $u \in \caln_{\cald}^p(x)$ we have
\begin{align*}
\la u,y-x \ra \leq \psi(x) \, \| u \| \, \| y-x \|^2.
\end{align*}
\end{lemma}

\begin{definition}\label{def-local-phi-convex}
The set $\cald$ is called \emph{locally $\varphi$-convex} if for each $x \in \cald$ there exists a closed neighborhood $C \subset H$ of $x$ such that $\cald \cap C$ is $\varphi$-convex.
\end{definition}

Let us recall the following well-known result about convex sets.

\begin{lemma}\label{lemma-dist-fct}
Suppose that the closed subset $\cald$ is convex, and let $x_0 \in H$ be arbitrary. Then for all $x \in \cald$ the following statements are equivalent:
\begin{enumerate}
\item[(i)] $\| x_0 - x \| = d_{\cald}(x_0)$.

\item[(ii)] $\la x_0-x, y-x \ra \leq 0$ for all $y \in \cald$.
\end{enumerate}
\end{lemma}

The next result shows that convex sets are indeed examples of $\varphi$-convex sets.

\begin{lemma}\label{lemma-convex-is-varphi-convex}
Suppose that the closed subset $\cald$ is convex. Then for all $x,y \in \cald$ and $u \in \caln_{\cald}^p(x)$ we have $\la u,y-x \ra \leq 0$.
\end{lemma}

\begin{proof}
Let $x \in \cald$ and $u \in \caln_{\cald}^p(x)$ be arbitrary. By Lemma \ref{lemma-prox-normal} we have $t \| u \| = d_{\cald}(x+tu)$ for some $t > 0$. Hence, setting $x_0 := x+tu$ we have $\| x_0 - x \| = d_{\cald}(x_0)$. Thus, by Lemma \ref{lemma-dist-fct} we obtain
\begin{align*}
\la u,y-x \ra = \frac{1}{t} \la x_0-x,y-x \ra \leq 0 \quad \text{for all $y \in \cald$,}
\end{align*}
completing the proof.
\end{proof}

\begin{proposition}\label{prop-phi-convex}
Consider the following four statements:
\begin{itemize}
\item[(i)] $\cald$ is convex.

\item[(ii)] For each $x \in \cald$ there exists a closed neighborhood $C \subset H$ of $x$ such that $\cald \cap C$ is convex.

\item[(iii)] $\cald$ is locally $\varphi$-convex.

\item[(iv)] We have $\caln_{\cald}^p(x) = \caln_{\cald}^c(x)$ for each $x \in \cald$.
\end{itemize}
Then we have the implications (i) $\Rightarrow$ (ii) $\Rightarrow$ (iii) $\Rightarrow$ (iv).
\end{proposition}

\begin{proof}
(i) $\Rightarrow$ (ii): This implication is obvious.

\noindent(ii) $\Rightarrow$ (iii): In view of Lemma \ref{lemma-convex-is-varphi-convex}, this follows from Lemma \ref{lemma-phi-conv-char} by choosing $\psi \equiv 0$.

\noindent(iii) $\Rightarrow$ (iv): Taking into account Proposition \ref{prop-normal-local}, this is a consequence of \cite[Prop. 6.2]{Colombo}.
\end{proof}

\begin{proposition}\label{prop-a-phi-convex}
Suppose that $\cald$ is locally $\varphi$-convex. Then for a mapping $a : H \to H$ the following statements are equivalent:
\begin{enumerate}
\item[(i)] We have $\la u,a(x) \ra \leq 0$ for all $x \in \cald$ and $u \in \caln_{\cald}^{1,\pr}(x)$.

\item[(ii)] We have $\la u,a(x) \ra \leq 0$ for all $x \in \cald$ and $u \in \caln_{\cald}^c(x)$.

\item[(iii)] We have $a(x) \in T_{\cald}^c(x)$ for all $x \in \cald$.
\end{enumerate}
\end{proposition}

\begin{proof}
(i) $\Rightarrow$ (ii): Let $x \in \cald$ and $u \in \caln_{\cald}^c(x)$ be arbitrary. By Proposition \ref{prop-phi-convex} we have $u \in \caln_{\cald}^p(x)$. Hence, there are $v \in \caln_{\cald}^{1,\pr}(x)$ and $\lambda \geq 0$ such that $u = \lambda v$, and it follows that $\la u,a(x) \ra = \lambda \la v,a(x) \ra \leq 0$.

\noindent(ii) $\Rightarrow$ (i): This implication is obvious.

\noindent(ii) $\Leftrightarrow$ (iii): This equivalence is a consequence of Lemma \ref{lemma-normal-cones}.
\end{proof}

\section{Submanifolds with boundary}\label{app-submanifolds}

In this appendix we provide the required results about submanifolds with boundary. For further details, we refer to \cite[Sec. 3]{FTT-appendix}. Let $H$ be a Hilbert space and let $m \in \mathbb{N}$ be a positive integer. Throughout this section, we will use the notation
\begin{align*}
\mathbb{R}_+^m = \mathbb{R}_+ \times \mathbb{R}^{m-1} = \{ y \in
\mathbb{R}^m : y_1 \geq 0 \}.
\end{align*}
We consider the relative topology on $ \mathbb{R}^m_+ $. Let $V$ be an open subset in $\mathbb{R}_+^m$, i.e., there exists an open set $ \tilde{V} \subset \mathbb{R}^m $ such that $ \tilde{V} \cap \mathbb{R}_+^m = V $. We denote by
\begin{align*}
\partial V = \{ y \in V : y_1 = 0 \}
\end{align*}
the set of all boundary points of $ V $. Let $k \in \bbn \cup \{ \infty \}$ be arbitrary.

\begin{definition}\label{def-Ck-map}
A map $ \phi: V \subset \mathbb{R}^m_+ \to H $ is
called a \emph{$ C^k $-map}, if there is an open set $ \tilde{V} \subset
\mathbb{R}^m $ together with a $ C^k $-map $ \tilde{\phi}: \tilde{V}
\to H $ such that $ \tilde{V} \cap \mathbb{R}_+^m = V $ and $
\tilde{\phi}|_{V} = \phi $.
\end{definition}

For a $C^k$-map $ \phi: V \subset \mathbb{R}^m_+ \to H $ and $y \in V$ we define the derivative $D \phi(y) := D \tilde{\phi}(y)$.
Note that this definition does not depend on the choice of $\tilde{\phi}$.

\begin{definition}
A nonempty subset $\mathcal{M} \subset H$ is called an \emph{$m$-dimensional $C^k$-submanifold with boundary of $H$} if for every $x \in
\mathcal{M}$ there exist an open neighborhood $U \subset H$ of $x$, a subset $V \subset \mathbb{R}^m_+$ which is open with respect to the relative topology, and a mapping $\phi \in C^k(V,H)$ such
that:
\begin{enumerate}
\item The mapping $\phi : V \rightarrow U \cap \mathcal{M}$ is a homeomorphism.
\item $D \phi(y) \in L(\bbr^m,H)$ is one-to-one for all $y \in V$.
\end{enumerate}
The mapping $\phi$ is called a \emph{local parametrization} of $\mathcal{M}$ around $x$.
\end{definition}

In what follows, let $\mathcal{M}$ be an $m$-dimensional $C^k$-submanifold with boundary of $H$.

\begin{definition}
The \emph{boundary} $\partial \calm$ is defined as the set of all points $x \in \calm$ such that $\phi^{-1}(x) \in \partial V$ for some local parametrization $\phi : V \to U \cap \calm$ around $x$.
\end{definition}

\begin{remark}
The boundary $\partial \calm$ is a submanifold without boundary of dimension $m-1$, and parametrizations of $ \partial \mathcal{M} $ are provided by restricting parametrizations $ \phi: V \to U \cap \calm $ to
the boundary $ \partial V $.
\end{remark}

By \cite[Lemma 6.1.1]{fillnm} and \cite[Lemma 3.3]{FTT-appendix} the following definitions of the tangent spaces $T_x \mathcal{M}$ and ${(T_x \mathcal{M})}_+$ do not depend on the choice of the parametrization.

\begin{definition}
Let $x \in \mathcal{M}$ be arbitrary, and let $\phi : V \subset \mathbb{R}_+^m \rightarrow U \cap \mathcal{M}$ be a local parametrization around $x$.
\begin{enumerate}
\item The \emph{tangent space} to $\mathcal{M}$ at $x$ is the subspace
\begin{align*}
T_x \mathcal{M} := D \phi(y) \mathbb{R}^m, \quad y = \phi^{-1}(x) \in V.
\end{align*}
\item For $ x \in \partial \mathcal{M} $ we can distinguish a half
space in $ T_x \mathcal{M} $, namely the closed convex cone of all inward pointing directions in $\mathcal{M}$, given by
\begin{align*}
{(T_x \mathcal{M})}_+ := D \phi(y) \mathbb{R}^m_+, \quad y =
\phi^{-1}(x) \in \partial V.
\end{align*}
\end{enumerate}
\end{definition}

\begin{remark}\label{rem-tang-boundary}
Let $x \in \mathcal{M}$ be arbitrary and let $\phi : V \subset \mathbb{R}_+^m \rightarrow U \cap \mathcal{M}$ be a local parametrization around $x$. Then we have
\begin{align*}
T_x \partial \mathcal{M} = D \phi(y) \partial \mathbb{R}_+^m, \quad y = \phi^{-1}(x) \in \partial V.
\end{align*}
In particular, we see that
\begin{align*}
T_x \partial \mathcal{M} &= (T_x \mathcal{M})_+ \cap - (T_x \mathcal{M})_+ \subset (T_x \mathcal{M})_+.
\end{align*}
\end{remark}

From now on, we assume that the submanifold $\calm$ is closed as a subset of $H$.

\begin{lemma}\label{lemma-submanifold-tang-1}
For each $x \in \calm$ the following statements are true:
\begin{enumerate}
\item If $x \in \calm \setminus \partial \calm$, then we have $T_x \calm \subset T_{\calm}^c(x)$.

\item If $x \in \partial \calm$, then we have $(T_x \calm)_+ \subset T_{\calm}^c(x)$.
\end{enumerate}
\end{lemma}

\begin{proof}
Let $v \in T_x \calm$ be arbitrary. If $x \in \partial \calm$, then we even assume that $v \in (T_x \calm)_+$. Let $\phi : V \to U \cap \calm$ be a local parametrization around $x$. We set $y := \phi^{-1}(x) \in V$. Then $\phi$ extends to a mapping $\phi : \tilde{V} \to H$ of class $C^k$ such that $\tilde{V} \cap \bbr_+^m = V$, and we may assume that $\tilde{V} \subset \bbr^m$ is an open, convex neighborhood of $y$. We set $w := D \phi(y)^{-1} v \in \bbr^m$. If $x \in \partial \calm$, then we have $y \in \partial V$ and $w \in \bbr_+^m$.

We wish to show that $v \in T_{\calm}^c(x)$ by using Proposition \ref{prop-cones}. For this purpose, let $(t_n)_{n \in \bbn} \subset (0,\infty)$ with $t_n \to 0^+$ be arbitrary, and let $(x_n)_{n \in \bbn} \subset \calm$ with $x_n \to x$ be arbitrary. Since $U$ is an open neighborhood of $x$, we may assume that $x_n \in U \cap \calm$ for each $n \in \bbn$. We define the sequence $(y_n)_{n \in \bbn} \subset V$ as $y_n := \phi^{-1}(x_n) \in V$. Then we have $y_n \to y$, and hence $y_n + t_n w \to y$. Recall that in case $x \in \partial \calm$ we have $w \in \bbr_+^m$. Therefore, and since $V \subset \bbr_+^m$ is an open neighborhood of $y$ with respect to the relative topology, we may assume that $y_n + t_n w \in V$ for each $n \in \bbn$.

Since $\tilde{V}$ is open and convex, we may apply Taylor's theorem (see, e.g. \cite[Thm. 2.4.15]{Abraham}), which provides the existence of a continuous mapping $R : \tilde{V} \times \tilde{V} \to L(\bbr^m,H)$ such that $R(z,z) = 0$ for all $z \in \tilde{V}$, and for each $n \in \bbn$ we have
\begin{align*}
\phi(y_n + t_n w) &= \phi(y_n) + D \phi(y_n) t_n w + R(y_n,y_n+t_n w) \cdot t_n w
\\ &= x_n + t_n \big( D \phi(y_n) w + R(y_n,y_n+t_n w) w \big).
\end{align*}
We define the sequence $(v_n)_{n \in \bbn} \subset H$ as
\begin{align*}
v_n := D \phi(y_n) w + R(y_n,y_n+t_n w) w, \quad n \in \bbn.
\end{align*}
Then we have
\begin{align*}
\lim_{n \to \infty} v_n = D \phi(y)w = v,
\end{align*}
and for each $n \in \bbn$ we have
\begin{align*}
x_n + t_n v_n = \phi(y_n + t_n w) \in \calm.
\end{align*}
Consequently, by Proposition \ref{prop-cones} we deduce that $v \in T_{\calm}^c(x)$.
\end{proof}

\begin{lemma}\label{lemma-submanifold-tang-2}
For each $x \in \calm$ the following statements are true:
\begin{enumerate}
\item If $x \in \calm \setminus \partial \calm$, then we have $T_{\calm}^{\sigma}(x) \subset T_x \calm$.

\item If $x \in \partial \calm$, then we have $T_{\calm}^{\sigma}(x) \subset (T_x \calm)_+$.
\end{enumerate}
\end{lemma}

\begin{proof}
Let $v \in T_x \calm$ be arbitrary, and let $\phi : V \to U \cap \calm$ be a local parametrization around $x$. We set $y := \phi^{-1}(x) \in V$. Then $\phi$ extends to a mapping $\phi : \tilde{V} \to H$ of class $C^k$ such that $\tilde{V} \cap \bbr_+^m = V$, and we may assume that $\tilde{V} \subset \bbr^m$ is an open, convex neighborhood of $y$. If $x \in \partial \calm$, then we have $y \in \partial V$.

By Definition \ref{def-cones} there are sequences $(t_n)_{n \in \bbn} \subset (0,\infty)$ with $t_n \to 0^+$ and $(v_n)_{n \in \bbn} \subset H$ with $v_n \overset{\sigma}{\to} v$ such that $x_n := x + t_n v_n \in \calm$ for each $n \in \bbn$. The sequence $(v_n)_{n \in \bbn}$ is bounded, and hence $t_n v_n \to 0$, implying that $x_n \to x$. Thus, defining the sequence $(y_n)_{n \in \bbn} \subset V$ as $y_n := \phi^{-1}(x_n)$ for each $n \in \bbn$, we obtain $y_n \to y$. Now, we define the sequence $(w_n)_{n \in \bbn} \subset \bbr^m$ as
\begin{align*}
w_n := \frac{y_n - y}{t_n}, \quad n \in \bbn.
\end{align*}
If $x \in \partial \calm$, then we have $y \in \partial V$, and hence $w_n \in \bbr_+^m$ for each $n \in \bbn$. Moreover, for each $n \in \bbn$ we have
\begin{align*}
x + t_n v_n = x_n = \phi(y_n) = \phi(y + t_n w_n).
\end{align*}
Since $\tilde{V}$ is open and convex, we may apply Taylor's theorem (see, e.g. \cite[Thm. 2.4.15]{Abraham}), which provides the existence of a continuous mapping $R : \tilde{V} \times \tilde{V} \to L(\bbr^m,H)$ such that $R(z,z) = 0$ for all $z \in \tilde{V}$, and for each $n \in \bbn$ we have
\begin{align*}
\phi(y + t_n w_n) &= \phi(y) + D \phi(y) (t_n w_n) + R(y,y+t_n w_n) \cdot t_n w_n
\\ &= x + t_n \big( D \phi(y) w_n + R(y,y+t_n w_n) w_n \big).
\end{align*}
Now, it follows that
\begin{align*}
v_n = \frac{\phi(y+t_n w_n) - x}{t_n} = D \phi(y) w_n + R(y,y+t_n w_n) w_n, \quad n \in \bbn.
\end{align*}
We define $z_n := D \phi(y) w_n$ for each $n \in \bbn$. Then we have $z_n \in T_x \calm$ for each $n \in \bbn$, and in case $x \in \partial \calm$ we even have $z_n \in (T_x \calm)_+$ for each $n \in \bbn$. Furthermore, we have $t_n w_n = y_n - y \to 0$, and therefore we obtain
\begin{align*}
v = \sigma \text{-} \lim_{n \to \infty} v_n = \sigma \text{-} \lim_{n \to \infty} z_n = \lim_{n \to \infty} z_n.
\end{align*}
Since $T_x \calm$ and $(T_x \calm)_+$ are closed subsets of $H$, we deduce that $v \in T_x \calm$, and in case $x \in \partial \calm$ we even have $v \in (T_x \calm)_+$.
\end{proof}

The following auxiliary result is a slight reformulation of \cite[Lemma 3.7]{FTT-appendix}. For a subset $A \subset H$ we use the notation
\begin{align*}
A^- := \{ y \in H : \la y,x \ra \leq 0 \text{ for all } x \in A \}.
\end{align*}

\begin{lemma}
For each $x \in \partial \calm$ there exists a unique vector $n_x \in  (T_x \partial \calm)^{\perp}$ with $-n_x \in (T_x \calm)_+$ and $\| n_x \| = 1$ such that
\begin{align}\label{normal-v-id-0}
T_x \calm = T_x \partial \calm \oplus \lin \{ n_x \}.
\end{align}
Moreover, for each $x \in \partial \calm$ we have
\begin{align}\label{normal-v-id-1}
T_x \partial \calm &= T_x \calm \cap \{ n_x \}^{\perp},
\\ \label{normal-v-id-2} (T_x \calm)_+ &= T_x \calm \cap \{ n_x \}^-.
\end{align}
\end{lemma}

\begin{definition}
For each $x \in \partial \calm$ we call $n_x$ the \emph{outward pointing normal vector} to $\partial \calm$ at $x$.
\end{definition}

\begin{remark}
In \cite{FTT-appendix} the inward pointing normal vector $\eta_x := -n_x$ was considered.
\end{remark}

The following result shows that the submanifold $\calm$ is a so-called \emph{$\sigma$-regular} set, which means that $\caln_{\cald}^{\sigma}(x) = \caln_{\cald}^c(x)$ for each $x \in \cald$; cf. \cite[Def. 4.3]{Colombo}. Note that the decomposition \eqref{manifolds-cones-4} below refers to the direct sum decomposition $H = (T_x \calm)^{\perp} \oplus T_x \calm$ of the Hilbert space.

\begin{proposition}\label{prop-submanifold-sigma-regular}
For each $x \in \calm$ the following statements are true:
\begin{enumerate}
\item If $x \in \calm \setminus \partial \calm$, then we have
\begin{align}\label{manifolds-cones-1}
T_{\calm}^c(x) &= T_{\calm}^{\sigma}(x) = T_x \calm,
\\ \label{manifolds-cones-2} \caln_{\calm}^{\sigma}(x) &= \caln_{\calm}^c(x) = (T_x \calm)^{\perp}.
\end{align}
\item If $x \in \partial \calm$, then we have
\begin{align}\label{manifolds-cones-3}
T_{\calm}^c(x) &= T_{\calm}^{\sigma}(x) = (T_x \calm)_+,
\\ \label{manifolds-cones-4} \caln_{\calm}^{\sigma}(x) &= \caln_{\calm}^c(x) = (T_x \calm)^{\perp} \oplus \lin^+ \{ n_x \}.
\end{align}
\end{enumerate}
\end{proposition}

\begin{proof}
The identities \eqref{manifolds-cones-1} and \eqref{manifolds-cones-3} are consequences of Lemmas \ref{lemma-submanifold-tang-1} and \ref{lemma-submanifold-tang-2}. Recalling that $\caln_{\calm}^{\sigma}(x) = T_{\calm}^{\sigma}(x)^{\circ}$ and $\caln_{\calm}^{c}(x) = T_{\calm}^{c}(x)^{\circ}$, we deduce that $T_{\calm}^c(x) = T_{\calm}^{\sigma}(x)$ for each $x \in \calm$. This shows \eqref{manifolds-cones-2}, because $T_x \calm$ is a subspace of $H$. In order to show \eqref{manifolds-cones-4}, let $x \in \partial \calm$ be arbitrary, and recall that
\begin{align*}
\caln_{\calm}^c(x) = (T_x \calm)_+^{\circ}.
\end{align*}
Let $y \in (T_x \calm)^{\perp}$ and $\lambda \geq 0$ be arbitrary. Then for each $z \in (T_x \calm)_+$ we have $\la y,z \ra = 0$, and by \eqref{normal-v-id-2} we have $\la n_x,z \ra \leq 0$. Therefore, we obtain
\begin{align*}
\la y + \lambda n_x,z \ra = \la y,z \ra + \lambda \la n_x,z \ra \leq 0,
\end{align*}
showing that $y + \lambda n_x \in \caln_{\calm}^c(x)$.

Conversely, let $y \in \caln_{\calm}^c(x) = (T_x \calm)_+^{\circ}$ be arbitrary. We define $\lambda := \la y,n_x \ra$. Then we have $\lambda \geq 0$, because $-n_x \in (T_x \calm)_+$. Decomposing $y = (y - \lambda n_x) + \lambda n_x$, we wish to show that $y - \lambda n_x \in (T_x \calm)^{\perp}$. For this purpose, let $z \in T_x \calm$ be arbitrary. By \eqref{normal-v-id-0} we have $z = w + \mu n_x$ with $w \in T_x \partial \calm$ and $\mu \in \bbr$. Since $y \in (T_x \calm)_+^{\circ} \subset (T_x \partial \calm)^{\circ} = (T_x \partial \calm)^{\perp}$, we have $\la y,w \ra = 0$, and since $n_x \in (T_x \partial \calm)^{\perp}$, we have $\la n_x,w \ra = 0$. Recalling that $\| n_x \| = 1$, we obtain
\begin{align*}
\la y - \lambda n_x, z \ra &= \la y - \lambda n_x, w + \mu n_x \ra = \la y,w \ra + \mu \la y,n_x \ra - \lambda \la n_x,w \ra - \lambda \mu
\\ &= \mu ( \la y,n_x \ra - \lambda ) = 0,
\end{align*}
showing that $y - \lambda n_x \in (T_x \calm)^{\perp}$.
\end{proof}

Recalling that in \cite{FTT-appendix} the inward pointing normal vector $\eta_x := -n_x$ was considered, the following result is a reformulation of \cite[Lemma 3.9]{FTT-appendix}.

\begin{lemma}\label{lemma-normal-scalar}
Let $\phi : V \subset \bbr_+^m \to U \cap \calm$ be a local parametrization. Then, for every $x \in U \cap \partial \calm$ there exists a unique number $\kappa(x) < 0$ such that
\begin{align*}
\la n_x,D\phi(y)v \ra = \kappa(x) \la e_1,v \ra \quad \text{for all $v \in \bbr^m$,}
\end{align*}
where $y = \phi^{-1}(x)$.
\end{lemma}

\begin{proposition}\label{prop-submanifold-phi-regular}
Suppose that the submanifold $\calm$ is of class $C^2$. Then $\calm$ is locally $\varphi$-convex, and for each $x \in \calm$ the following statements are true:
\begin{enumerate}
\item If $x \in \calm \setminus \partial \calm$, then we have
\begin{align*}
\caln_{\calm}^p(x) = \caln_{\calm}^c(x) = (T_x \calm)^{\perp}.
\end{align*}

\item If $x \in \partial \calm$, then we have
\begin{align*}
\caln_{\calm}^p(x) = \caln_{\calm}^c(x) = (T_x \calm)^{\perp} \oplus \lin^+ \{ n_x \}.
\end{align*}
\end{enumerate}
\end{proposition}

\begin{proof}
Let $x_0 \in \calm$ be arbitrary, and let $\phi : V \to U \cap \calm$ be a local parametrization around $x_0$. We set $y_0 := \phi^{-1}(x_0) \in V$. Then $\phi$ extends to a mapping $\phi : \tilde{V} \to H$ of class $C^2$ such that $\tilde{V} \cap \bbr_+^m = V$, and we may assume that $\tilde{V} \subset \bbr^m$ is an open, convex neighborhood of $y_0$. Moreover, as a consequence of the inverse mapping theorem (see \cite[Prop. 6.1.1]{fillnm}) we may assume that $\Phi := \phi^{-1} : \phi(\tilde{V}) \to \tilde{V}$ is Lipschitz continuous.

By Taylor's theorem (see, e.g. \cite[Thm. 2.4.15]{Abraham}) there is a continuous mapping $R : \tilde{V} \times \tilde{V} \to L_s^2(\bbr^m,H)$ with $R(v,v) = 0$ for all $v \in \tilde{V}$ such that for all $u,v \in \tilde{V}$ we have
\begin{align*}
\phi(v) - \phi(u) &= D \phi(u) \cdot (v-u) + \frac{1}{2} D^2 \phi(u) \cdot (v-u,v-u)
\\ &\quad + R(u,v) \cdot (v-u,v-u).
\end{align*}
Hence, for all $x,y \in U \cap \calm$ we obtain
\begin{align*}
y-x &= \phi(\Phi(y)) - \phi(\Phi(x))
\\ &= D \phi(\Phi(x)) \cdot (\Phi(y)-\Phi(x)) + \frac{1}{2} D^2 \phi(\Phi(x)) (\Phi(y)-\Phi(x), \Phi(y)-\Phi(x))
\\ &\quad + R(\Phi(x),\Phi(y)) \cdot (\Phi(y)-\Phi(x), \Phi(y)-\Phi(x)).
\end{align*}
Now, let $C \subset U$ be a closed neighborhood of $x_0$. Denoting by $L$ the Lipschitz constant of $\Phi$, we define the continuous mapping $\varphi : (C \cap \calm) \times (C \cap \calm) \to \bbr_+$ as
\begin{align*}
\varphi(x,y) := L^2 \bigg\| \frac{1}{2} D^2 \phi(\Phi(x)) + R(\Phi(x),\Phi(y)) \bigg\|.
\end{align*}
Let $x,y \in C \cap \calm$ and $u \in \caln_{C \cap \calm}^{\sigma}(x)$ be arbitrary. Suppose first that $x \in \calm \setminus \partial \calm$. Taking into account Proposition \ref{prop-normal-local}, by Proposition \ref{prop-submanifold-sigma-regular} we have $u \in (T_x \calm)^{\perp}$, and hence
\begin{align*}
\la u,y-x \ra \leq \varphi(x,y) \, \| u \| \, \| y-x \|^2.
\end{align*}
Now, consider the case $x \in \partial \calm$. By Proposition \ref{prop-submanifold-sigma-regular} we have $u \in (T_x \calm)^{\perp} \oplus \lin^+ \{ n_x \}$. Thus we have $u = v + \lambda n_x$ for some $v \in (T_x \calm)^{\perp}$ and $\lambda \geq 0$. Set $z := \Phi(x)$ and $w := \Phi(y)$. Then we have $z \in \partial \bbr_+^m$, and hence $w-z \in \bbr_+^m$. By Proposition \ref{lemma-normal-scalar} we obtain
\begin{align*}
\la n_x,D \phi(z) ( w-z ) \ra = \kappa(x) \la e_1, w-z \ra \leq 0.
\end{align*}
Furthermore, since $n_x \in T_x \calm$ with $\| n_x \| = 1$, by orthogonality we have $\| u \|^2 = \| v \|^2 + \lambda^2$, and hence $\| v \| \leq \| u \|$. Consequently, we arrive at
\begin{align*}
\la u,y-x \ra = \la v,y-x \ra + \lambda \la n_x,y-x \ra &\leq \varphi(x,y) \, \| v \| \, \| y-x \|^2
\\ &\leq \varphi(x,y) \, \| u \| \, \| y-x \|^2.
\end{align*}
This proves that $\calm$ is locally $\varphi$-convex. Now, the remaining statements are a consequence of Proposition \ref{prop-phi-convex} and Proposition \ref{prop-submanifold-sigma-regular}.
\end{proof}

The following example shows that for a $C^1$-submanifold $\calm$ the inclusion $$\caln_{\calm}^p(x) \subset \caln_{\calm}^{\sigma}(x)$$ may be strict.

\begin{example}\label{example-graph}
Consider the state space $H=\bbr^2$ and the submanifold
\begin{align*}
\calm = \{ (x,|x|^{3/2}) : x \in \bbr \}.
\end{align*}
Then $\calm$ is a $C^1$-submanifold, which is not of class $C^2$. According to Proposition \ref{prop-submanifold-sigma-regular} we have $\caln_{\calm}^{\sigma}(0) = \lin \{ e_2 \}$. We will show that $\caln_{\calm}^p(0) = - \lin^+ \{ e_2 \}$. Indeed, each $y \in \calm$ is of the form $y = (x,|x|^{3/2})$ for some $x \in \bbr$. Therefore, we have
\begin{align*}
\frac{\la e_2,y \ra}{\| y \|^2} = \frac{|x|^{3/2}}{|x|^2 + |x|^3}.
\end{align*}
Moreover, concerning the reciprocal we have
\begin{align*}
\frac{|x|^2 + |x|^3}{|x|^{3/2}} = |x|^{1/2} + |x|^{3/2} \to 0 \quad \text{for $x \to 0$.}
\end{align*}
Using Proposition \ref{prop-normal-2} we deduce that $-e_2 \in \caln_{\calm}^p(0)$, whereas $e_2 \notin \caln_{\calm}^p(0)$. This shows $\caln_{\calm}^p(0) = - \lin^+ \{ e_2 \}$.
\end{example}

\section{Martingales in Banach spaces}\label{app-martingales-Banach}

In this appendix we provide the required results about martingales in Banach spaces. Let $E$ be a Banach space, and let $(\Omega,\calf,(\calf_t)_{t \in \bbr_+},\bbp)$ be a filtered probability space satisfying the usual conditions.

\begin{lemma}\label{lemma-MT-Banach}
Let $M$ be an $E$-valued adapted process such that $\bbe[\| M_t \|] < \infty$ for all $t \in \bbr_+$. Then the following statements are equivalent:
\begin{enumerate}
\item[(i)] $M$ is a martingale.

\item[(ii)] $x'(M)$ is a martingale for each $x' \in E'$.
\end{enumerate}
\end{lemma}

\begin{proof}
An adapted process $N$ with values in some Banach space $F$ satisfying $\bbe[\| N_t \|] < \infty$ for all $t \in \bbr_+$ is a martingale if and only if $\bbe[N_t \bbI_A] = \bbe[N_s \bbI_A]$ for all $s \leq t$ and $A \in \calf_s$. Recalling that for all $x,y \in E$ we have $x = y$ if and only if $x'(x) = x'(y)$ for all $x' \in E'$, this provides the stated equivalence (i) $\Leftrightarrow$ (ii).
\end{proof}

\begin{lemma}\label{lemma-Doob-Banach}
Let $M$ be an $E$-valued c\`{a}dl\`{a}g, adapted process such that $\bbe[\| M_\tau \|] < \infty$ for each stopping time $\tau$. Then the following statements are equivalent:
\begin{enumerate}
\item[(i)] $M$ is a martingale.

\item[(ii)] For each stopping time $\tau$ we have $\bbe[M_{\tau}] = \bbe[M_0]$.

\item[(iii)] For each stopping time $\tau$ the stopped process $M^{\tau}$ is a martingale.
\end{enumerate}
\end{lemma}

\begin{proof}
This is a consequence of Lemma \ref{lemma-MT-Banach} and \cite[Lemma I.1.44]{Jacod-Shiryaev}.
\end{proof}

\begin{lemma}\label{lemma-loc-MT-functionals}
Let $M$ be an $E$-valued continuous, adapted process such that $M_0$ is bounded. Then the following statements are equivalent:
\begin{enumerate}
\item[(i)] $M$ is a local martingale.

\item[(ii)] $x'(M)$ is a local martingale for each $x \in E'$.
\end{enumerate}
\end{lemma}

\begin{proof}
We only need to prove (ii) $\Rightarrow$ (i). We define the localizing sequence $(\tau_n)_{n \in \bbn}$ of stopping times as
\begin{align*}
\tau_n := \inf \{ t \in \bbr_+ : \| M_t \| \geq n \}, \quad n \in \bbn.
\end{align*}
Let $n \in \bbn$ be arbitrary. Then the stopped process $M^{\tau_n}$ is bounded. Hence, for each $x' \in E'$ the process $x'(M^{\tau_n})$ is a real-valued continuous, bounded local martingale, and hence a martingale. Therefore, by Lemma \ref{lemma-MT-Banach} we deduce that $M^{\tau_n}$ is a martingale. Consequently, the process $M$ is a local martingale.
\end{proof}

\begin{lemma}\label{lemma-bounded-loc-MT}
Let $M$ be an $E$-valued continuous, bounded local martingale. Then $M$ is a martingale.
\end{lemma}

\begin{proof}
For each $x' \in E'$ the process $x'(M)$ is a real-valued continuous, bounded local martingale, and hence a martingale. By Lemma \ref{lemma-MT-Banach} we deduce that $M$ is a martingale.
\end{proof}

\begin{proposition}\label{prop-Banach-square-MT}
Let $M$ be an $E$-valued continuous local martingale such that $M_0$ is bounded and
\begin{align}\label{E-sup}
\bbe \bigg[ \sup_{t \in [0,T]} \| M_t \|^2 \bigg] < \infty \quad \text{for each $T \in \bbr_+$.}
\end{align}
Then $M$ is a square-integrable martingale.
\end{proposition}

\begin{proof}
By \eqref{E-sup} the process $M$ is square-integrable. Thus, it suffices to show that $(M_t)_{t \in [0,T]}$ is a martingale for each $T \in \bbr_+$. Let $T \in \bbr_+$ be arbitrary and define the localizing sequence $(\tau_n)_{n \in \bbn}$ of stopping times as
\begin{align*}
\tau_n := \inf \{ t \in [0,T] : \| M_t \| \geq n \}, \quad n \in \bbn.
\end{align*}
Then $(M_t^{\tau_n})_{t \in [0,T]}$ is a continuous, bounded local martingale, and by Lemma \ref{lemma-bounded-loc-MT} we deduce that $(M_t^{\tau_n})_{t \in [0,T]}$ is a martingale. Let $\tau \leq T$ be an arbitrary stopping time. Using Lemma \ref{lemma-Doob-Banach} we obtain
\begin{align*}
\bbe[M_{\tau \wedge \tau_n}] = \bbe[M_{\tau}^{\tau_n}] = \bbe[M_0^{\tau_n}] = \bbe[M_0].
\end{align*}
Note that by \eqref{E-sup} we have $\bbe[\| M_{\tau} \|] < \infty$. Furthermore, taking into account the continuity of the sample paths of $M$, by \eqref{E-sup} and Lebesgue's dominated convergence theorem we have
\begin{align*}
\lim_{n \to \infty} \bbe \big[ \| M_{\tau \wedge \tau_n} - M_{\tau} \| \big] = 0.
\end{align*}
Since the expectation operator $\bbe : L^1(\Omega,\calf,\bbp) \to E$ given by the Bochner integral is a continuous linear operator, we obtain
\begin{align*}
\bbe[M_{\tau}] = \lim_{n \to \infty} \bbe[M_{\tau \wedge \tau_n}] = \bbe[M_0].
\end{align*}
By Lemma \ref{lemma-Doob-Banach}, this proves that $(M_t)_{t \in [0,T]}$ is a martingale.
\end{proof}

\section{Smooth functions in Banach spaces}\label{app-functions-Banach}

In this appendix we provide the required results about smooth functions in Banach spaces. In what follows, let $E,F,G,H$ be Banach spaces, and let $k \in \bbn_0 \cup \{ \infty \}$ be a nonnegative integer, which may be infinite. For a function $f : U \to F$ defined on a subset $U \subset E$ we agree on the notation $D^0 f := f$. Furthermore, we agree that in case $k = \infty$ the notation $j=0,\ldots,k$ means $j \in \bbn_0$.

\begin{definition}
Let $U \subset E$ be an open subset, and let $f : U \to F$ be a function of class $C^k$.
\begin{enumerate}
\item We say that $f$ is of class $C_b^k$ if the functions $D^j f$, $j=0,\ldots,k$ are bounded.

\item We say that $f$ is of class $C_{b,\loc}^k$ if $f|_V : V \to F$ is of class $C_b^k$ for every open and bounded subset $V \subset U$.
\end{enumerate}
\end{definition}

\begin{proposition}\cite[Prop. 2.4.8]{Abraham}\label{prop-Cb1-Lipschitz}
Suppose that $U \subset E$ is open and convex. Then every mapping $f : U \to F$ of class $C_b^1$ is Lipschitz continuous.
\end{proposition}

\begin{proposition}[Leibniz or product rule]\label{prop-Leibniz}
Let $U \subset E$ be an open subset, let $f : U \to F$, $g : U \to G$ be differentiable mappings, and let $B \in L^{(2)}(F \times G,H)$ be a continuous bilinear operator. We define the new mapping
\begin{align*}
B(f,g) : U \to H, \quad x \mapsto B(f(x),g(x)).
\end{align*}
Then the following statements are true:
\begin{enumerate}
\item[(a)] The mapping $B(f,g)$ is also differentiable, and we have
\begin{align}\label{Leibniz-rule}
D(B(f,g))(x)v = B(Df(x)v,g(x)) + B(f(x),Dg(x)v)
\end{align}
for all $x \in U$ and $v \in E$.

\item[(b)] If $f$ and $g$ are of class $C^k$, then $B(f,g)$ is also of class $C^k$.

\item[(c)] If $f$ and $g$ are of class $C_b^k$, then $B(f,g)$ is also of class $C_b^k$.
\end{enumerate}
\end{proposition}

\begin{proof}
Parts (a) and (b) are immediate consequences of \cite[Thm. 2.4.4]{Abraham}, and part (c) is a consequence of the multidimensional Leibniz rule, which follows inductively from \eqref{Leibniz-rule}.
\end{proof}

\begin{proposition}[Chain rule]\label{prop-smooth-chain}
Let $U \subset E$ be an open subset, and let $f : U \to F$, $g : F \to G$ be differentiable mappings. Then the following statements are true:
\begin{enumerate}
\item[(a)] The composition $g \circ f : E \to G$ is also differentiable, and we have
\begin{align*}
D (g \circ f)(x) = Dg(f(x)) \circ Df(x), \quad x \in U.
\end{align*}
\item[(b)] If $f$ and $g$ are of class $C^k$, then $g \circ f$ is also of class $C^k$.

\item[(c)] If $f$ and $g$ are of class $C_b^k$, then $g \circ f$ is also of class $C_b^k$.

\item[(d)] If $f$ and $g$ are of class $C_{b,\loc}^k$, then $g \circ f$ is of class $C_{b,\loc}^k$.
\end{enumerate}
\end{proposition}

\begin{proof}
Parts (a) and (b) follow from \cite[Thm. 2.4.3]{Abraham}, and parts (c) and (d) are a consequence of the higher order chain rule; see \cite[p. 88]{Abraham}.
\end{proof}

\begin{proposition}\label{prop-smooth-linear}
Let $U \subset E$ be an open subset, let $f : U \to F$ be differentiable, and let $T \in L(F,G)$ be a continuous linear operator. Then the following statements are true:
\begin{enumerate}
\item[(a)] The composition $T \circ f : U \to G$ is also differentiable, and we have
\begin{align*}
D(T \circ f)(x)v = T ( D f(x)v ) \quad \text{for all $x \in U$ and $v \in E$.}
\end{align*}

\item[(b)] If $f$ is of class $C^k$, then $T \circ f$ is also of class $C^k$, and for all $j=0,\ldots,k$ we have
\begin{align}\label{comp-linear-f}
D^j(T \circ f) = T \circ D^j f.
\end{align}

\item[(c)] If $f$ is of class $C_b^k$, then $T \circ f$ is also of class $C_b^k$.
\end{enumerate}
\end{proposition}

\begin{proof}
Parts (a) and (b) are immediate consequences of the chain rule (Proposition \ref{prop-smooth-chain}), and part (c) follows from the formula \eqref{comp-linear-f}.
\end{proof}

\begin{corollary}\label{cor-self-adjoint}
Let $U \subset E$ be an open subset, let $f : U \to F$ be a mapping of class $C^k$, and let $T \in L(F)$ be a continuous linear operator such that $T f(x) = f(x)$ for all $x \in U$. Then for all $j=0,\ldots,k$ we have
\begin{align*}
T \circ D^j f = D^j f.
\end{align*}
\end{corollary}

\begin{proof}
This is an immediate consequence of Proposition \ref{prop-smooth-linear}.
\end{proof}

The following result is also a consequence of the chain rule (Proposition \ref{prop-smooth-chain}).

\begin{proposition}\label{prop-smooth-linear-two}
Let $S \in L(E,F)$ be a linear operator, and let $f : F \to G$ be a differentiable function. Then the following statements are true:
\begin{enumerate}
\item[(a)] The composition $f \circ S : E \to G$ is also differentiable, and we have
\begin{align*}
D ( f \circ S )(x) = Df(Sx) \circ S, \quad x \in E,
\end{align*}
or, in other words,
\begin{align*}
D ( f \circ S )(x)v = Df(Sx)Sv \quad \text{for all $x,v \in E$.}
\end{align*}
\item[(b)] If $f$ is of class $C^k$, then $f \circ S$ is also of class $C^k$.
\end{enumerate}
\end{proposition}

\begin{proposition}\label{prop-restrict-Cb}
Let $U \subset E$ be an open subset, and let $f : U \to F$ be a mapping of class $C^k$ for some finite number $k \in \bbn_0$. For each $x \in U$ there exists an open neighborhood $U_0 \subset U$ of $x$ such that $f|_{U_0}$ is of class $C_b^k$.
\end{proposition}

\begin{proof}
Recalling the notation $D^0 f := f$, we define the finite number $r > 0$ as
\begin{align*}
r := 1 + \sum_{j=0}^k \| D^j f(x) \|.
\end{align*}
Now, we define the subset $U_0 \subset U$ as
\begin{align*}
U_0 := \bigcap_{j=0}^k \{ y \in U : \| D^j f(y) \| < r \}.
\end{align*}
By the definition of $r$ and the continuity of $\| D^j f \|$ for $j=0,\ldots,k$, the subset $U_0$ is an open neighborhood of $x$, and $f|_{U_0}$ is of class $C_b^k$.
\end{proof}

\begin{lemma}\label{lemma-derivative-sequence}
Let $U \subset E$ be an open subset, and let $f : U \to F$ be a mapping of class $C^1$. Let $x \in U$ and $v \in E$ be arbitrary. Let $(t_n)_{n \in \bbn} \subset \bbr \setminus \{ 0 \}$ and $(v_n)_{n \in \bbn} \subset E$ be sequences such that $t_n \to 0$ and $v_n \to v$. Then we have
\begin{align*}
D f(x) v = \lim_{n \to \infty} \frac{f(x+t_n v_n) - f(x)}{t_n}.
\end{align*}
\end{lemma}

\begin{proof}
By Proposition \ref{prop-restrict-Cb} there exists an open convex neighborhood $U_0 \subset U$ of $x$ such that $f|_{U_0}$ is of class $C_b^1$. By Proposition \ref{prop-Cb1-Lipschitz} the mapping $f|_{U_0}$ is Lipschitz continuous. Since $t_n \to 0$, may assume that $x + t_n v \in U_0$ and $x + t_n v_n \in U_0$ for each $n \in \bbn$. Denoting by $L > 0$ the Lipschitz constant of $f|_{U_0}$, we have
\begin{align*}
\bigg\| \frac{f(x + t_n v) - f(x + t_n v_n)}{t_n} \bigg\| \leq \frac{L |t_n| \, \|v-v_n\|}{|t_n|} = L\|v-v_n\| \to 0 \quad \text{as $n \to \infty$,}
\end{align*}
and it follows that
\begin{align*}
Df(x)v &= \lim_{n \to \infty} \frac{f(x + t_n v) - f(x)}{t_n}
\\ &= \lim_{n \to \infty} \frac{f(x + t_n v_n) - f(x)}{t_n} + \lim_{n \to \infty} \frac{f(x + t_n v) - f(x + t_n v_n)}{t_n}
\\ &= \lim_{n \to \infty} \frac{f(x+t_n v_n) - f(x)}{t_n}.
\end{align*}
This completes the proof.
\end{proof}

\begin{lemma}\label{lemma-bump-r-R}
Let $H$ be a Hilbert space, and let $0 < r < R < \infty$ be arbitrary. There exists a function $\varphi : H \to [0,1]$ of class $C_b^{\infty}$ such that $\varphi(x) = 1$ for all $x \in H$ with $\| x \| \leq r$ and $\varphi(x) = 0$ for all $x \in H$ with $\| x \| \geq R$.
\end{lemma}

\begin{proof}
There exists a function $f : \bbr \to [0,1]$ of class $C^{\infty}$ such that $f(y) = 1$ for all $y \in \bbr$ with $|y|^2 \leq r$ and $f(y) = 0$ for all $y \in \bbr$ with $|y|^2 \geq R$. Since $f$ has compact support, it is even of class $C_b^{\infty}$. Now, we define $g : H \to \bbr_+$ as $g(x) = \| x \|^2$, $x \in H$. Then $g$ is of class $C_{b,\loc}^{\infty}$. Furthermore, we define $\varphi := f \circ g : H \to [0,1]$. By Proposition \ref{prop-smooth-chain} the mapping $\varphi$ is of class $C_{b,\loc}^{\infty}$, and we have $\varphi(x) = 1$ for all $x \in H$ with $\| x \| \leq r$ and $\varphi(x) = 0$ for all $x \in H$ with $\| x \| \geq R$. The latter property ensures that $\varphi$ is even of class $C_b^{\infty}$.
\end{proof}

\begin{proposition}\label{prop-modify-fct-bdd}
Let $H$ be a Hilbert space, let $U \subset H$ be an open subset, and let $f : U \to F$ be a mapping of class $C^k$ for some finite number $k \in \bbn_0$. For each $x \in U$ there exist an open neighborhood $N(x) \subset U$ of $x$ and a mapping $\hat{f} : H \to F$ of class $C_b^k$ such that $\hat{f}|_{N(x)} = f|_{N(x)}$.
\end{proposition}

\begin{proof}
By Proposition \ref{prop-restrict-Cb} there is an open neighborhood $U_0 \subset U$ of $x$ such that $f|_{U_0}$ is of class $C_b^k$. There exists $R > 0$ such that $\{ y \in H : \| x-y \| < R \} \subset U_0$. We set $r := \frac{R}{2}$ and define the subset $N(x) \subset U_0$ as the open ball
\begin{align*}
N(x) := \{ y \in H : \| x-y \| < r \}.
\end{align*}
Then $N(x)$ is an open neighborhood of $x$. Furthermore, by Lemma \ref{lemma-bump-r-R} there exists a function $\varphi : H \to [0,1]$ of class $C_b^{\infty}$ such that $\varphi(x) = 1$ for all $x \in H$ with $\| x \| \leq r$ and $\varphi(x) = 0$ for all $x \in H$ with $\| x \| \geq R$. Using Proposition \ref{prop-Leibniz}, the mapping $\hat{f} := \varphi \cdot f : H \to F$ is of class $C_b^k$, and moreover we have $\hat{f}|_{N(x)} = f|_{N(x)}$.
\end{proof}

We will need the following version of the abstract implicit function theorem; see, e.g. \cite[Thm. I.5.9]{Lang}.

\begin{theorem}[Implicit Function Theorem]\label{thm-implicit}
Let $U \subset E$, $V \subset F$ be open and $f : U \times V \to G$ be of class $C^k$ for some $k \in \bbn$. We assume there are $x_0 \in U$ and $y_0 \in V$ such that $f(x_0,y_0) = 0$, and that the linear operator $D_2 f(x_0,y_0) \in L(F,G)$ is an isomorphism. Then there are an open neighborhood $U_0$ of $x_0$ and a unique map $g : U_0 \to V$ of class $C^k$ such that $g(x_0) = y_0$ and
\begin{align*}
f(x,g(x)) = 0 \quad \forall x \in U_0.
\end{align*}
\end{theorem}

We will also require the following Nagumo type theorem. Let $a : E \to E$ be a Lipschitz continuous mapping. Then existence and uniqueness of solutions for the $E$-valued ODE
\begin{align}\label{ODE}
\left\{
\begin{array}{rcl}
y'(t) & = & a(y(t))
\\ y(0) & = & y
\end{array}
\right.
\end{align}
holds true.

\begin{theorem}\label{thm-Nagumo-Banach}
For a closed subset $\cald \subset E$ the following statements are equivalent:
\begin{enumerate}
\item[(i)] $\cald$ is invariant for the ODE \eqref{ODE}.

\item[(ii)] We have $a(y) \in T_{\cald}^b(y)$ for all $y \in \cald$.
\end{enumerate}
\end{theorem}

\begin{proof}
Since the subset $\cald$ is closed, it is also locally closed in the terminology of \cite{Vrabie}. Moreover, since the mapping $a$ is Lipschitz continuous, it satisfies the linear growth condition, and hence it is positively sublinear in the terminology of \cite{Vrabie}. Consequently, the stated result is a consequence of \cite[Thms. 1.1 and 1.2]{Vrabie} (applied with $f \equiv 0$ and $h(t,y) = a(y)$).
\end{proof}

\section{Linear operators in Hilbert spaces}\label{app-operators-Hilbert}

In this appendix we provide the required results about linear operators in Hilbert spaces. In the sequel the symbols $H$ and $H_i$ for some $i \in \bbn$ denote separable Hilbert spaces. We use the notations $L(H_1,H_2)$, $L_1(H_1,H_2)$, $L_2(H_1,H_2)$ for the spaces of all bounded linear operators, nuclear operators and Hilbert-Schmidt operators from $H_1$ to $H_2$.

\begin{proposition}\cite[Prop. B.0.7]{Liu-Roeckner}
The space $L_2(H_1,H_2)$ is also a separable Hilbert space.
\end{proposition}

Furthermore, the notation $L^+$ indicates the respective subsets of nonnegative definite operators, and $L^{++}$ indicates the respective subsets of positive definite operators. Let us recall the following well-known result.

\begin{lemma}\label{lemma-nuc-HS-inequalities}
The following statements are true:
\begin{enumerate}
\item For $T \in L(H_1,H_2)$, $S \in L_1(H_2,H_3)$ and $R \in L(H_3,H_4)$ we have $RST \in L_1(H_1,H_4)$ and
\begin{align*}
\| RST \|_{L_1(H_1,H_4)} \leq \| R \|_{L(H_3,H_4)} \| S \|_{L_1(H_2,H_3)} \| T \|_{L(H_1,H_2)}.
\end{align*}
\item For $T \in L(H_1,H_2)$, $S \in L_2(H_2,H_3)$ and $R \in L(H_3,H_4)$ we have $RST \in L_2(H_1,H_4)$ and
\begin{align*}
\| RST \|_{L_2(H_1,H_4)} \leq \| R \|_{L(H_3,H_4)} \| S \|_{L_2(H_2,H_3)} \| T \|_{L(H_1,H_2)}.
\end{align*}
\item For $T \in L_2(H_1,H_2)$ and $S \in L_2(H_2,H_3)$ we have $ST \in L_1(H_1,H_3)$ and
\begin{align*}
\| ST \|_{L_1(H_1,H_3)} \leq \| S \|_{L_2(H_2,H_3)} \| T \|_{L_2(H_1,H_2)}.
\end{align*}
\end{enumerate}
\end{lemma}

\begin{proof}
See \cite[Satz VI.5.4]{Werner} and \cite[Rem. B.0.6.iii and Prop. B.0.8]{Liu-Roeckner}.
\end{proof}

\begin{lemma}\label{lemma-adjoint-isometry}
The following statements are true:
\begin{enumerate}
\item The mapping
\begin{align*}
\big( L(H_1,H_2), \| \cdot \|_{L(H_1,H_2)} \big) \to \big( L(H_2,H_1), \| \cdot \|_{L(H_2,H_1)} \big), \quad T \mapsto T^*
\end{align*}
is a linear isometry.

\item The mapping
\begin{align*}
\big( L_1(H_1,H_2), \| \cdot \|_{L_1(H_1,H_2)} \big) \to \big( L_1(H_2,H_1), \| \cdot \|_{L_1(H_2,H_1)} \big), \quad T \mapsto T^*
\end{align*}
is a linear isometry.

\item The mapping
\begin{align*}
\big( L_2(H_1,H_2), \| \cdot \|_{L_2(H_1,H_2)} \big) \to \big( L_2(H_2,H_1), \| \cdot \|_{L_2(H_2,H_1)} \big), \quad T \mapsto T^*
\end{align*}
is a linear isometry.
\end{enumerate}
\end{lemma}

\begin{proof}
The first statement is a consequence of \cite[Satz V.5.2.d]{Werner}. For $x \in H_1$ and $y \in H_2$ the adjoint of the linear operator $S \in L(H_1,H_2)$, $S = \la \cdot, x \ra y$ is given by $S^* = \la \cdot,y \ra x$. Thus, the second statement follows from the definition of the nuclear norm. For the third statement we refer to \cite[Rem. B.0.6.i]{Liu-Roeckner}.
\end{proof}

The following auxiliary result is also well-known; see, e.g. \cite[Thm. 12.10]{Rudin}. Recall that $A^{\circ}$ denotes the polar of a set $A$; see Definition \ref{def-polar}, and also Remark \ref{rem-polar-cone}.

\begin{lemma}\label{lemma-adjoint-ker-ran}
For every $T \in L(H_1,H_2)$ we have $\ker(T^*) = \ran(T)^{\perp} = \ran(T)^{\circ}$.
\end{lemma}

We will often consider the spaces $L(H) := L(H,H)$, $L_1(H) := L_1(H,H)$ and $L_2(H) := L_2(H,H)$. It is well-known that $L_1(H) \hookrightarrow L_2(H) \hookrightarrow L(H)$ with continuous embedding. More precisely, we have the following result; see, e.g. \cite[Satz VI.6.2]{Werner}.

\begin{lemma}
The following statements are true:
\begin{enumerate}
\item We have $L_1(H) \subset L_2(H)$ and $\| T \|_{L_2(H)} \leq \| T \|_{L_1(H)}$ for all $T \in L_1(H)$.

\item We have $L_2(H) \subset L(H)$ and $\| T \|_{L(H)} \leq \| T \|_{L_2(H)}$ for all $T \in L_2(H)$.
\end{enumerate}
\end{lemma}

\begin{lemma}\label{lemma-self-adjoint-transfer}
Let $T \in L(H_1)$ and $A \in L(H_1,H_2)$ be linear operators. We define $S \in L(H_2)$ as $S := A T A^*$. Then the following statements are true:
\begin{enumerate}
\item If $T$ is self-adjoint, then $S$ is self-adjoint as well.

\item If $T$ is nonnegative definite, then $S$ is nonnegative definite as well.
\end{enumerate}
\end{lemma}

\begin{proof}
Sine $S^* = A T^* A$, the first statement is evident. If $T$ is nonnegative definite, then for all $x \in H_2$ we have
\begin{align*}
\la Sx,x \ra = \la A T A^* x,x \ra = \la T A^* x, A^* x \ra \geq 0,
\end{align*}
proving the second statement.
\end{proof}

For $T \in L_1(H)$ we denote by $\tr(T)$ the \emph{trace} of the nuclear operator $T$, which is defined as
\begin{align*}
\tr(T) := \sum_{j=1}^{\infty} \la T e_j, e_j \ra,
\end{align*}
independent of the choice of the orthonormal basis $\{ e_j \}_{j \in \bbn}$ of $H$.

\begin{lemma}\cite[Satz VI.5.8]{Werner}\label{lemma-trace-functional}
The following statements are true:
\begin{enumerate}
\item The trace $L_1(H) \to \bbr$, $T \mapsto \tr(T)$ is a continuous linear functional with $\| \tr \| = 1$.

\item We have $\tr(T) = \tr(T^*)$ for all $T \in L_1(H)$.
\end{enumerate}
\end{lemma}

In view of the upcoming result, let us recall the notation $|T| := (T^* T)^{1/2}$ for a compact operator $T \in L(H_1,H_2)$. Note that $|T| \in L(H_1)$ is always self-adjoint and nonnegative definite. Furthermore, note that for every self-adjoint and nonnegative definite compact operator $T \in L(H)$ we have $T = |T|$.

\begin{lemma}\label{lemma-nuclear-HS}
The following statements are true:
\begin{enumerate}
\item For every $T \in L_1(H)$ we have $|T| \in L_1^+(H)$ and $\| T \|_{L_1(H)} = \tr( |T| )$.

\item We have $| \tr(T) | \leq \| T \|_{L_1(H)}$ for all $T \in L_1(H)$.

\item For every self-adjoint operator $T \in L_1^+(H)$ we have $\| T \|_{L_1(H)} = \tr(T)$.

\item The inner product on $L_2(H)$ is given by $\la T,S \ra_{L_2(H)} = \tr(S^* T)$ for all $T,S \in L_2(H)$.

\item In particular, we have $\| T \|_{L_2(H)}^2 = \| T^* T \|_{L_1(H)}$ for every $T \in L_2(H)$.
\end{enumerate}
\end{lemma}

\begin{proof}
The statements follow from \cite[Satz VI.5.5]{Werner}, \cite[Rem. B.0.4]{Liu-Roeckner} and \cite[Satz VI.6.2.f]{Werner}.
\end{proof}

\begin{lemma}\label{lemma-trace-commute}
Let $T,S \in L(H)$ be such that $T \in L_1(H)$ or $T,S \in L_2(H)$. Then we have $ST,TS \in L_1(H)$ and $\tr(ST) = \tr(TS)$.
\end{lemma}

\begin{proof}
In case $T \in L_1(H)$ the result follows from \cite[Satz VI.5.8.c]{Werner}. If $T,S \in L_2(H)$, then using Lemmas \ref{lemma-adjoint-isometry}, \ref{lemma-nuclear-HS} and \ref{lemma-trace-functional} we obtain
\begin{align*}
\tr(ST) = \la T,S^* \ra_{L_2(H)} = \la T^*,S \ra_{L_2(H)} = \tr(S^* T^*) = \tr((TS)^*) = \tr(TS),
\end{align*}
completing the proof.
\end{proof}

The following result is a consequence of \cite[Prop. B.0.10]{Liu-Roeckner}.

\begin{lemma}\label{lemma-trace-commute-LR}
Let $T \in L_2(H)$ and $S \in L(H)$ be arbitrary. Then we have
\begin{align*}
S T T^*, T^* S T \in L_1(H) \quad \text{and} \quad \tr(S T T^*) = \tr(T^* ST).
\end{align*}
\end{lemma}

In what follows, we denote by $S_1^+(H) \subset L_1(H)$ the subset of all nuclear operators which are self-adjoint and nonnegative definite. The following result shows that the mapping
\begin{align*}
\big( L_1(H), \| \cdot \|_{L_1(H)} \big) \to \big( S_1^+(H), \| \cdot \|_{L_1(H)} \big), \quad T \mapsto |T|
\end{align*}
is a nonlinear continuous isometry.

\begin{lemma}\label{lemma-abs-operator-continuous}
The following statements are true:
\begin{enumerate}
\item For every $T \in L_1(H)$ we have $|T| \in S_1^+(H)$ and $\| \, |T| \, \|_{L_1(H)} = \| T \|_{L_1(H)}$.

\item For all $T,S \in L_1(H)$ we have
\begin{align}\label{T-abs-value}
\| \, |T| - |S| \, \|_{L_1(H)} \leq \big( 2 \| T+S \|_{L_1(H)} \| T-S \|_{L_1(H)} \big)^{1/2}.
\end{align}
\end{enumerate}
\end{lemma}

\begin{proof}
The first statement is an immediate consequence of Lemma \ref{lemma-nuclear-HS}. Since $L_1(H)$ is the predual of the $W^*$-algebra $L(H)$, the second statement is a consequence of the Theorem on page 123 in \cite{Kosaki}.
\end{proof}

\begin{remark}
In the finite dimensional situation $\dim H < \infty$, the inequality \eqref{T-abs-value} is also a consequence of \cite[Thm. X.2.1]{Bhatia}.
\end{remark}

Moreover, we denote by $S_2^+(H) \subset L_2(H)$ the subset of all Hilbert-Schmidt operators which are self-adjoint and nonnegative definite. The following result shows that the mapping
\begin{align*}
\big( S_1^+(H), \| \cdot \|_{L_1(H)} \big) \to \big( S_2^+(H), \| \cdot \|_{L_2(H)} \big), \quad T \mapsto T^{1/2}
\end{align*}
is a nonlinear continuous mapping such that
\begin{align*}
\| T^{1/2} \|_{L_2(H)} = \sqrt { \| T \|_{L_1(H)} } \quad \text{for all $T \in S_1^+(H)$.}
\end{align*}

\begin{lemma}\label{lemma-square-root-continuous}
The following statements are true:
\begin{enumerate}
\item For every $T \in S_1^+(H)$ we have $T^{1/2} \in S_2^+(H)$ and
\begin{align*}
\| T^{1/2} \|_{L_2(H)}^2 = \| T \|_{L_1(H)}.
\end{align*}

\item For all $T,S \in S_1^+(H)$ we have
\begin{align*}
\| T^{1/2} - S^{1/2} \|_{L_2(H)}^2 \leq \| T-S \|_{L_1(H)}.
\end{align*}
\end{enumerate}
\end{lemma}

\begin{proof}
For $T \in S_1^+(H)$ consider the spectral decomposition
\begin{align*}
T = \sum_{j=1}^{\infty} \lambda_j \la \cdot, e_j \ra \, e_j
\end{align*}
with eigenvalues $\lambda_j \geq 0$, $j \in \bbn$ and an orthonormal basis $\{ e_j \}_{j \in \bbn}$ of $H$. Then we have
\begin{align*}
\| T^{1/2} \|_{L_2(H)}^2 &= \sum_{j=1}^{\infty} \| T^{1/2} e_j \|^2 = \sum_{j=1}^{\infty} \la T^{1/2} e_j, T^{1/2} e_j \ra = \sum_{j=1}^{\infty} \la T e_j, e_j \ra
\\ &= \tr(T) = \| T \|_{L_1(H)},
\end{align*}
proving the first statement. For the second statement we refer to \cite[Lemma 4.1]{Powers-Stormer}.
\end{proof}

As an immediate consequence of Lemmas \ref{lemma-abs-operator-continuous} and \ref{lemma-square-root-continuous} we obtain the following result, showing that the mapping
\begin{align*}
\big( L_1(H), \| \cdot \|_{L_1(H)} \big) \to \big( S_2^+(H), \| \cdot \|_{L_2(H)} \big), \quad T \mapsto |T|^{1/2}
\end{align*}
is a nonlinear continuous mapping such that
\begin{align*}
\| |T|^{1/2} \|_{L_2(H)} = \sqrt { \| T \|_{L_1(H)} } \quad \text{for all $T \in L_1(H)$.}
\end{align*}

\begin{corollary}\label{cor-square-abs-continuous}
The following statements are true:
\begin{enumerate}
\item For every $T \in L_1(H)$ we have $|T|^{1/2} \in S_2^+(H)$ and
\begin{align*}
\| \, |T|^{1/2} \, \|_{L_2(H)}^2 = \| T \|_{L_1(H)}.
\end{align*}
\item The mapping $L_1(H) \to S_2^+(H)$, $T \mapsto |T|^{1/2}$ is continuous.
\end{enumerate}
\end{corollary}

For the following result we consider self-adjoint compact operators $T \in L(H)$ with spectral decomposition
\begin{align*}
T = \sum_{j=1}^{\infty} \lambda_j(T) \la \cdot,e_j \ra \, e_j.
\end{align*}
Concerning the ordering of the eigenvalues $(\lambda_j(T))_{j \in \bbn}$ we assume that $|\lambda_j(T)| \geq |\lambda_{j+1}(T)|$ for all $j \in \bbn$, and that for all $j < k$ with $|\lambda_j(T)| = |\lambda_k(T)|$ and $\lambda_j(T) \neq \lambda_k(T)$ we have $\lambda_j(T) > 0$ and $\lambda_k(T) < 0$.

\begin{proposition}\label{prop-eigenvalue-map}
For all self-adjoint nuclear operators $T,S \in L_1(H)$ we have
\begin{align*}
\sum_{j=1}^{\infty} | \lambda_j(T) - \lambda_j(S) | \leq \| T-S \|_{L_1(H)}.
\end{align*}
\end{proposition}

\begin{proof}
In the finite dimensional situation, this is a consequence of \cite[Lemma IV.3.2]{Bhatia}, and in the general situation, the inequality follows from approximating the operators $T$ and $S$ according to their spectral decompositions.
\end{proof}

\begin{lemma}\label{lemma-kernels}
For every $T \in L(H_1,H_2)$ we have $\ker(T) = \ker(T^* T)$, and thus $\ker(T^*) = \ker(T T^*)$.
\end{lemma}

\begin{proof}
For each $x \in \ker(T)$ we have $T^* T x = 0$, showing that $\ker(T) \subset \ker(T^* T)$. Furthermore, for each $x \in \ker(T^* T)$ we have
\begin{align*}
\| Tx \|^2 = \la Tx,Tx \ra = \la T^* T x,x \ra = 0,
\end{align*}
and hence $Tx = 0$, showing that $\ker(T^* T) \subset \ker(T)$.
\end{proof}

Now we will provide the required results about the Moore-Penrose pseudoinverse. This inverse will be introduced for linear operators $T \in L(H_1,H_2)$ with closed range. Before introducing the Moore-Penrose pseudoinverse, we prepare some auxiliary results about the closed range property.

\begin{lemma}\label{lemma-closed-range-adjoint}
Let $T \in L(H_1,H_2)$ be a linear operator. Then $T$ has closed range if and only if $T^*$ has closed range.
\end{lemma}

\begin{proof}
This is a consequence of the closed range theorem; see, e.g. \cite[Thm. 4.14]{Rudin}.
\end{proof}

\begin{lemma}\label{lemma-closed-range-square-root-2}
Let $T \in L(H_1,H_2)$ be a closed range operator. Then we have $\ran(T) = \ran(T T^*)$. In particular $T T^*$ also has closed range.
\end{lemma}

\begin{proof}
It is clear that $\ran(T T^*) \subset \ran(T)$. Now, let $y \in \ran(T)$ be arbitrary. Then there exists $x \in H_1$ such that $Tx = y$. By Lemma \ref{lemma-closed-range-adjoint} the adjoint operator $T^*$ also has closed range, and hence we can consider the decomposition $x = x_1 + x_2$ according to $H_1 = \ran(T^*) \oplus \ran(T^*)^{\perp}$. Noting that $\ran(T^*)^{\perp} = \ker(T)$ due to Lemma \ref{lemma-adjoint-ker-ran}, we obtain $T x_1 = Tx = y$, showing that $y \in \ran(T T^*)$.
\end{proof}

\begin{lemma}\label{lemma-closed-range-square-root}
Let $T \in L(H_1,H_2)$ be such that $T T^*$ has closed range. Then we have $\ran(T) = \ran(T T^*)$. In particular $T$ also has closed range.
\end{lemma}

\begin{proof}
Using Lemma \ref{lemma-adjoint-ker-ran} and Lemma \ref{lemma-kernels} we have
\begin{align*}
\ran(T T^*) \subset \ran(T) \subset \overline{\ran(T)} = \ker(T^*)^{\perp} = \ker(T T^*)^{\perp} = \overline{\ran(T T^*)} = \ran(T T^*),
\end{align*}
proving that $\ran(T) = \ran(T T^*)$.
\end{proof}

As an immediate consequence of Lemmas \ref{lemma-closed-range-adjoint}--\ref{lemma-closed-range-square-root} we obtain the following result about closed range operators:

\begin{proposition}\label{prop-closed-range-char}
For a linear operator $T \in L(H_1,H_2)$ the following statements are equivalent:
\begin{enumerate}
\item[(i)] $T$ has closed range.

\item[(ii)] $T^*$ has closed range.

\item[(iii)] $T T^*$ has closed range.

\item[(iv)] $T^* T$ has closed range.
\end{enumerate}
If these equivalent conditions are fulfilled, then we have
\begin{align*}
\ran(T) = \ran(T T^*) \quad \text{and} \quad \ran(T^*) = \ran(T^* T).
\end{align*}
\end{proposition}

For compact operators, the closed range property can be characterized as follows.

\begin{theorem}\label{thm-closed-range-fd}\cite[Thm. 4.18.b]{Rudin}
For a compact operator $T \in L(H_1,H_2)$ the following statements are equivalent:
\begin{enumerate}
\item[(i)] $T$ has closed range.

\item[(ii)] $\ran(T)$ is finite dimensional.
\end{enumerate}
\end{theorem}

The proofs of the following two results (Theorems \ref{thm-inverse-existence} and \ref{thm-MP-characterizations}) are a consequence of the results from \cite[Chap. II]{Groetsch}.

\begin{theorem}\label{thm-inverse-existence}
For each operator $T \in L(H_1,H_2)$ with closed range there exists a unique operator $T^+ \in L(H_2,H_1)$ such that
\begin{enumerate}
\item $T T^+ = (T T^+)^*$.

\item $T^+ T = (T^+ T)^*$.

\item $T T^+ T = T$.

\item $T^+ T T^+ = T^+$.
\end{enumerate}
\end{theorem}

\begin{definition}
For each $T \in L(H_1,H_2)$ with closed range we call $T^+$ the \emph{Moore-Penrose pseudoinverse} of $T$.
\end{definition}

\begin{theorem}\label{thm-MP-characterizations}
Let $T \in L(H_1,H_2)$ be an operator with closed range, and let $T^+ \in L(H_2,H_1)$ be another linear operator. Then the following statements are equivalent:
\begin{enumerate}
\item[(i)] $T^+$ the Moore-Penrose pseudoinverse of $T$.

\item[(ii)] We have $T^+ T x = x$ for all $x \in \ker(T)^{\perp}$ and $T^+ y = 0$ for all $y \in \ran(T)^{\perp}$.

\item[(iii)] We have $T T^+ = P_{\ran(T)}$ and $T^+ T = P_{\ran(T^+)}$.
\end{enumerate}
\end{theorem}

The following result in particular shows that for a closed range operator $T$ the Moore-Penrose pseudoinverse $T^+$ has closed range as well.

\begin{theorem}\cite[Thms. 2.1.2 and 2.1.5]{Groetsch}\label{thm-range-of-inverse}
Let $T \in L(H_1,H_2)$ be a closed range operator. Then the following statements are true:
\begin{enumerate}
\item We have $\ran(T^+) = \ran(T^*) = \ran(T^+ T)$. In particular, the operators $T^+$ and $T^+ T$ have closed range.

\item We have $T^+ = (T^* T)^+ T^* = T^* (T T^*)^+$.
\end{enumerate}
\end{theorem}

\begin{lemma}\label{lemma-inverse-adjoint}
Let $T \in L(H_1,H_2)$ be a closed range operator. Then $T^*$ is also a closed range operator, and we have $(T^*)^+ = (T^+)^*$.
\end{lemma}

\begin{proof}
By Proposition \ref{prop-closed-range-char} the adjoint operator $T^*$ also has closed range. Setting $S := (T^+)^*$, we have $S^* = T^+$, and by Theorem \ref{thm-inverse-existence} we obtain
\begin{align*}
T^* S &= (T^+ T)^* = T^+ T = ( T^* S )^*,
\\ S T^* &= (T T^+)^* = T T^+ = (S T^*)^*,
\\ T^* S T^* &= (T T^+ T)^* = T^*,
\\ S T^* S &= (T^+ T T^+)^* = S.
\end{align*}
In view of Theorem \ref{thm-inverse-existence}, this completes the proof.
\end{proof}

\begin{lemma}\label{lemma-inverse-rules}
Let $T \in L(H)$ be a self-adjoint operator with closed range. Then the following statements are true:
\begin{enumerate}
\item $T^+ \in L(H)$ is also self-adjoint; that is $(T^+)^* = T^+$.

\item We have $T T^+ = T^+ T$.

\item More generally, we have $T^n T^+ = T^+ T^n$ for each $n \in \bbn$.

\item $T^2$ also has closed range, and we have $(T^2)^+ = (T^+)^2$.

\item We have $T^2 = T^2 (T^+)^2 T^2$.
\end{enumerate}
\end{lemma}

\begin{proof}
The first statement is an immediate consequence of Lemma \ref{lemma-inverse-adjoint}. Furthermore, by Theorem \ref{thm-inverse-existence} we have
\begin{align*}
T T^+ = (T T^+)^* = (T^+)^* T^* = T^+ T.
\end{align*}
In order to prove the more general identity, we proceed by induction and suppose that $T^n T^+ = T^+ T^n$. By induction hypothesis we obtain
\begin{align*}
T^{n+1} T^+ = T T^n T^+ = T T^+ T^n = T^+ T T^n = T^+ T^{n+1}.
\end{align*}
Moreover, by Proposition \ref{prop-closed-range-char} the operator $T^2$ also has closed range, and by Theorem \ref{thm-range-of-inverse} and Theorem \ref{thm-inverse-existence} we have
\begin{align*}
(T^+)^2 = (T^2)^+ T \circ T (T^2)^+ = (T^2)^+ T^2 (T^2)^+ = (T^2)^+.
\end{align*}
Now, using Theorem \ref{thm-inverse-existence} again it follows that
\begin{align*}
T^2 = T^2 (T^2)^+ T^2 = T^2 (T^+)^2 T^2,
\end{align*}
completing the proof.
\end{proof}

For a closed subspace $U \subset H$ we denote by $P_U$ the orthogonal projection on $U$. Furthermore, for a linear operator $T \in L(H)$ we denote by $P_T$ the orthogonal projection on the closure of the range of $T$; that is $P_T := P_U$ with $U := \overline{\ran(T)}$.

\begin{lemma}\label{lemma-orth-proj-range}
Let $T \in L(H)$ be such that $T$, or equivalently $T T^*$, has closed range. Then we have $P_{T T^*} T = T$.
\end{lemma}

\begin{proof}
By Proposition \ref{prop-closed-range-char} the operator $T$ has closed range if and only if $T T^*$ has closed range, and in this case we have $P_T = P_{T T^*}$. Since $P_T x = x$ for all $x \in \ran(T)$, the claimed identity $P_{T T^*} T = T$ follows.
\end{proof}

\begin{proposition}\label{prop-series-app}
Let $U \subset H$ be an open subset, and let $C : U \to L_1(H)$ and $T : U \to L(H)$ be mappings such that $C$ is of class $C^1$. Let $\{ e_j \}_{j \in \bbn}$ be an orthonormal basis of $H$. For each $j \in \bbn$ we denote by $C^j : H \to H$ and $T^j : H \to H$ the mappings given by
\begin{align*}
C^j(x) := C(x)e_j \quad \text{and} \quad T^j(x) := T(x)e_j \quad \text{for each $x \in U$.}
\end{align*}
Then for all $x \in U$ the series
\begin{align}\label{series-general-result}
\sum_{j=1}^{\infty} D C^j (x) T^j(x)
\end{align}
is weakly convergent, and for each $u \in H$ we have
\begin{align}\label{series-trace-general-result}
\sum_{j=1}^{\infty} \la u, D C^j(x) T^j(x) \ra = \tr \big( D C^*(x) T(x) u \big).
\end{align}
\end{proposition}

\begin{proof}
For any mapping $S : U \to L(H)$ and any $y \in H$ we denote by $Sy : U \to H$ the mapping $z \mapsto S(z)y$. Let $x \in U$ be arbitrary. Using Proposition \ref{prop-smooth-linear}, for each $u \in H$ we obtain
\begin{align*}
&\sum_{j=1}^{\infty} \la u, D C^j(x) T^j(x) \ra = \sum_{j=1}^{\infty} \la u, D C(x) T^j(x) e_j \ra = \sum_{j=1}^{\infty} \la (D C(x) T^j(x))^* u, e_j \ra
\\ &= \sum_{j=1}^{\infty} \la D C^*(x) T^j(x) u, e_j \ra = \sum_{j=1}^{\infty} \la D (C^* u)(x) T^j(x), e_j \ra
\\ &= \sum_{j=1}^{\infty} \la D (C^* u)(x) T(x) e_j, e_j \ra = \tr \big( D (C^* u)(x) T(x) \big) = \tr \big( D C^*(x) T(x) u \big),
\end{align*}
showing \eqref{series-trace-general-result}. Moreover, noting that due to Lemma \ref{lemma-trace-functional} the mapping
\begin{align*}
H \to \bbr, \quad u \mapsto \tr ( \Phi u )
\end{align*}
is a continuous linear functional for every $\Phi \in L(H,L_1(H))$, by the Fr\'{e}chet-Riesz theorem it follows that the series \eqref{series-general-result} is weakly convergent.
\end{proof}

Now, we provide the required results about continuous linear functionals on Hilbert spaces. By the Riesz representation theorem we have the following well-known result.

\begin{lemma}\label{lemma-Riesz-diff-1}
The mapping $H \to L(H,\bbr)$ given by $x \mapsto \la x,\cdot \ra$ is an isometric isomorphism.
\end{lemma}

Let $S(H) \subset L(H)$ be the subspace of all self-adjoint operators $T \in L(H)$, and let $L_s^{(2)}(H,\bbr) \subset L^{(2)}(H,\bbr)$ be the subspace of all symmetric bilinear operators $B \in L^{(2)}(H,\bbr)$. By the Lax-Milgram theorem and the open mapping theorem we have the following result.

\begin{lemma}\label{lemma-Lax-Milgram}
The mapping $\Psi : L(H) \to L^{(2)}(H,\bbr)$ given by
\begin{align*}
(\Psi T)(x,y) = \la Tx,y \ra, \quad T \in L(H)
\end{align*}
is a linear isomorphism such that $\Psi(S(H)) = L_s^{(2)}(H,\bbr)$.
\end{lemma}

\begin{remark}\label{remark-derivatives}
Let $U \subset H$ be an open subset, and let $\phi : U \to \bbr$ be a mapping. We fix an arbitrary element $x \in U$.
\begin{enumerate}
\item Suppose that $\phi$ is of class $C^1$. Then we have $D \phi(x) \in L(H,\bbr)$. By Lemma \ref{lemma-Riesz-diff-1} we may regard the first order derivative as an element from $H$, which justifies the notation
\begin{align*}
D \phi(x)v = \la D \phi(x),v \ra, \quad  v \in H.
\end{align*}
\item Suppose that $\phi$ is even of class $C^2$. Then we have $D^2 \phi(x) \in L_s^{(2)}(H,\bbr)$; see, e.g. \cite[Prop. 2.4.14]{Abraham}. By Lemma \ref{lemma-Lax-Milgram} we may regard the second order derivative as an operator from $S(H)$, which justifies the notation
\begin{align*}
D^2 \phi(x)(v,w) = \la D^2 \phi(x)v,w \ra, \quad v,w \in H.
\end{align*}
\end{enumerate}
\end{remark}

Combining Lemma \ref{lemma-Riesz-diff-1} and Parseval's identity we obtain the following result.

\begin{lemma}\label{lemma-Riesz-Parseval}
We have $H \cong L(H,\bbr) = L_2(H,\bbr)$, and
\begin{align*}
\| x \| = \| \la x,\cdot \ra \|_{L(H,\bbr)} = \| \la x,\cdot \ra \|_{L_2(H,\bbr)} \quad \text{for all $x \in H$.}
\end{align*}
\end{lemma}

As an immediate consequence we obtain:

\begin{proposition}\label{prop-Riesz-Parseval-1}
We have $H \cong L_2(H,\bbr) \cong \ell^2(\bbn)$, and the mapping
\begin{align*}
H \mapsto \ell^2(\bbn), \quad x \mapsto \big( \la x,e_j \ra \big)_{j \in \bbn}
\end{align*}
is an isometric isomorphism, where $\{ e_j \}_{j \in \bbn}$ denotes any orthonormal basis of $H$.
\end{proposition}

Furthermore, the following result holds true.

\begin{proposition}\label{prop-Riesz-Parseval-2}
The following statements are true:
\begin{enumerate}
\item We have $L_2(H) \cong L_2(H,L_2(H,\bbr)) \cong \ell^2(\bbn \times \bbn)$, and the mapping
\begin{align}\label{double-sequence}
L_2(H) \to \ell^2(\bbn \times \bbn), \quad T \mapsto \big( \la T e_i,e_j \ra \big)_{i,j \in \bbn}
\end{align}
is an isometric isomorphism, where $\{ e_j \}_{j \in \bbn}$ denotes any orthonormal basis of $H$.

\item An operator $T \in L_2(H)$ is self-adjoint if and only if the associated double sequence in \eqref{double-sequence} is symmetric.
\end{enumerate}
\end{proposition}

\begin{proof}
The linear operator $\Phi : L_2(H) \to L_2(H,L_2(H,\bbr))$ given by
\begin{align*}
(\Phi T)x = \la Tx,\cdot \ra, \quad T \in L_2(H)
\end{align*}
is an isometric isomorphism. Indeed, let $T \in L_2(H)$ be arbitrary. By Lemma \ref{lemma-Riesz-Parseval} we have
\begin{align*}
\| \Phi T \|_{L_2(H,L_2(H,\bbr))}^2 &= \sum_{i=1}^{\infty} \| (\Phi T)e_i \|_{L_2(H,\bbr)}^2 = \sum_{i=1}^{\infty} \| \la T e_i, \cdot \ra \|_{L_2(H,\bbr)}^2
\\ &= \sum_{i=1}^{\infty} \| T e_i \|^2 = \| T \|_{L_2(H)}^2.
\end{align*}
Moreover, we have
\begin{align*}
\sum_{i=1}^{\infty} \| \la T e_i, \cdot \ra \|_{L_2(H,\bbr)}^2 = \sum_{i=1}^{\infty} \sum_{j=1}^{\infty} |\la T e_i, e_j \ra|^2,
\end{align*}
proving the first statement. For the second statement, it suffices to show that every linear operator $T \in L(H)$ such that $\la T e_i,e_j \ra = \la T e_j, e_i \ra$ for all $i,j \in \bbn$ is self-adjoint. Indeed, for all $x,y \in H$ we have
\begin{align*}
\la Tx,y \ra &= \bigg\la T \bigg( \sum_{i=1}^{\infty} \la x, e_i \ra e_i \bigg), \sum_{j=1}^{\infty} \la y, e_j \ra e_j \bigg\ra
\\ &= \sum_{i=1}^{\infty} \sum_{j=1}^{\infty} \la x, e_i \ra \la y, e_j \ra \la T e_i,e_j \ra = \sum_{i=1}^{\infty} \sum_{j=1}^{\infty} \la x, e_i \ra \la y, e_j \ra \la T e_j,e_i \ra
\\ &= \bigg\la T \bigg( \sum_{j=1}^{\infty} \la y, e_j \ra e_j \bigg), \sum_{i=1}^{\infty} \la x, e_i \ra e_i \bigg\ra = \la Ty,x \ra,
\end{align*}
completing the proof.
\end{proof}

\end{appendix}

\end{document}